\documentclass[onefignum,onetabnum]{siamart190516}

\usepackage{lipsum}
\usepackage{amsfonts}
\usepackage{graphicx}
\usepackage{epstopdf}
\usepackage{algorithm}
\usepackage{algorithmic}
\usepackage{amsopn}
\usepackage{amssymb}
\usepackage{booktabs}
\usepackage{tikz}
\usepackage{pgfplots}
	\pgfplotsset{width=10cm,compat=1.9}
\usepackage{csvsimple}
\usepackage{siunitx}
\usepackage{array}
\ifpdf
  \DeclareGraphicsExtensions{.eps,.pdf,.png,.jpg}
\else
  \DeclareGraphicsExtensions{.eps}
\fi
\usepackage{url}
\usepackage{chngcntr}
\usepackage{longtable}
\usepackage{arydshln}

\newsiamremark{remark}{Remark}
\newsiamremark{criteria}{Criteria}
\newsiamthm{assumption}{Assumption}
\newsiamremark{example}{Example}
\crefname{example}{Example}{Examples}

\headers{Randomized Block Adaptive Solvers}{Patel, Jahangoshahi, \& Maldonado}

\title{Randomized Block Adaptive Linear System Solvers\thanks{Authors are funded by UW-Madison WARF Award AAD5914, and DOE Contract DE- AC02-06CH11347.}}

\author{Vivak Patel\thanks{University of Wisconsin, Madison, WI 
  (\email{vivak.patel@wisc.edu}).}
\and Mohammad Jahangoshahi\thanks{Susquehanna International Group, Bala Cynwyd, PA 
  (\email{mjahangoshahi@uchicago.edu}).}
\and Daniel Adrian Maldonado\thanks{Argonne National Laboratories, Lemont, IL (\email{maldonadod@anl.gov}).}}

\usepackage{nomencl}
\usepackage{etoolbox}
\usepackage{xstring}
\usepackage{xpatch}

\patchcmd{\thenomenclature}{\section*}{\section}{}{}

\patchcmd{\thenomenclature}
  {\leftmargin\labelwidth}
  {\leftmargin\labelwidth\itemindent 0em }
  {}{}

\newcommand{\nomenclheader}[1]{%
  \item[\hspace*{-1\itemindent}\normalfont\itshape#1]}
\renewcommand\nomgroup[1]{%
  \bigskip
  \IfStrEqCase{#1}{%
   {A}{\nomenclheader{Linear Algebra}}%
   {B}{\nomenclheader{Probability}}%
   {C}{\nomenclheader{Iterative Methods}}%
  }%
}

\DeclareMathOperator{\row}{row}
\DeclareMathOperator{\col}{col}

\DeclareMathOperator*{\argmax}{argmax}
\DeclareMathOperator*{\argmin}{argmin}

\newcommand{\innorm}[1]{\Vert #1 \Vert}
\newcommand{\norm}[1]{\left\Vert #1 \right\Vert}
\newcommand{\1}[1]{\textbf{1}\left[ #1 \right]}
\newcommand{\Prb}[1]{\mathbb{P}\left[ #1 \right]}
\newcommand{\inPrb}[1]{\mathbb{P}[ #1 ]}
\newcommand{\E}[1]{\mathbb{E}\left[ #1 \right]}

\newcommand{\cond}[2]{\mathbb{E}\left[\left. #1 \right\vert #2 \right]}
\newcommand{\incond}[2]{\mathbb{E}[ #1 \vert #2 ]}
\newcommand{\condPrb}[2]{\mathbb{P}\left[\left. #1 \right\vert #2 \right]}
\newcommand{\incondPrb}[2]{\mathbb{P}[ #1 \vert #2 ]}

\newcommand{\inlinspan}[1]{\mathrm{span}[ #1 ]}
\newcommand{\linspan}[1]{\mathrm{span}\left[ #1 \right]}
\newcommand{\rnk}[1]{\mathrm{rank}\left( #1 \right)}

\newcommand{\T}{^\intercal}
\newcommand{\Prj}{\mathcal{P}}

\ifpdf
\hypersetup{
  pdftitle={Randomized Block Adaptive Solvers},
  pdfauthor={Patel, Jahangoshahi, \& Maldonado}
}
\fi

\makenomenclature

\begin{document}

\maketitle

\begin{abstract}
Randomized linear solvers randomly compress and solve a linear system with compelling theoretical convergence rates and computational complexities. 
However, such solvers suffer a substantial disconnect between their theoretical rates and actual efficiency in practice. Fortunately, these solvers are quite flexible and can be adapted to specific problems and computing environments to ensure high efficiency in practice, even at the cost of lower effectiveness (i.e., having a slower theoretical rate of convergence). While highly efficient adapted solvers can be readily designed by application experts, will such solvers still converge and at what rate? To answer this, we distill three general criteria for randomized adaptive solvers, which, as we show, will guarantee a worst-case exponential rate of convergence of the solver applied to consistent and inconsistent linear systems irrespective of whether such systems are over-determined, under-determined or rank-deficient. As a result, we enable application experts to design randomized adaptive solvers that achieve efficiency and can be verified for effectiveness using our theory. We demonstrate our theory on twenty-six solvers, nine of which are novel or novel block extensions of existing methods to the best of our knowledge.
\end{abstract}



\section{Introduction}
Solving linear systems and least squares problems remain critical operations in scientific and engineering applications. As the size of systems or the sheer number of systems that need to be solved grow, faster and approximate linear solvers have become essential to scalability. Recently, randomized linear solvers have become of interest as they can compress the information in the original system in a problem-blind fashion, which can then be used to inexpensively and approximately solve the original system \cite{woodruff2014}. Moreover, by iterating on this procedure, randomized linear solvers will converge exponentially fast to the solution of the original system \cite{pilanci2016}. In fact, a rather simple randomized linear system solver was recently shown to achieve a \textit{universal} exponential rate of convergence for any consistent linear system with high probability \cite{steinerberger2022}. 

Despite such an incredible result, as we show through a salient example (see \Cref{section-ce}), randomized linear solvers suffer a substantial disconnect between their convergence rate theory and actual efficiency in practice because they often violate simple computing principles (e.g., the locality principle \cite{denning2005}). Briefly, in the example in \Cref{section-ce}, an ``oracle'' linear solver inspired by \cite{steinerberger2022} is applied to a specific $10^7 \times 100$ system such that it only requires $100$ arithmetic operations to find a solution with absolute error of $10^{-16}$, yet is \textit{slower} than block Kaczmarz---which, in theory, requires over $10^{10}$ arithmetic operations to find such a solution---because of access patterns that violate data locality.
Unfortunately, nearly all variations of such linear solvers that exist  \cite{agmon1954,motzkin1954,strohmer2009,bai2013,zouzias2013,gower2015,needell2014,needell2015,nutini2016,bai2018,haddock2019,necoara2019,gower2021,
rebrova2021,steinerberger2021} can be shown to suffer from this disconnect between their theoretical convergence rates and actual efficiency by specific choices of the linear system, software environment or hardware.

A pessimistic view of these solvers would imply that they should be wholly abandoned. An alternative perspective would suggest a better prognosis: because of the adaptability of such solvers, they can be highly tailored to specific linear systems, software environments, and hardware to achieve high efficiency even at the expense of worse theoretical convergence rates. This latter view is the one adopted in this work.

A bevy of adapted methods can be designed and deployed by atomizing, composing and customizing key components of randomized linear solvers.\footnote{We are implementing a software package to enable this approach. See \url{https://github.com/numlinalg/RLinearAlgebra.jl}.} Owing to the freedom of creating such solvers, understanding whether the efficient highly-adapted method will still converge and at what cost to the rate (e.g., will the solver now converge sub-exponentially?) becomes integral to a practitioner's decision to implement the method.

To address this consideration, a handful of adaptive solvers were shown to retain exponential convergence by \cite{gower2021}, but in a limited context: the set of projections must be finite; and the exactness assumption \cite[Assumption 2]{richtarik2020} must be satisfied, which is generally difficult to verify in practice.\footnote{See \Cref{subsection-adaptive-sketch-and-project} on how we can eliminate this assumption for an important class of methods.} In our previous work \cite{PAJAMA2021implicit,PAJAMA2021adaptive}, adaptive solvers relying on vector operations were shown to retain exponential convergence. While our previous work accounted for a number of existing solvers (e.g., \cite{strohmer2009,zouzias2013,gower2015,steinerberger2021,agmon1954,motzkin1954,bai2018,haddock2019}), adaptive solvers using high-efficiency block operations did not fall within our results. As block operations have been critical to achieving high-efficiency in traditional factorizations (e.g., QR \cite[Ch. 5]{golub2012}), classical iterative methods (e.g., Krylov Iterations \cite[Ch. 6]{saad2003}), randomized factorization methods \cite[\S 16.2]{martinsson2020}, and on GPUs \cite{baker2006}, adaptive solvers using block operations must be shown to retain exponential convergence.

Therefore, in this work, we provide generic sufficient conditions, that if satisfied by an a randomized block adaptive solver (RBAS), will guarantee a worst-case (i.e., with probability one) exponential rate of convergence.\footnote{Other worst-case rates can be provided using similar ideas that we present herein, but we do not know of a context where such rates are useful.} In particular, we provide these generic sufficient conditions and consequent worst-case exponential convergence rates in two contexts:
\begin{remunerate}
\item for row-action RBASs on consistent linear systems, which may be over-determined, under-determined or rank deficient (see \cref{theorem-convergence-row-finite,theorem-convergence-row-infinite}); and
\item for column-action RBASs for linear least squares problems, which may be over-determined, under-determined or rank deficient (see \cref{theorem-convergence-col-finite,theorem-convergence-col-infinite}).
\end{remunerate}

We then show how to apply these results to twenty-six different solvers, nine of which---to the best of our knowledge---are either novel or novel block-operation extensions of existing methods. Thus, in this work, we give end-users the tools to design effective solvers for their specific problems and environments.

The remainder of this work is organized as follows. In \cref{section-ce}, we demonstrate the disconnect between rates of convergence and efficiency. In \cref{section-rbas}, we present the two archetype RBASs, provide examples for each, state and discuss the refined properties that such solvers satisfy, and state our convergence results for each type. In \cref{section-convergence}, we provide a common formulation for the two types of RBASs, prove the convergence of these methods using this common formulation, and interlace numerical experiments that demonstrate key parts of the theory. In \cref{section-examples}, we show how to apply our convergence theory to a variety of existing and novel RBASs, and provide numerical experiments where appropriate. In \cref{section-conclusion}, we conclude.

\section{Counter Example} \label{section-ce}
Here, we demonstrate that the theoretical convergence rates of randomized solvers can be quite disconnected from their actual efficiency in practice. 
Consider a consistent, linear system with $n=10^7$ equations and $d=100$ unknowns represented with double precision. Owing to the size of the system relative to the 4 Gigabytes of memory available on an Intel i5 8th Generation CPU computer, the system is split into $0.5$ Gigabyte chunks, which contain at most 66,666 equations each. 

Consider an ``oracle" solver inspired by \cite{steinerberger2022}, which can randomly replace $d$ equations in the original system in such a way that the coefficients of the resulting replaced $d$ equations correspond to the rows of the $d \times d$ identity matrix and the system is still consistent. Then, with knowledge of the index of these $d$ equations, the solver applies Kaczmarz to these rows to solve the system. As a result, the oracle solver requires $d$ iterations and $\mathcal{O}(d)$ arithmetic operations. For this specific example, the oracle solver requires about $100$ arithmetic operations.

Consider an alternative solver, the random block Kaczmarz solver, which will randomly choose a chunk from the system and perform a block updated to its iterate. In our example, a single block Kaczmarz update requires approximately $10^7$ arithmetic operations, and, with an expected squared error rate of convergence of at least $0.993$ \cite[Theorem 1.2]{needell2014}, will required over 5,500 iterations and, correspondingly, over $5 \times 10^{10}$ operations to achieve an expected absolute \textbf{squared} error of $10^{-16}$. 

Clearly, from a theoretical perspective, the ``oracle'' solver is substantially faster than the random block Kaczmarz solver as the former requires $10$ fold fewer iterations and $10^8$ fewer operations. However, when applied to the system, the ``oracle'' solver is trounced by random block Kacmzarz (see \cref{figure-optimal-randblkkacz}). To understand this, the ``oracle'' solver needs to read in a new chunk (in expectation and in reality) to access the equations that it has embedded, which is highly expensive as it violates data locality. On the other hand, the block Kaczmarz solver simply does what it can with the information that is given in a single chunk, which turns out to contain sufficient information for finding a high quality solution in one iteration. To summarize, these solvers behave very differently in their theoretical convergence rates and in practice as this example shows. 

This observation is motivation to adapt such solvers to ensure that they are efficient for specific problems and computing environments. However, as the next example will demonstrate, such efficient solvers can lose effectiveness (i.e., suffer from arbitrarily slow convergence behavior). As a result, in this work, we provide sufficient conditions that, if satisfied by an adapted solver, will be effective---that is, the solver will have a worst-case exponential rate of convergence.

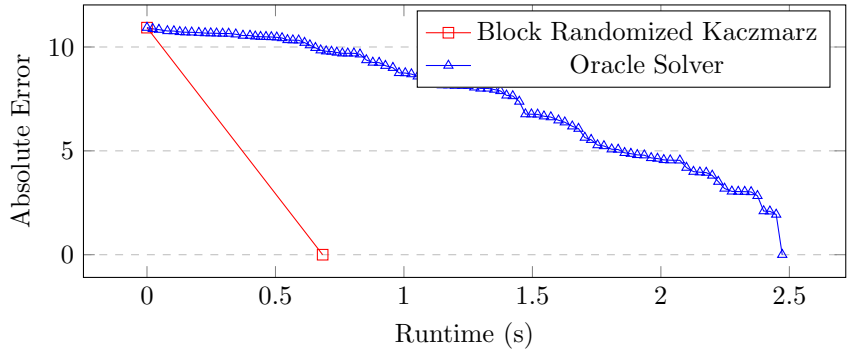
\begin{figure}{ht}
\centering
\begin{tikzpicture}
\begin{axis}[
	width=0.9\textwidth,
	height=0.4\textwidth,
    title={A Comparison of an Optimal Algorithm against Randomized Block Kaczmarz},
    xlabel={Runtime (s)},
    ylabel={Absolute Error},
    ymajorgrids=true,
    grid style=dashed,
]

\addplot[
    color=red,
    mark=square,
    ]
    coordinates {
(0.0,10.93145780259175)
(0.683665263,6.2264926340403e-15)
}; \addlegendentry{Block Randomized Kaczmarz}

\addplot[
	color=blue,
	mark=triangle,
	]
	coordinates {
(0.0,10.93145780259175)
(0.022315075,10.876403368686105)
(0.046519127,10.835090212025673)
(0.079532246,10.771471322165468)
(0.104694081,10.767037399712052)
(0.127604155,10.726112142734697)
(0.149542484,10.710453187725468)
(0.17397205100000002,10.706115771482123)
(0.201824734,10.689142973118772)
(0.226120585,10.669883375872791)
(0.248206377,10.665675292090325)
(0.270764163,10.658012895118631)
(0.294310388,10.648484702628767)
(0.317800569,10.648412162782169)
(0.34559338399999995,10.61244143524098)
(0.37360098599999997,10.549725585065078)
(0.40093215099999996,10.548642678987513)
(0.42144303699999996,10.512310699583866)
(0.44197764599999995,10.495241457582022)
(0.46317474399999997,10.490054961124988)
(0.483922078,10.480790559291213)
(0.50512859,10.465216458137293)
(0.525667202,10.432380225932208)
(0.546802315,10.339455141588715)
(0.5677540759999999,10.326961298199796)
(0.5902841029999999,10.322503625759197)
(0.6131673319999998,10.209402010092015)
(0.6339654459999998,10.089912105375454)
(0.6555636419999998,9.961356569583552)
(0.6760122939999998,9.84490076135242)
(0.6973620639999999,9.795387654885952)
(0.7180466419999999,9.78062368451608)
(0.7387459929999999,9.735150477368238)
(0.7603369839999998,9.696359771511869)
(0.7813899519999998,9.696297473965863)
(0.8052917529999999,9.695343956582063)
(0.8290688969999999,9.648100770717376)
(0.8537232959999999,9.373256804204903)
(0.8788717079999998,9.253493711117827)
(0.9034927559999998,9.246159156880818)
(0.9279020669999998,9.099502722214888)
(0.9571231869999998,9.00100765860374)
(0.9809784289999998,8.7530190483633)
(1.0051705699999998,8.74535112787664)
(1.0291035959999997,8.695741786139557)
(1.0537008279999998,8.588258661277983)
(1.0777405629999997,8.588082137335347)
(1.1017507519999996,8.312462261733014)
(1.1482986809999995,8.162456911246373)
(1.1735741499999996,8.159492400003808)
(1.1983503999999996,8.141109112221768)
(1.2269192559999995,8.136430369153011)
(1.2506244869999996,8.13523856201498)
(1.2737980009999996,8.057945442792699)
(1.2983119209999996,8.005502792253479)
(1.3257328409999996,8.005067812098419)
(1.3515176039999997,7.951737563964889)
(1.3754149499999997,7.8845337242092794)
(1.3992558059999998,7.67156704553693)
(1.4235409509999997,7.640624691066642)
(1.4489333089999996,7.37542879465068)
(1.4728056919999997,6.769341420493522)
(1.4961438229999997,6.763917665896411)
(1.5202655489999997,6.755227164923572)
(1.5448687889999997,6.666329075591486)
(1.5717271209999997,6.624893862122083)
(1.6015008349999997,6.472895088958492)
(1.6260112349999998,6.377460201441408)
(1.6551557859999997,6.182859432666703)
(1.6793680399999997,6.070462027693682)
(1.7040348349999996,5.6402681543623014)
(1.7285786759999997,5.534097523167512)
(1.7533819499999996,5.278849656060666)
(1.7783515849999996,5.233978565418249)
(1.8085300249999996,5.077560553490748)
(1.8336940679999996,5.06320647115571)
(1.8588950649999996,4.904909332918346)
(1.8832931169999996,4.864282559374986)
(1.9081139039999995,4.805563743150632)
(1.9354578899999995,4.790228233578456)
(1.9622719889999996,4.659528681909759)
(1.9867675469999995,4.628928758850661)
(2.0113829899999995,4.564585893879709)
(2.0364233079999994,4.5497780189231545)
(2.0739109899999995,4.543163036769238)
(2.0995512649999997,4.193390792855714)
(2.1268475359999996,3.999279835575316)
(2.1517400229999994,3.970181353266921)
(2.1764603609999993,3.952959287699454)
(2.1998190229999994,3.8141382586668375)
(2.2234156849999995,3.51842918573277)
(2.2477797119999994,3.1884679099052864)
(2.2761822959999995,3.048772375338734)
(2.3010806099999996,3.038032618166135)
(2.3262755999999998,3.0347933032154795)
(2.3511924609999997,3.0200298076967305)
(2.3758806549999996,2.83404169685307)
(2.4000134969999998,2.10152709591045)
(2.4242497349999996,2.0886421487918456)
(2.4487154549999994,1.9357000729466936)
(2.472261531999999,0.0)
}; \addlegendentry{Oracle Solver}
\end{axis}
\end{tikzpicture}
\caption{A comparison runtime of the ``Oracle'' Algorithm against Block Kaczmarz for the described system. The optimal algorithm achieves an absolute error of $0$ in $100$ iterations requiring $2.47$ seconds. Block Kaczmarz achieves an absolute error of $10^{-15}$ in $1$ iteration requiring $0.68$ seconds.}
\label{figure-optimal-randblkkacz}
\end{figure}

\label{section-notation}

\renewcommand{\nompreamble}{Below is a list of notation that is used throughout the text. The notation is organized into several groups: notation related to linear algebra; notation related to probability; and notation related to the iterative methods considered in this work.}

\nomenclature[A, 1]{$A$}{A coefficient matrix in $\mathbb{R}^{n \times d}$.}
\nomenclature[A, 2]{$b$}{A constant vector in $\mathbb{R}^{n}$}
\nomenclature[A, 3]{$\mathcal{P}$}{A projection operator onto the solution space of the linear system, $Ax=b$.}
\nomenclature[A, 5]{$Q_{\cdot}$}{A set of vectors that are orthonormal.}
\nomenclature[A, 6]{$q_{i,\cdot}$}{The $i^\mathrm{th}$ element of $Q_{\cdot}$.}
\nomenclature[A, 7]{$\col(\cdot)$}{The column space of a given matrix.}
\nomenclature[A, 8]{$\row(\cdot)$}{The row space of a given matrix.}

\nomenclature[B, 1]{$\mathbb{P}$}{The probability operator for a probability space.}
\nomenclature[B, 2]{$\mathbb{E}$}{The expectation operator for a probability space.}
\nomenclature[B, 3]{$\mathcal{F}_{\cdot}^{\cdot}$}{$\sigma$-algebras generated by a given set of random variables.}
\nomenclature[B, 4]{$\xi$}{A generic stopping time.}

\nomenclature[C, 1]{$x_k$}{Vectors in $\mathbb{R}^d$ generated by an iterative procedure.}
\nomenclature[C, 2]{$W_k$}{Matrix-valued or vector-valued random quantities generated by a possibly adaptive sampling procedure.}
\nomenclature[C, 3]{$\zeta_k$}{A random quantity that holds summary information about a possibly adaptive sampling procedure.}
\nomenclature[C, 4]{$\varphi$}{A possibly random, adaptive sampling, sketching or selection procedure.}
\nomenclature[C, 5]{$\chi_k$}{An indicator to whether the residual has a nonzero projection with respect to a given subspace.}
\printnomenclature

\section{Randomized Block Adaptive Solvers} \label{section-rbas}
Consider solving the consistent linear system
\begin{equation} \label{eqn-system}
Ax = b,
\end{equation}
or consider finding the least squares solution for a (possibly) inconsistent system by solving
\begin{equation} \label{eqn-ols}
\min_{x} \norm{ A x - b}_2,
\end{equation}
where $A \in \mathbb{R}^{n \times d}$; $x \in \mathbb{R}^d$; and $b \in \mathbb{R}^n$. We emphasize \textbf{we have not} required that $n < d$, $n > d$ or that $A$ has full rank; in other words, we allow for underdetermined systems, over determined systems and rank deficient linear systems.
To solve these systems, we will consider two archetypes of RBAS methods: row-action RBAS methods for \cref{eqn-system} or column-action RBAS methods for \cref{eqn-system,eqn-ols}. We will define each variation below, provide examples, state the assumptions, and present the main convergence results. 

\subsection{Row-action RBASs}
For row-action methods, we will need to assume
\begin{assumption} \label{assumption-consistency}
The system, \cref{eqn-system}, is consistent. That is, the set $\mathcal{H}:= \lbrace x \in \mathbb{R}^d : Ax=b \rbrace$ is nonempty.
\end{assumption}

With this assumption, we begin with an iterate $x_0 \in \mathbb{R}^d$ and some prior information, encapsulated by $\zeta_{-1} \in \mathfrak{Z}$, where $\mathfrak{Z}$ is finite in some sense (e.g., the product of a finite set and a finite dimensional linear space). We then generate a sequence of iterates, $\lbrace x_k : k \in \mathbb{N} \rbrace$, according to
\begin{equation} \label{eqn-base-row-update}
x_{k+1} = x_k - A^\intercal W_k (W_k^\intercal A A^\intercal W_k)^\dagger W_k^\intercal (A x_k - b),
\end{equation}
where $\cdot^\dagger$ represents a pseudo-inverse; and $\lbrace W_k \in \mathbb{R}^{n \times n_k} \rbrace$ are possibly random quantities (i.e., vectors or matrices) generated according to a possibly random, adaptive procedure, $\varphi_R$, which supplies
\begin{equation} \label{eqn-row-adaptive}
W_k, \zeta_k = \varphi_R(A, b, \lbrace x_j : j \leq k \rbrace, \lbrace W_j : j < k \rbrace, \lbrace \zeta_j : j < k \rbrace) \in \mathbb{R}^{n\times n_k} \times \mathfrak{Z}.
\end{equation}
We make several comments about this procedure. First, $n_k$ can be selected adaptively so long as it is known given the arguments of $\varphi_R$. Second, $\zeta_k$ contains information generated from previous iterations that may be essential to the operation of the adaptive procedure (see examples below). Third, we can change the inner product space as is done is \cite{gower2015} without issue (see \Cref{subsection-adaptive-sketch-and-project}). The next examples illustrate this formulation of row-action RBASs.

\begin{example}[Cyclic Vector Kaczmarz] \label{example-cyclic-vector-kaczmarz}
The cyclic vector Kaczmarz method cycles through the equations of $Ax=b$ (without reordering) and updates the the current iterate by projecting it onto the hyperplane that solves the selected equation. To rephrase the cyclic vector Kaczmarz method in our framework, let $\lbrace e_i : i =1,\ldots,n \rbrace$ denote the standard basis elements of $\mathbb{R}^n$. Moreover, let $\mathfrak{Z} = \lbrace 0 \rbrace \cup \mathbb{N}$, and $\zeta_{-1} = 0$. We then define $\varphi_R$ to be
\begin{equation}
\varphi_R( A, b, \lbrace x_j : j \leq k \rbrace, \lbrace W_j : j < k \rbrace, \lbrace \zeta_j : j < k \rbrace) = (e_{\mathrm{rem}(\zeta_{k-1}, n) + 1}, \zeta_{k-1}+1).
\end{equation}
With this choice of $(W_k, \zeta_k)$, we readily see that the described cyclic vector Kaczmarz method is equivalent to
\begin{equation}
x_{k+1} = x_k - A^\intercal e_{\mathrm{rem}(\zeta_{k-1},n) + 1} \frac{e_{\mathrm{rem}(\zeta_{k-1},n)+1}^\intercal(Ax_k - b)}{\norm{ A^\intercal e_{\mathrm{rem}(\zeta_{k-1},n) + 1} }_2^2},
\end{equation}
which is exactly \cref{eqn-base-row-update}. We highlight that $\varphi_R$ only depends on $\zeta_{k-1}$ and the number of equations in the linear system, which will be important in our discussion below. \hfill $\blacksquare$
\end{example}

\begin{example}[Random Permutation Block Kaczmarz] \label{example-random-permute-kaczmarz}
The random permutation block Kaczmarz method partitions the equations of $Ax = b$ (not necessarily equal partitions) into blocks of equations, generates a random permutation of the blocks, selects a block by cycling through the permutation, updates the current iterate by projecting it onto the hyperplane that solves all of the equations in the block, and, if the random permutation is exhausted, generates a new random permutation of the blocks.

To rephrase this method in our framework, let $\lbrace E_i \rbrace$ be matrices whose columns are generated by some partitioning of the identity matrix in $\mathbb{R}^{n \times n}$, and let $\epsilon = | \lbrace E_i \rbrace|$. Moreover, let $\mathfrak{Z}$ be product of the set of all permutations of $\lbrace 1,\ldots, \epsilon \rbrace$ with the empty set, and $\lbrace 0 \rbrace \cup \mathbb{N}$. Let $\lbrace Z_k : k+1 \in \mathbb{N} \rbrace$ be an independent random permutations of $\lbrace 1,\ldots, \epsilon \rbrace$. Let $\zeta_{-1} = (Z_0, 0)$. Then, we can define $\varphi_R$ to be
\begin{equation}
\begin{aligned}
&\varphi_R( A, b, \lbrace x_j : j \leq k \rbrace, \lbrace W_j : j < k \rbrace, \lbrace \zeta_j : j < k \rbrace) \\
&= \begin{cases}
(E_{\zeta_{k-1}[1][\mathrm{rem}(\zeta_{k-1}[2], \epsilon) + 1]}, (\zeta_{k-1}[1], \zeta_{k-1}[2] + 1)) & \mathrm{rem}(\zeta_{k-1}[2], \epsilon) < \epsilon -1 \\
(E_{\zeta_{k-1}[1][\epsilon]}, (Z_{\mathrm{div}(\zeta_{k-1}[2]+1, \epsilon)}, \zeta_{k-1}[2] + 1)) & \mathrm{rem}(\zeta_{k-1}[2], \epsilon) = \epsilon -1,
\end{cases}
\end{aligned}
\end{equation}
where $\zeta_{k}[1]$ is the permutation component of $\zeta_k$; $\zeta_{k}[1][j]$ is the $j^\mathrm{th}$ element of the permutation; and $\zeta_{k}[2]$ is the iteration counter. With this choice of $(W_k, \zeta_k)$, it is easy to see that the random permutation block Kaczmarz method can be equivalently written as \cref{eqn-base-row-update}. We highlight that $\varphi_R$ only depends on $\zeta_{k-1}$, the partitioning of the identity matrix, and the size of the partition. \hfill $\blacksquare$
\end{example}

\begin{example}[Greedy Block Selection Kaczmarz] \label{example-greedy-block-kaczmarz}
This method partitions the equations of $Ax=b$, computes the residual norm of each block at the given iteration, selects the block with the largest residual norm, and updates the current iterate by projecting it onto the hyperplane that solves all of the equations in the block.

To rephrase this method in our framework, let $\lbrace E_i \rbrace$ be matrices whose columns are generated by some partitioning of the $n \times n$ identity matrix, and let $\epsilon$ be the size of this set. Moreover, let $\mathfrak{Z} = \lbrace \emptyset \rbrace$, $\zeta_{-1} = \emptyset$, and let
\begin{equation}
\pi(k) = \argmax_{i = 1,\ldots,\epsilon} \norm{ E_i^\intercal (A x_k - b) }_2.
\end{equation}
Then, we can define $\varphi_R$ to be
\begin{equation}
\varphi_R(A, b, \lbrace x_j : j \leq k \rbrace, \lbrace W_j : j < k \rbrace, \lbrace \zeta_j : j < k \rbrace) = ( E_{\pi(k)}, \emptyset ).
\end{equation}
With this choice of $(W_k, \zeta_k)$, it is easy to see that this method is of the form \cref{eqn-base-row-update}. We emphasize that $\varphi_R$ only depends on $A$, $b$, $x_{k-1}$, and the partitioning of the identity matrix. \hfill $\blacksquare$
\end{example}

One of the key properties that is apparent in the examples above is that they are \textit{forgetful}. In other words, the choice of $(W_k, \zeta_{k})$ only depends on some finite number of previous iterations. To state this formally, for all $j+1 \in \mathbb{N}$ and $k \in [1,j+1] \cap \mathbb{N}$, let
\begin{equation} \label{eqn-sigma-algebra}
\mathcal{F}_{k}^j = \sigma(\zeta_{j-k}, x_{j-k+1}, W_{j-k+1},\ldots, W_{j-1}, \zeta_{j-1}, x_j),
\end{equation}
that is, the $\sigma$-algebra generated by the random variables indicated. Note, we take $\mathcal{F}_1^j = \sigma(\zeta_{j-1}, x_j)$ and $\mathcal{F}_0^j$ to be the trivial $\sigma$-algebra. Then, we can formalize this forgetfulness property as follows.

\begin{definition}[Markovian] \label{def-markovian-row}
A row-action RBAS is Markovian if there exists a finite $M \in \mathbb{N}$ such that for any measurable sets $\mathcal{W} \subset \mathbb{R}^{n \times n_k}$ and $\mathcal{Z} \subset \mathfrak{Z}$,
\begin{equation}
\condPrb{ W_k \in \mathcal{W}, \zeta_k \in \mathcal{Z}}{\mathcal{F}_{k+1}^{k}} = \condPrb{ W_k \in \mathcal{W}, \zeta_k \in \mathcal{Z}}{\mathcal{F}_{\min\lbrace M,\, k+1 \rbrace }^k}.
\end{equation}
\end{definition}
\begin{remark} \label{remark-markovian}
As discussed in \cite[Ch. 3]{meyn2012}, a Markov process that depends on some extended period of information can be rewritten into a Markov process that only depends on the most recent information only, which can be achieved by expanding the state of the Markov process. For a Markovian RBAS, we can do the same by adding this information in $\zeta_k$, so long as we ensure that $\mathfrak{Z}$ is finite. Thus, the value of $M$ in the preceding definition can always be taken as $1$. We also note that if $\mathfrak{Z}$ is finite, then it cannot be used to store all previous iterates.
\end{remark}

Another key property of the above examples is that either the iterate will be updated within some reasonable amount of time or the current iterate is the solution. For instance, in the random permutation block Kaczmarz method, if $x_0$ is not a solution then within $\epsilon$ iterations from $k=0$, we will find an $E_i^\intercal (Ax_0 - b) \neq 0$. As a result, $x_{0}$ will eventually be updated. We can generalize this property as follows.

\begin{definition}[N,$\pi$-Exploratory] \label{def-exploratory-row}
A row-action RBAS is $N,\pi$-Exploratory for some $N \in \mathbb{N}$ and $\pi \in (0,1]$ if
\begin{equation}
\sup_{\substack{ x_0 \in \mathbb{R}^d:\, x_0 \neq \Prj_{\mathcal{H}}x_0  \\ \zeta_{-1} \in \mathfrak{Z} }} \condPrb{ \bigcap_{j=0}^{N-1} \left\lbrace \col(A\T W_j) \perp x_0 - \Prj_{\mathcal{H}} x_0 \right\rbrace }{\mathcal{F}_1^0} \leq 1 - \pi.
\end{equation}
\end{definition}

Here, we come to a bifurcation point in the theory of RBAS methods based on whether $\lbrace \col(A^\intercal W_k) \rbrace$ is a finite set or if it is an infinite set. In all of the examples above, $\lbrace \col(A^\intercal W_k) \rbrace$ belong to a finite set. In this case, we have the following result.

\begin{corollary} \label{theorem-convergence-row-finite}
Let $A \in \mathbb{R}^{n \times d}$ and $b \in \mathbb{R}^n$, satisfying \cref{assumption-consistency}. Let $x_0 \in \mathbb{R}^d$ and $\zeta_{-1} \in \mathfrak{Z}$. Let $\lbrace x_k : k \in \mathbb{N} \rbrace$ be a sequence generated by \cref{eqn-base-row-update,eqn-row-adaptive} satisfying \cref{def-markovian-row} and \cref{def-exploratory-row} for some $N \in \mathbb{N}$ and $\pi \in (0,1]$. If the elements of $\lbrace \col(A^\intercal W_k) : k+1 \in \mathbb{N} \rbrace$ take value in a finite set, then either
\begin{remunerate}
\item there exists a stopping time $\tau$ with finite expectation such that $x_{\tau} = \Prj_{\mathcal{H}} x_0$; or
\item there exists a sequence of non-negative stopping times $\lbrace \tau_{j} : j +1 \in \mathbb{N} \rbrace$ for which $\E{ \tau_{j} } \leq j [ (\rnk{A} - 1) (N /\pi) + 1]$, and there exist $\gamma \in (0,1)$ and a sequence of random variables $\lbrace \gamma_j : j+1 \in \mathbb{N}\rbrace \subset (0,\gamma]$, such that
\begin{equation}
\Prb{ \bigcap_{j=0}^\infty \left\lbrace \norm{ x_{\tau_j} - \Prj_{\mathcal{H}}x_0 }_2^2 \leq \left( \prod_{\ell=0}^{j-1} \gamma_{\ell} \right) \norm{ x_0 - \Prj_{\mathcal{H}} x_0 }_2^2 \right\rbrace} = 1.
\end{equation}
\end{remunerate}
\end{corollary} 

Comparing \cref{theorem-convergence-row-finite} to classical results about the convergence of cyclic Kaczmarz-type methods (see \cite[Theorem 1]{dai2015}), we see that our result is a probabilistic analogue: rather than guaranteeing a certain amount of convergence within a fixed number of iterations, we offer a certain amount of convergence within a random number of iterations whose expectation is controlled by a regularly increasing value (i.e., $\E{ \tau_{j} } \leq j [ (\rnk{A} - 1) (N /\pi) + 1]$). Moreover, \cref{theorem-convergence-row-finite} includes the important possibility of the procedure terminating in a finite amount of time. Finally, we have a guaranteed worst case rate (i.e., with probability one) of convergence for all such methods. Of course, this rate is pessimistic, but, given the generality of the methods (e.g., adaptive, deterministic, random, etc.) that fall within the scope of our result, it is quite surprising that such a bound can be found under such few, very general assumptions.

Now, the alternative case to $\lbrace \col(A^\intercal W_k) \rbrace$ belonging to a finite set is that it belongs to an infinite set, for which the canonical example is the row-action analogue to \cref{example-gaussian-column}.\footnote{If $n_k > \rnk{A}$, then $\col( A^\intercal W_k) = \row(A)$ with probability one, which is covered by \cref{theorem-convergence-row-finite}.} Unfortunately, our strategy for proving \cref{theorem-convergence-row-finite} will break down for the infinite set case: in the proof of \cref{theorem-convergence-row-finite}, we set $\gamma$ to be the maximum over a finite set of elements that are all strictly less than one; however, if we attempt to use the same strategy for the infinite set case, we can find systems and methods such that the supremum over the same set produces a $\gamma = 1$ (an explicit example is constructed in \cref{subsection-infiniteset}).
Thus, rather than looking at the supremum, we can attempt to control the distribution of $\lbrace \gamma_{\ell}: \ell+1 \in \mathbb{N} \rbrace$. Surprisingly, we will only need to control the mean behavior of these random quantities rather than the entire distribution. 

To state this notion of control, we will need some notation. First, for each $\ell + 1 \in \mathbb{N}$, let
\begin{equation} \label{eqn-progress-indicator-row}
\chi_{\ell} = \begin{cases}
1 & x_{\ell+1} \neq x_{\ell}, \\
0 & \mathrm{otherwise},
\end{cases}
\end{equation}
be an indicator of whether we make progress in a given iteration. Moreover, for each $\ell+1 \in \mathbb{N}$, let $\mathfrak{Q}_{\ell}$ denote the collection of sets of vectors that are orthonormal and are a basis of $\col(A^\intercal W_\ell \chi_{\ell})$, and define $\mathcal{G}(Q_0,\ldots,Q_{\ell})$ to be the set of matrices whose columns are maximal linearly independent subsets of $\cup_{s=0}^{\ell} Q_{s}$ where $Q_s \in \mathfrak{Q}_s$.
With this notation, we have the following definition to control the distribution of $\lbrace 1 - \gamma_{\ell} : \ell+1 \rbrace$.
\begin{definition}[Uniformly Nontrivial] \label{def-unif-control-expectation-row}
A row-action RBAS is uniformly nontrivial if for any $\lbrace \mathcal{A}_k : \mathbb{R}^d \times \mathfrak{Z} \to \mathcal{F}_{k+1}^k \rbrace_{k+1 \in \mathbb{N}}$ such that
$\lim_{k \to \infty} \inf_{ x_0: x_0 \neq \Prj_{\mathcal{H}} x_0, \zeta_{-1} \in \mathfrak{Z} }$ $\incondPrb{ \mathcal{A}_{k}(x_0, \zeta_{-1})}{\mathcal{F}_{1}^0} = 1,$
there exists a $g_{\mathcal{A}} \in (0,1]$ such that
\begin{equation} \label{eqn-unif-control-expectation-row}
\inf_{ \substack{x_0: x_0 \neq \Prj_{\mathcal{H}} x_0 \\ \zeta_{-1} \in \mathfrak{Z}} } \sup_{k \in \mathbb{N} \cup \lbrace 0 \rbrace } \cond{ \sup_{ \substack{ Q_s \in \mathfrak{Q}_s \\ s \in \lbrace 0,\ldots,k \rbrace}} \min_{ G \in \mathcal{G}(Q_0,\ldots,Q_k)} \det(G^\intercal G) \1{ \mathcal{A}_k(x_0, \zeta_{-1})}}{\mathcal{F}^0_1} \geq g_{\mathcal{A}}.
\end{equation}
\end{definition}

Before stating the result, we point out some important connections and features of \cref{def-unif-control-expectation-row}. First, so long as $G \in \mathcal{G}(Q_0, \ldots,Q_k)$ is nontrivial, $\det(G^\intercal G) > 0$ with probability one. Thus, for each $x_0$ such that $x_0 \neq \Prj_{\mathcal{H}} x_0$ and $\zeta_{-1} \in \mathfrak{Z}$, there exists a $k \in \mathbb{N} \cup \lbrace 0 \rbrace$ such that
\begin{equation}
\cond{ \sup_{ \substack{ Q_s \in \mathfrak{Q}_s \\ s \in \lbrace 0,\ldots,k \rbrace}} \min_{ G \in \mathcal{G}(Q_0,\ldots,Q_k)} \det(G^\intercal G) \1{ \mathcal{A}_k(x_0, \zeta_{-1})}}{\mathcal{F}^0_1} > 0.
\end{equation}
Unfortunately, when we take the infimum over all allowed values of $x_0$ and $\zeta_{-1}$, we can no longer guarantee that the lower bound is zero, as supplied by \cref{def-unif-control-expectation-row}. 

Second, \cref{def-unif-control-expectation-row} is closely related, yet complementary to the foundational notion of \textit{uniformly integrable random variables}. To be specific, when a family of random variables is uniformly integrable, then the expected absolute value of the random variables in the family are uniformly bounded from above. Analogously and quite roughly, when we satisfy \cref{def-unif-control-expectation-row}, then the expected value of the random variables in the family are uniformly bounded from below.\footnote{We say this roughly as we ignore the supremum over $k$ to demonstrate the parallels between uniformly integrable families and a uniformly nontrivial RBAS.} Thus, we believe \cref{def-unif-control-expectation-row} to be quite a foundational property and will need to be validated on a case-by-case basis (possibly with the help of tools such as analogues to the theorems of \cite{poussin1915,dunford1939}).

Finally, we are only controlling the expected behavior in \cref{def-unif-control-expectation-row}, and we do not need to make any statements about higher moments, which is surprising as $\lbrace \gamma_{\ell} : \ell + 1 \rbrace$ is a dependent sequence, and usually dependencies require more complex moment statements (e.g., covariance relationships as in stationary processes). With these observations, we are ready for the next statement. 

\begin{corollary} \label{theorem-convergence-row-infinite}
Let $A \in \mathbb{R}^{n \times d}$ and $b \in \mathbb{R}^n$ satisfy \cref{assumption-consistency}. Let $x_0 \in \mathbb{R}^d$ and $\zeta_{-1} \in \mathfrak{Z}$. Let $\lbrace x_k : k \in \mathbb{N} \rbrace$ be a sequence generated by \cref{eqn-base-row-update,eqn-row-adaptive} satisfying \cref{def-markovian-row}, \cref{def-exploratory-row} for some $N \in \mathbb{N}$ and $\pi \in (0,1]$, and \cref{def-unif-control-expectation-row}. One of the following is true.
\begin{remunerate}
\item There exists a stopping time $\tau$ with finite expectation such that $x_{\tau} = \Prj_{\mathcal{H}} x_0$.
\item There exists a sequence of non-negative stopping times $\lbrace \tau_{j} : j +1 \in \mathbb{N} \rbrace$ for which $\E{ \tau_{j} } \leq j [ (\rnk{A} - 1) (N /\pi) + 1]$, there exists $\bar \gamma \in (0,1)$, and there exists a sequence of random variables $\lbrace \gamma_j : j+1 \in \mathbb{N}\rbrace \subset (0,1)$, such that
\begin{equation}
\Prb{ \bigcap_{j=0}^\infty \left\lbrace \norm{ x_{\tau_j} - \Prj_{\mathcal{H}}x_0 }_2^2 \leq \left( \prod_{\ell=0}^{j-1} \gamma_{\ell} \right) \norm{ x_0 - \Prj_{\mathcal{H}} x_0 }_2^2 \right\rbrace} = 1,
\end{equation}
where for any $\gamma \in (\bar \gamma,1)$,  $\inPrb{\cup_{L=0}^\infty \cap_{j=L}^\infty \lbrace \prod_{\ell=0}^{j-1} \gamma_{\ell} \leq \gamma^j \rbrace} = 1$.
\end{remunerate}
\end{corollary} 

\begin{remark} \label{remark-io-interpretation}
$\inPrb{\cup_{L=0}^\infty \cap_{j=L}^\infty \lbrace \prod_{\ell=0}^{j-1} \gamma_{\ell} \leq \gamma^j \rbrace} = 1$ is equivalent to: there exists a finite random variable, $L$, such that, for any $j \geq L$, $\prod_{\ell=0}^{j-1} \gamma_{\ell} \leq \gamma^j$ with probability one.
\end{remark}
\subsection{Column-action RBASs}
In contrast to row-action RBASs, column-action RBASs \textit{do not need} to assume that the system is consistent. Thus, we simply begin with an iterate $x_0 \in \mathbb{R}^d$ and some prior information, encapsulated by $\zeta_{-1} \in \mathfrak{Z}$, where $\mathfrak{Z}$ is finite in some sense. We then generate a sequence of iterates, $\lbrace x_k : k \in \mathbb{N} \rbrace$, according to
\begin{equation} \label{eqn-base-col-update}
x_{k+1} = x_k - W_k (W_k^\intercal A^\intercal A W_k)^\dagger W_k^\intercal A^\intercal (A x_k - b),
\end{equation}
where $\cdot^\dagger$ represents a pseudo-inverse; and $\lbrace W_k \in \mathbb{R}^{d \times n_k} \rbrace$ are possibly random quantities (i.e., vectors or matrices) generated according to a possibly random, adaptive procedure, $\varphi_C$, which supplies
\begin{equation} \label{eqn-col-adaptive}
W_k, \zeta_k = \varphi_C(A, b, \lbrace x_j : j \leq k \rbrace, \lbrace W_j : j < k \rbrace, \lbrace \zeta_j : j < k \rbrace) \in \mathbb{R}^{d \times n_k} \times \mathfrak{Z}.
\end{equation}
Note, our remarks about row-action RBASs apply here as well. We now present several examples.

\begin{example}[Cyclic Vector Coordinate Descent] \label{example-cyclic-cd}
Let $\lbrace e_i :i=1,\ldots d \rbrace$ denote the standard basis elements of $\mathbb{R}^d$. In this method, we update the iterate $x_k$ to $x_{k+1}$ by one coordinate at a time according to $x_{k+1} = x_k + e_i \alpha_k$ where $\alpha_k$ solves
\begin{equation}
\min_{\alpha \in \mathbb{R}} \norm{ (b - Ax_k) - A e_i \alpha }_2,
\end{equation}
which produces
\begin{equation}
x_{k+1} = x_k + e_i \frac{e_i^\intercal A^\intercal (b - A x_k)}{\norm{Ae_i}_2^2}.
\end{equation}
The choice of $e_i$ is determined by simply cycling through the basis elements in order. To rephrase this method within our formulation, we define $\mathfrak{Z} = \lbrace 0 \rbrace \cup \mathbb{N}$, $\zeta_{-1} = 0$, and
\begin{equation}
\varphi_C(A, b, \lbrace x_j : j \leq k \rbrace, \lbrace W_j : j < k \rbrace, \lbrace \zeta_j : j < k \rbrace) = ( e_{\mathrm{rem}(\zeta_{k-1},d) + 1}, \zeta_{k-1} + 1).
\end{equation}
With this choice of $(W_k, \zeta_k)$, we see that the cyclic vector coordinate descent method is equivalent to \cref{eqn-base-col-update}. We underscore that $\varphi_C$ only depends on $\zeta_{k-1}$ and the standard basis elements. \hfill $\blacksquare$
\end{example}

\begin{example}[Random Permutation Block Coordinate Descent] \label{example-cyclic-block-cd}
Let $\lbrace E_i : i = 1,\ldots, \epsilon \rbrace$ be matrices whose columns are generated by some partitioning of the $d \times d$ identity matrix. In this method, we have the update $x_{k+1} = x_k + E_i v_k$, where $v_k$ solves
\begin{equation}
\min_{v} \norm{ (b - Ax_k) - A E_i v }_2,
\end{equation}
which produces the update
\begin{equation}
x_{k+1} = x_k + E_i (E_i^\intercal A^\intercal A E_i)^\dagger E_i^\intercal A^\intercal (b - Ax_k).
\end{equation}
To choose $E_i$, we begin by randomly permuting $\lbrace E_i : i = 1,\ldots, \epsilon \rbrace$, pass through this permutation until it is exhausted, select a new random permutation, pass through this permutation until it is exhausted, and repeat. By following the column-action analogue of \cref{example-random-permute-kaczmarz}, we can rephrase this method within our formulation. \hfill $\blacksquare$
\end{example}

\begin{example}[Block Gaussian Column Space Descent] \label{example-gaussian-column}
Let $\lbrace W_k : k + 1 \in \mathbb{N} \rbrace$ be matrices with independent, identically distributed standard Gaussian components. In this method, we use the update $x_{k+1} = x_k + W_k v_k$, where $v_k$ solves
\begin{equation}
\min_{v} \norm{ (b - Ax_k) - A W_k v }_2,
\end{equation}
which produces the update
\begin{equation}
x_{k+1} = x_k + W_k ( W_k^\intercal A^\intercal A W_k)^\dagger W_k^\intercal A^\intercal ( b - A x_k).
\end{equation}
It is clear that this update is exactly in the form of \cref{eqn-base-col-update}. Moreover, we can choose $\mathfrak{Z} = \lbrace \emptyset \rbrace$, $\zeta_{-1} = \emptyset$, and we can define
\begin{equation}
\varphi_C( A, b, \lbrace x_j : j \leq k \rbrace, \lbrace W_j : j < k \rbrace, \lbrace \zeta_j : j < k \rbrace) = (W_k, \emptyset).
\end{equation}
Thus, this method fits within our formulation. \hfill $\blacksquare$
\end{example}

As these example demonstrate, column-action RBAS methods are also forgetful---that is, they satisfy the following analogue of \cref{def-markovian-row}.
\begin{definition}[Markovian] \label{def-markovian-col}
A column-action RBAS is Markovian if there exists a finite $M \in \mathbb{N}$ such that for any measurable sets $\mathcal{W} \subset \mathbb{R}^{d \times n_k}$ and $\mathcal{Z} \subset \mathfrak{Z}$,
\begin{equation}
\condPrb{ W_k \in \mathcal{W}, \zeta_k \in \mathcal{Z}}{\mathcal{F}_{k+1}^{k}} = \condPrb{ W_k \in \mathcal{W}, \zeta_k \in \mathcal{Z}}{\mathcal{F}_{\min\lbrace M,\, k+1 \rbrace }^k}.
\end{equation}
\end{definition}
\begin{remark}
See \cref{remark-markovian}.
\end{remark}

Similarly, just as with row-action methods, column-action RBASs are also $N,\pi$-Exploratory. To state this definition, define $r^* = -\Prj_{\ker(A^\intercal)} b$.
\begin{definition}[$N,\pi$-Exploratory] \label{def-exploratory-col}
A column-action RBAS is $N,\pi$-\\ Exploratory for some $N \in \mathbb{N}$ and $\pi \in (0,1]$ if
\begin{equation}
\sup_{\substack{ x_0 \in \mathbb{R}^d:\, Ax_0 - b \neq r^*  \\ \zeta_{-1} \in \mathfrak{Z} }} \condPrb{ \bigcap_{j=0}^{N-1} \left\lbrace \col(A W_j) \perp Ax_0-b \right\rbrace }{\mathcal{F}_1^0} \leq 1 - \pi.
\end{equation}
\end{definition}
Note, the Block Gaussian Column Space Descent method, \cref{example-gaussian-column}, is $1,1$-Exploratory.

Just as for row-action methods, we will have a bifurcation of the theory for the convergence of column-action methods based on whether the elements of $\lbrace \col(A W_k) \rbrace$ take value in a finite set. In the case that they do, we have the following analogue of \cref{theorem-convergence-row-finite}.
\begin{corollary} \label{theorem-convergence-col-finite}
Let $A \in \mathbb{R}^{n \times d}$, $b \in \mathbb{R}^n$, and $r^* = -\Prj_{\ker(A^\intercal)} b$. Let $x_0 \in \mathbb{R}^d$ and $\zeta_{-1} \in \mathfrak{Z}$. Let $\lbrace x_k : k \in \mathbb{N} \rbrace$ be a sequence generated by \cref{eqn-base-col-update,eqn-col-adaptive} satisfying \cref{def-markovian-col} and \cref{def-exploratory-col} for some $N \in \mathbb{N}$ and $\pi \in (0,1]$. If the elements of $\lbrace \col(AW_k) : k+1 \in \mathbb{N} \rbrace$ take value in a finite set, then either
\begin{remunerate}
\item there exists a stopping time $\tau$ with finite expectation such that $Ax_{\tau} - b = r^*$; or
\item there exists a sequence of non-negative stopping times $\lbrace \tau_{j} : j +1 \in \mathbb{N} \rbrace$ for which $\E{ \tau_{j} } \leq j [ (\rnk{A} - 1) (N /\pi) + 1]$,  and there exist $\gamma \in (0,1)$ and a sequence of random variables $\lbrace \gamma_j : j+1 \in \mathbb{N}\rbrace \subset (0,\gamma]$, such that
\begin{equation}
\Prb{ \bigcap_{j=0}^\infty \left\lbrace \norm{ Ax_{\tau_j} - b - r^* }_2^2 \leq \left( \prod_{\ell=0}^{j-1} \gamma_{\ell} \right) \norm{ Ax_0 - b - r^* }_2^2 \right\rbrace} = 1.
\end{equation}
\end{remunerate}
\end{corollary} 

The same comments for \cref{theorem-convergence-row-finite} apply to \cref{theorem-convergence-col-finite}. Also, just as for \cref{theorem-convergence-row-finite}, \cref{theorem-convergence-col-finite} does not cover \cref{example-gaussian-column} if $n_k < \rnk{A}$. For the infinite set case, we will make use of the same notation as before with the following modifications. First,
\begin{equation} \label{eqn-progress-indicator-col}
\chi_{\ell} = \begin{cases}
1 & Ax_{\ell+1} - b \neq Ax_{\ell} - b, \\
0 & Ax_{\ell+1} - b = Ax_{\ell} - b.
\end{cases}
\end{equation}
Second, let $\mathfrak{Q}_{\ell}$ denote the collection of sets of vectors that are orthonormal and are a basis of $\col(A W_\ell \chi_{\ell})$. We now state the analogues of \cref{def-unif-control-expectation-row,theorem-convergence-row-infinite}.
\begin{definition}[Uniformly Nontrivial] \label{def-unif-control-expectation-col}
A column-action RBAS is uniformly nontrivial if for any $\lbrace \mathcal{A}_k : \mathbb{R}^d \times \mathfrak{Z} \to \mathcal{F}_{k+1}^k \rbrace_{ k+1 \in \mathbb{N}}$ such that 
$\lim_{k \to \infty} \inf_{x_0: Ax_0 \neq b, \zeta_{-1} \in \mathfrak{Z} }$ $\incondPrb{ \mathcal{A}_{k}(x_0, \zeta_{-1})}{\mathcal{F}_{1}^0} = 1,
$
there exists a $g_{\mathcal{A}} \in (0,1]$ such that
\begin{equation} \label{eqn-unif-control-expectation-col}
\inf_{ \substack{x_0: Ax_0 - b \neq r^* \\ \zeta_{-1} \in \mathfrak{Z}} } \sup_{k \in \mathbb{N} \cup \lbrace 0 \rbrace } \cond{ \sup_{ \substack{ Q_s \in \mathfrak{Q}_s \\ s \in \lbrace 0,\ldots,k \rbrace}} \min_{ G \in \mathcal{G}(Q_0,\ldots,Q_k)} \det(G^\intercal G) \1{ \mathcal{A}_k(x_0, \zeta_{-1})}}{\mathcal{F}^0_1} \geq g_{\mathcal{A}}.
\end{equation}
\end{definition}

\begin{corollary} \label{theorem-convergence-col-infinite}
Let $A \in \mathbb{R}^{n \times d}$, $b \in \mathbb{R}^n$, and $r^* = -\Prj_{\ker(A^\intercal)} b$. Let $x_0 \in \mathbb{R}^d$ and $\zeta_{-1} \in \mathfrak{Z}$. Let $\lbrace x_k : k \in \mathbb{N} \rbrace$ be a sequence generated by \cref{eqn-base-col-update,eqn-col-adaptive} satisfying \cref{def-markovian-col}, \cref{def-exploratory-col} for some $N \in \mathbb{N}$ and $\pi \in (0,1]$, and \cref{def-unif-control-expectation-col}. One of the following is true.
\begin{remunerate}
\item There exists a stopping time $\tau$ with finite expectation such that $A x_{\tau} - b= r^*$.
\item There exists a sequence of non-negative stopping times $\lbrace \tau_{j} : j +1 \in \mathbb{N} \rbrace$ for which $\E{ \tau_{j} } \leq j [ (\rnk{A} - 1) (N /\pi) + 1]$, there exists $\bar \gamma \in (0,1)$, and there exists a sequence of random variables $\lbrace \gamma_j : j+1 \in \mathbb{N}\rbrace \subset (0,1)$, such that
\begin{equation}
\Prb{ \bigcap_{j=0}^\infty \left\lbrace \norm{ Ax_{\tau_j} - b - r^* }_2^2 \leq \left( \prod_{\ell=0}^{j-1} \gamma_{\ell} \right) \norm{ Ax_0 -b - r^*}_2^2 \right\rbrace} = 1,
\end{equation}
where for any $\gamma \in (\bar \gamma,1)$,  $\inPrb{\cup_{L=0}^\infty \cap_{j=L}^\infty \lbrace \prod_{\ell=0}^{j-1} \gamma_{\ell} \leq \gamma^j \rbrace} = 1$.
\end{remunerate}
\end{corollary} 

\begin{remark}
See \cref{remark-io-interpretation}.
\end{remark}

\section{Convergence Theory} \label{section-convergence}
We now prove \cref{theorem-convergence-row-finite,theorem-convergence-row-infinite,theorem-convergence-col-finite,theorem-convergence-col-infinite} by the following steps.
\begin{remunerate}
\item In \cref{subsection-common}, we will write row-action and column-action methods using a common form, which reveals that the iterates (in the common form) are generated by products of orthogonal projections, which raises the questions: when will this sequence of products of orthogonal projections produce a reduction in the norms of the iterates and how big will this reduction be?
\item In \cref{subsection-meany}, we will answer this question by proving a generalized block Meany inequality, which states that when the iterate is in a space generated by the a sequence of projection matrices, we are guaranteed a certain amount of reduction in the norms of the iterates. Of course, this raises the question: when will the iterate be in this space?
\item In \cref{subsection-stopping}, we define a stopping time for each iterate that, when finite, implies that the iterate will be in the aforementioned space. We show that when a RBAS is Markovian and $N,\pi$-exploratory, then, starting at any iterate, this stopping time is finite in expectation and we derive an explicit bound on this expectation. 
\item Once we have established the finiteness of this stopping time, we can then apply our generalized block Meany's inequality to guarantee a reduction in the norm of the iterates. However, owing to the possible randomness of the procedure and the stopping times, we will need to find a deterministic control over the reduction constant provided by our generalized block Meany's inequality. In \cref{subsection-finiteset}, we will find this deterministic value by using the worst case over a finite set, which will prove \cref{theorem-convergence-row-finite,theorem-convergence-col-finite}. In \cref{subsection-infiniteset}, we will find this deterministic value by using the uniformly nontrivial property, which will prove \cref{theorem-convergence-row-infinite,theorem-convergence-col-infinite}.
\end{remunerate}
\subsection{Common Formulation} \label{subsection-common} 
Our first step will be to rewrite row-action and column-action RBASs, and the corresponding definitions using a common formulation. To this end, we define
\begin{equation} \label{eqn-common-y-def}
y_k = \begin{cases}
x_k - \Prj_{\mathcal{H}} x_0 & \text{if \cref{eqn-base-row-update} and \cref{assumption-consistency}}, \\
Ax_k - b - r^* & \text{if \cref{eqn-base-col-update}},
\end{cases}
\end{equation}
where $r^* = - \Prj_{\ker(A^\intercal)} b$.
Owing to this definition, the update $y_k$ to $y_{k+1}$ is
\begin{equation} \label{eqn-base-common-update}
y_{k+1} = (I - \Prj_k ) y_k,
\end{equation}
where  $\Prj_k$ are orthogonal projection matrices defined by
\begin{equation} \label{eqn-base-proj-def}
\Prj_k = \begin{cases}
A^\intercal W_k ( W_k^\intercal A A^\intercal W_k)^\dagger W_k^\intercal A & \text{if \cref{eqn-base-row-update}} \\
AW_k (W_k^\intercal A^\intercal A W_k)^\dagger W_k^\intercal A^\intercal & \text{if \cref{eqn-base-col-update}}.
\end{cases}
\end{equation}
Thus, with these definitions, it is enough to prove convergence and rate of convergence results about $\lbrace y_k \rbrace$.
\begin{remark}
We can change the inner product space as done in \cite{gower2019}, and we would still recover \cref{eqn-base-common-update} with a simple change of variables. See \cite{pritchard2022}.
\end{remark}

To focus on $\lbrace y_k \rbrace$, we can update some of our definitions in terms of $\lbrace y_k \rbrace$ and $\lbrace \Prj_k \rbrace$.
\begin{definition}[Markovian] \label{def-markovian}
An RBAS (see \cref{eqn-base-common-update}) is Markovian if there exists a finite $M \in \mathbb{N}$ such that for any measurable sets $\mathcal{W}$ and $\mathcal{Z} \subset \mathfrak{Z}$,
\begin{equation}
\condPrb{ W_k \in \mathcal{W}, \zeta_k \in \mathcal{Z}}{\mathcal{F}_{k+1}^{k}} = \condPrb{ W_k \in \mathcal{W}, \zeta_k \in \mathcal{Z}}{\mathcal{F}_{\min\lbrace M,\, k+1 \rbrace }^k}.
\end{equation}
\end{definition}
\begin{definition}[$N,\pi$-Exploratory] \label{def-exploratory}
An RBAS (see \cref{eqn-base-common-update}) is $N,\pi$-Exploratory for some $N \in \mathbb{N}$ and $ \pi \in (0,1]$ if
\begin{equation}
\sup_{\substack{ y_0: y_0 \neq 0  \\ \zeta_{-1} \in \mathfrak{Z}}} \condPrb{ \bigcap_{j=0}^{N-1} \left\lbrace \col(\Prj_j) \perp y_0 \right\rbrace }{\mathcal{F}_1^0} \leq 1 - \pi.
\end{equation}
\end{definition}
\begin{definition}[Uniformly Nontrivial] \label{def-unif-control}
An RBAS (see \cref{eqn-base-common-update}) is uniformly nontrivial if for any $\lbrace \mathcal{A}_k : \mathbb{R}^d \times \mathfrak{Z} \to \mathcal{F}_{k+1}^k : k+1 \in \mathbb{N} \rbrace$ such that
$\lim_{k \to \infty} \inf_{ y_0: y_0 \neq 0, \zeta_{-1} \in \mathfrak{Z}}$ $ \incondPrb{ \mathcal{A}_{k}(x_0, \zeta_{-1})}{\mathcal{F}_{1}^0} = 1,$
there exists a $g_{\mathcal{A}} \in (0,1]$ such that
\begin{equation} \label{eqn-unif-control}
\inf_{ \substack{y_0: y_0 \neq 0 \\ \zeta_{-1} \in \mathfrak{Z}} } \sup_{k \in \mathbb{N} \cup \lbrace 0 \rbrace } \cond{ \sup_{ \substack{ Q_s \in \mathfrak{Q}_s \\ s \in \lbrace 0,\ldots,k \rbrace}} \min_{ G \in \mathcal{G}(Q_0,\ldots,Q_k)} \det(G^\intercal G) \1{ \mathcal{A}_k(x_0, \zeta_{-1})}}{\mathcal{F}^0_1} \geq g_{\mathcal{A}}.
\end{equation}
\end{definition}

\subsection{Generalized Block Meany Inequality} \label{subsection-meany}
From \cref{eqn-base-common-update}, we see that $\lbrace y_k \rbrace$ are updated by applying a sequence of orthogonal projection matrices. We can now ask whether this application of projection matrices will drive $\lbrace \innorm{y_k}_2 \rbrace$ to zero and at what rate. This question was first answered for products of projections of the form $I - q q^\intercal$ in \cite{meany1969}, where the $q$'s in the product are linearly independent---a result known as Meany's inequality. Meany's inequality has been generalized in two ways. First, Meany's inequality was extended to products of the form $I - QQ^\intercal$ in \cite[Theorem 4.1]{bai2013}, where each $Q$ has orthonormal columns and the concatenation of all $Q$'s in the product form a nonsingular matrix. Second, in \cite[Theorem 4.1]{PAJAMA2021implicit}, Meany's inequality was generalized to the case in which there is a loss of independence between the $q$'s. Here, we generalize all of these results. 

To state our result, we will need to update some notation. 
For any $k$, let $\chi_k$ be $1$ if $\Prj_k y_k \neq 0$ and zero otherwise, 
which we see is equivalent to \cref{eqn-progress-indicator-row,eqn-progress-indicator-col} for the two different RBAS types.
Let $\mathfrak{Q}_j$ be the set of orthogonal bases of $\col(\Prj_j)$ for all $j+1 \in \mathbb{N}$. Finally, let $\mathcal{C}_k^j = \col(\Prj_j \chi_j) + \cdots + \col(\Prj_{j+k} \chi_{j+k})$.

\begin{theorem}[Generalized Block Meany's Inequality] \label{theorem-block-meany}
Let $j+1, k +1 \in \mathbb{N}$. Then, for any $y \in \mathcal{C}_k^j$ with $\innorm{y}_2 = 1$, 
$\innorm{ (I - \Prj_{j+k}\chi_{j+k})\cdots ( I - \Prj_j\chi_j) y }_2^2$ is no greater than
$1 - \sup_{Q_i \in \mathfrak{Q}_i, i \in \lbrace j,\ldots,j+k \rbrace}\min_{ G \in \mathcal{G}(Q_j,\ldots,Q_{j+k})} \det(G^\intercal G)$.

\end{theorem}
\begin{proof}
Let $n_i = \dim( \col{\Prj_i})$. Begin by fixing $Q_i \in \mathfrak{Q}_i$ for $i=j,\ldots,j+k$, and let $\lbrace q_{i,\ell} : \ell=1,\ldots,n_k\rbrace$ denote the elements of $Q_i$. We can now follow the strategy of \cite{bai2013}.
Letting the product notation indicate terms with increasing index are being multiplied from the left, note that
$
\prod_{i=j}^{j+k} (I - \Prj_i\chi_{i}) = \prod_{i=j}^{k+j} [ \prod_{\ell=1}^{n_{i}} (I - q_{i,\ell}q_{i,\ell}^\intercal\chi_{i}) ].
$
Therefore, \cite[Theorem 4.1]{PAJAMA2021implicit} provides
\begin{equation}
\begin{aligned}
\norm{ (I - \Prj_{j+k}\chi_{j+k})\cdots ( I - \Prj_j\chi_j) y }_2^2 \leq \left( 1 - \min_{ G \in \mathcal{G}(Q_j,\ldots,Q_{j+k})} \det(G^\intercal G)\right) \norm{y}_2^2.
\end{aligned}
\end{equation}
This statement holds for every choice of $Q_i \in \mathfrak{Q}_i$. Therefore, the result follows.
\end{proof}

We pause for a moment to explain the importance of the supremum term in \cref{theorem-block-meany}. We can first ask whether the choice of $Q_i \in \mathfrak{Q}_k$ will make any tangible difference. Consider the very simple situation of applying a block row selection method of a $4 \times 3$ matrix such that $A^\intercal W_1$ and $A^\intercal W_2$ generate
\begin{equation} \label{eqn-meany-constant-example}
\begin{bmatrix}
2 & 1 & 0 \\
-1 & 2 & 3
\end{bmatrix}
\quad\mathrm{and}\quad
\begin{bmatrix}
1 & -3 & 6 \\
0 & 1 & -5
\end{bmatrix},
\end{equation}
respectively. We can compute $\min_{G \in \mathcal{G}(Q_1,Q_{2})} \det(G^\intercal G )$, which we will refer to as Meany's constant, from $10,000$ uniformly sampled bases for the row spaces of these two matrices. The average value of Meany's constant is $0.12$ with a standard deviation of $0.11$, and the quantiles from this experiment are shown in \cref{table-meany-constant-example}. The supremum of Meany's constant is also included in \cref{table-meany-constant-example}. We see that the supremum is at least $8$ fold larger than the average, and over $10^6$ fold larger than the $0.001$ quantile. Thus, the supremum term is extremely important in finding better bounds on the rate of convergence. 

\begin{table}[htb]
\caption{Supremum and quantiles for randomly sampled Meany's Constant for the example in \cref{eqn-meany-constant-example}.}
\label{table-meany-constant-example}
\centering
\begin{tabular}{@{}lccccccc|c@{}} \toprule
\textbf{Quantile} & $0.001$ & $0.05$ & $0.25$ & $0.5$ & $0.75$ & $0.95$ & $0.999$ & \textbf{Sup.} \\
\textbf{Value} & $1.9 \times 10^{-7}$ & $6.5 \times 10^{-4}$ & $0.02$ & $0.09$ & $0.20$ & $0.34$ & $0.47$ & $0.9995$ \\\bottomrule
\end{tabular}
\end{table}

Moreover, the supremum term also underscores the importance of block methods over vector methods (see \cite{PAJAMA2021adaptive}) from a theoretical perspective. For vector methods, there are only two choices in the set $\mathfrak{Q}_i$, and both produce the same value of Meany's constant. Thus, for vector methods, Meany's constant will only differ based on which vectors are seen. To demonstrate this, we run cyclic Kaczmarz and block cyclic Kacmzarz on the coefficient matrix in \cref{eqn-meany-constant-example} until an absolute error of $10^{-4}$ is achieved. We plot one minus the ratio in norm error squared for each method, and the corresponding Meany's constants in \cref{figure-vector-vs-block}. Clearly, we see that the block method is substantially superior over the corresponding vector method both in practice and in theory.

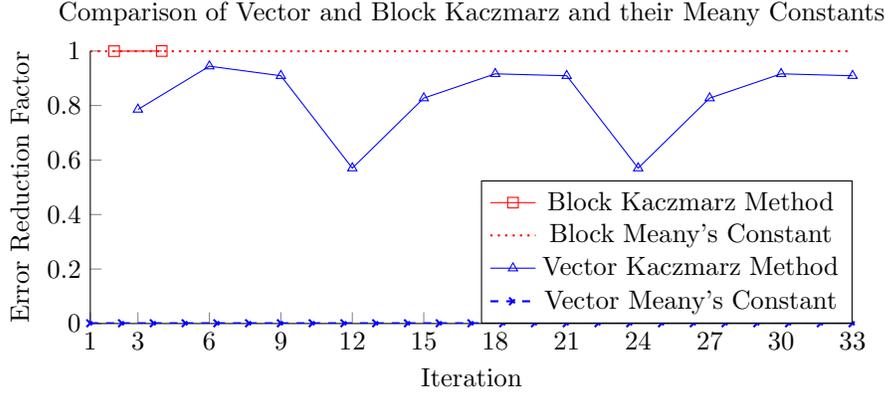
\begin{figure}
\centering
\begin{tikzpicture}
\begin{axis}[
	axis y line*=left,
	axis x line*=none,
	width=0.9\textwidth,
	height=0.4\textwidth,
    title={Comparison of Vector and Block Kaczmarz and their Meany Constants},
    xlabel={Iteration},
    ylabel={Error Reduction Factor},
    xmin=1, xmax=33,
    ymin=0, ymax=1,
    xtick={1,3,6,9,12,15,18,21,24,27,30,33},
    ytick={0,0.2,0.4,0.6,0.8,1.0},
    ymajorgrids=false,
    grid style=dashed,
    legend style={at={(1,0)},anchor=south east}
]

\addplot[
	color=red,
	mark=square,
	]
	coordinates {
(2, 0.9996256691164543)
(4, 0.9999997147955173)	
}; \addlegendentry{Block Kaczmarz Method}

\addplot[domain = 1:33,
	thick,
	dotted,
	red,
	]{0.9995}; \addlegendentry{Block Meany's Constant}

\addplot[
    color=blue,
    mark=triangle,
    ]
    coordinates {
(3, 0.7856332703213611)
(6, 0.9443511540544351)
(9, 0.9093111952955963)
(12, 0.5703308926931956)
(15, 0.8271971450463169)
(18, 0.9163110729664143)
(21, 0.909311195295603)
(24, 0.5703308926931752)
(27, 0.8271971450463463)
(30, 0.9163110729665436)
(33, 0.9093111952954628)
}; \addlegendentry{Vector Kaczmarz Method}

\addplot[domain = 1:33,
	thick,
	dashed,
	mark=x,
	blue,
    ]{0.0009555661729575714}; \addlegendentry{Vector Meany's Constant}

\end{axis}

\end{tikzpicture}
\caption{A plot of one less the ratio of norm errors squared between every three iterates for the vector method, and every two iterates for the block method. The horizontal lines correspond to the values of Meany's constant in \cref{theorem-block-meany}.}
\label{figure-vector-vs-block}
\end{figure}

\subsection{Stopping Times} \label{subsection-stopping}
To apply \cref{theorem-block-meany}, we need to determine for which $k$, $y_j \in \mathcal{C}_k^j$. 
Given that $\mathcal{C}_{k}^j$ is random, we will have to allow the time at which this occurs to be random, as follows.\footnote{It is understood that if the condition fails to occur then the stopping time is infinite.} For $j+1 \in \mathbb{N}$, let
\begin{equation} \label{eqn-nu-stop}
\nu(j) = \min \left\lbrace k \geq 0: y_{j} \in \mathcal C_k^j  \right\rbrace.
\end{equation}
Thus, when $\nu(j)$ is finite, \cref{theorem-block-meany} implies $\innorm{y_{j + \nu(j) +1}}_2^2 / \innorm{y_j}_2^2$ is no greater than 
$1 - \sup_{Q_i \in \mathfrak{Q}_i , i \in \lbrace j,\ldots,j+\nu(j) \rbrace }\min_{ G \in \mathcal{G}(Q_j,\ldots,Q_{j+\nu(j)})} \det(G^\intercal G)$.
Hence, we need to determine whether $\nu(j)$ is finite for all $j$, and, ideally, we want to bound it, at the very least, in expectation.
To this end, we will study another stopping time that is an upper bound on $\nu(j)$, and will find a bound on this new stopping time's expectation. We will begin by specifying this stopping time and showing that it is an upper bound on $\nu(j)$.
\begin{lemma} \label{lemma-nu-upper-bound}
For any $j+1 \in \mathbb{N}$, let $\nu(j)$ be defined as in \cref{eqn-nu-stop}. Then, $\nu(j) \leq \min \lbrace k \geq 0: y_{j+k+1} \in \linspan{y_{j},\ldots,y_{j+k}}, ~\chi_{j+k} \neq 0 \rbrace$.
\end{lemma}
\begin{proof}
We begin with a key fact. By \cref{eqn-base-common-update}, $y_{k+1} - y_k \in \col(\Prj_k \chi_k)$ for any $k+1 \in \mathbb{N}$. It follows that $y_i \in \inlinspan{ \mathcal{C}_{k}^j \cup \lbrace y_{\ell} \rbrace }$ for any $\ell, i \in [j,\ldots,j+k+1] \cap \mathbb{N}$.

Now, let $\nu(j)' = \min \lbrace k \geq 0: y_{j+k+1} \in \linspan{y_{j},\ldots,y_{j+k}}, ~\chi_{j+k} \neq 0 \rbrace$. Then, by the preceding fact, if 
$y_{j+\nu(j)'} \in \mathcal{C}_{\nu(j)'}^j$, then $y_j \in \mathcal{C}_{\nu(j)'}^j$. Thus, $\nu(j) \leq \nu(j)'$ by the minimality of $\nu(j)$. So it is enough to show $y_{j+\nu(j)'} \in \mathcal{C}_{\nu(j)'}^j$.

Let $r$ denote the dimension of $\inlinspan{ \mathcal{C}_{\nu(j)'}^j \cup \lbrace y_{j + \nu(j)'} \rbrace }$. Then, by the Gram-Schmidt procedure, there exist $\phi_1,\ldots,\phi_{r-1} \in \mathcal{C}_{\nu(j)'}^j$ such that the set of vectors $\lbrace y_{j+\nu(j)'},\phi_1,\ldots,\phi_{r-1} \rbrace$ are an orthogonal basis for $\inlinspan{ \mathcal{C}_{\nu(j)'}^j \cup \lbrace y_{j + \nu(j)'} \rbrace }$. Now, by the definition of $\nu(j)'$, there exist scalars $c_0,\ldots,c_{r-1}$ such that
$y_{j+\nu(j)'+1} = c_0 y_{j+\nu(j)'} + c_1 \phi_1 + \cdots c_{r-1} \phi_{r-1}.$
Plugging this into \cref{eqn-base-common-update},
\begin{equation} \label{eqn-proof-nu-prime}
\begin{aligned}
c_0 y_{j+\nu(j)'} + c_1 \phi_1 + \cdots c_{r-1} \phi_{r-1} 
 = y_{j + \nu(j)'} - \Prj_{j + \nu(j)'} y_{j + \nu(j)'} \chi_{j + \nu(j)'},
\end{aligned}
\end{equation}
which gives rise to two cases. In the first case, we assume that $c_0 \neq 1$. Then, rearranging \cref{eqn-proof-nu-prime}, we conclude $y_{j + \nu(j)'} \in \inlinspan{\phi_1,\ldots,\phi_{r-1}} + \col( \Prj_{j + \nu(j)'} \chi_{j + \nu(j)'}) =\mathcal{C}_{\nu(j)'}^j$. 
In the second case, $c_0 =1$. Then, multiplying both sides of \cref{eqn-proof-nu-prime} by $y_{j + \nu(j)'}^\intercal$,
$\sum_{i=1}^{r-1} c_i y_{j + \nu(j)'}^\intercal \phi_i = - \innorm{ \Prj_{j + \nu(j)'} y_{j + \nu(j)'}}_2^2 \chi_{j + \nu(j)'}$. By the orthogonality of $\phi_i$ and $y_{j+\nu(j)'}$, the left hand side is zero. The right hand side can only be zero if $\chi_{j + \nu(j)'} = 0$, which contradicts the definition of $\nu(j)'$. To summarize these two cases, we showed that $y_{j+\nu(j)'} \in \mathcal{C}_{\nu(j)'}^j$. The result follows.
\end{proof}

\begin{theorem} \label{theorem-nu-finite}
Let $\xi$ be an arbitrary, finite stopping time with respect to $\lbrace \mathcal{F}_{k+1}^k : k + 1 \in \mathbb{N} \rbrace$, and let $\mathcal{F}_{\xi+1}^\xi$ denote the stopped $\sigma$-algebra. Given that $\lbrace y_k : k+1 \rbrace$ are well-defined (see \cref{eqn-common-y-def}), let $y_\xi$ be generated by an $N,\pi$-Exploratory, Markovian RBAS. If $y_\xi \neq 0$, then $\nu(\xi)$ is finite, and $\incond{\nu(\xi)}{\mathcal{F}_{\xi+1}^\xi} \leq (\rnk A - 1) ( N/\pi)$. 
\end{theorem}
\begin{proof}
We need only bound the upper bound in  \cref{lemma-nu-upper-bound}.
At any given $k \geq 0$, there are three possible cases, either (Case 1) $\chi_{\xi+k} = 0$; (Case 2) $\chi_{\xi+k} = 1$ and $y_{\xi+k+1} \not\in \linspan{ y_{\xi},\ldots,y_{\xi+k}}$; or (Case 3) $\chi_{\xi+k} = 1$ and $y_{\xi+k+1} \in \linspan{ y_{\xi},\ldots,y_{\xi+k}}$. We will show that Cases 1 and 2 cannot hold for all $k \geq 0$ with probability one.

To this end, define $s(j) = \min\lbrace k\geq 0: \chi_{j+k} \neq 0 \rbrace$ and let $s_1 = s(\xi)$ and $s_{j+1} = s(\xi + s_1 + \cdots + s_j)$ for all $j \in \mathbb{N}$. With this notation, the Markovian property and the $N,\pi$-exploratory property,
\begin{align}
&\condPrb{ s_1 \geq N }{\mathcal{F}_{\xi+1}^\xi} \nonumber \\
&= \condPrb{ \bigcap_{j=0}^{N-1} \lbrace \chi_{\xi+j} = 0 \rbrace }{\mathcal{F}_{\xi+1}^\xi }
= \condPrb{ \bigcap_{j=0}^{N-1} \lbrace \col (\Prj_{\xi+j}) \perp y_{\xi} \rbrace }{\mathcal{F}_{\xi+1}^\xi} \\
&= \condPrb{ \bigcap_{j=0}^{N-1} \lbrace \col (\Prj_{\xi+j}) \perp y_{\xi} \rbrace }{\mathcal{F}_{1}^\xi} 
\leq 1 - \pi,
\end{align}
where the last line is a consequence of \cref{remark-markovian}.
Now, using induction and the Markovian property, $\incondPrb{ s_1 \geq N \ell }{\mathcal{F}_{\xi+1}^{\xi}} \leq (1 - \pi)^\ell$ for all $\ell \in \mathbb{N}$. Therefore, $s_1$ is finite with probability one and $\incond{s_1}{\mathcal{F}_{\xi+1}^\xi} \leq N/\pi$. Moreover, since $\xi$ is an arbitrary stopping time, it follows that $\lbrace s_j \rbrace$ are finite with probability one and $\incond{ s_j }{\mathcal{F}_{\xi+1}^\xi} \leq N/\pi$. Thus, Case 1 cannot occur for all $k \geq 0$, and Cases 2 or 3 must occur infinitely often.

Now, the dimension of $\linspan{ y_{\xi},\ldots, y_{\xi+s_1+\cdots+s_j}}$ is $j + 1$. Since $\lbrace y_k \rbrace$ are either in $\row(A)$ or $\col(A)$, $j+1 \leq \rnk A$. Thus, Case 2 cannot be the only situation to occur when $\chi_{\xi+k} \neq 0$. In conclusion, the largest value of $j$ is $\rnk A-1$, which implies $\nu(\xi) \leq s_1 + \cdots + s_{\rnk A - 1}$, which are the sum of exponentially distributed random variables. The result follows by using $\incond{s_j}{\mathcal{F}_{\xi+1}^\xi} \leq N/ \pi$.
\end{proof}

Now, putting together \cref{theorem-block-meany,theorem-nu-finite} supplies the following result.
\begin{corollary} \label{corollary-convergence}
Suppose $A \in \mathbb{R}^{n \times d}$ and $b \in \mathbb{R}^n$. Given that $\lbrace y_k : k+1 \rbrace$ are well-defined (see \cref{eqn-common-y-def}), suppose that they are generated by a Markovian, $N,\pi$-exploratory RBAS. Then, one of the following to cases occurs.
\begin{remunerate}
\item There exist a stopping time $\tau$ with finite expectation such that $y_{\tau} = 0$.
\item There exist stopping times $\lbrace \tau_j : j+1 \in \mathbb{N} \rbrace$ such that $\E{ \tau_j} \leq j [ (\rnk{A} - 1) (N/\pi) + 1 ]$ for all $j+1 \in \mathbb{N}$, and 
$
\inPrb{ \cap_{j=1}^\infty \lbrace \innorm{y_{\tau_j}}_2^2 \leq ( \prod_{\ell=0}^{j-1} \gamma_{\ell}  ) \innorm{ y_0}_2^2 \rbrace } = 1,
$
where $
\gamma_{\ell} = 1 - \sup_{Q_i \in \mathfrak{Q}_i, i \in \lbrace \tau_\ell,\ldots,\tau_{\ell} + \nu(\tau_{\ell}) \rbrace}\min_{ G \in \mathcal{G}(Q_{\tau_{\ell}},\ldots,Q_{\tau_{\ell}+\nu(\tau_{\ell})})} \det(G^\intercal G) \in (0,1).
$
\end{remunerate}
\end{corollary}
\begin{proof}
The proof proceeds by induction. For $j=0$, recall $\tau_0 = 0$. Now, either $y_{\tau_0} = 0$ or $y_{\tau_0} \neq 0$. In the former case, the statement of the result is true. In the latter case, define $\tau_1 = \nu(\tau_0) + 1$. Then, $\tau_1$ is finite with probability one and $\E{ \tau_1} \leq (\rnk{A} - 1) N/\pi + 1$ by \cref{theorem-nu-finite}. Moreover, by \cref{theorem-block-meany}, $\innorm{ y_{\tau_1}}_2^2 \leq \gamma_0 \innorm{ y_{\tau_0}}_2^2$. Thus, we have established the base case.

For the induction hypothesis, suppose that for $j \in \mathbb{N}$, $A x_{\tau_{j-1}} \neq b$, $\E{ \tau_{k}} \leq k [ (\rnk{A} - 1) N/\pi + 1]$ for $k \in [0,j-1]\cap\mathbb{N}$, and $\innorm{ y_{\tau_k}}_2^2 \leq \innorm{y_0}_2^2\prod_{\ell=0}^{k-1} \gamma_\ell$ for $k \in [1,j-1] \cap \mathbb{N}$. 

To conclude, define $\tau_j = \tau_{j-1} + \nu(\tau_{j-1}) + 1$. By \cref{theorem-nu-finite}, $\tau_{j}$ is finite and $\E{\tau_{j}} \leq (j-1) [ (\rnk{A} - 1) N/\pi + 1] + (\rnk{A} - 1) N/\pi + 1 = j[ (\rnk{A} - 1) N/\pi + 1]$. Finally, either $y_{\tau_{j}} = 0$ or $y_{\tau_{j}} \neq 0$. In the latter case, \cref{theorem-block-meany} implies $\innorm{ y_{\tau_{j}}}_2^2 \leq \gamma_{j-1} \innorm{ y_{\tau_{j-1}}}_2^2$. The result follows.
\end{proof}

Our final task is to control the joint behavior of $\lbrace \gamma_{\ell}: \ell +1 \in \mathbb{N} \rbrace \subset (0,1)$ in the latter case of \cref{corollary-convergence}. Depending on our goal, we could require two different types of control. For instance, to ensure convergence of $\lbrace y_k \rbrace$ to $0$, we need to ensure that $\liminf_{\ell \to \infty} \gamma_{\ell} < 1$. However, for a rate of convergence, we need to ensure that $\limsup_{\ell \to \infty} \gamma_{\ell} < 1$. As the latter case is more desirable in practice, we will focus on ensuring that $\limsup_{\ell \to \infty} \gamma_{\ell} < 1$. This will give rise to two separate cases in our theory of convergence of RBAS methods, which we now address one at a time.

\subsection{Convergence for a Finite Set} \label{subsection-finiteset}

In the first case, we have that $\lbrace \col( \Prj_k ) \rbrace$ take value in finite sets, as in \cref{example-cyclic-vector-kaczmarz,example-random-permute-kaczmarz,example-greedy-block-kaczmarz,example-cyclic-cd,example-cyclic-block-cd}.

\begin{theorem} \label{theorem-convergence-common-finite}
Let $A \in \mathbb{R}^{n \times d}$ and $b \in \mathbb{R}^n$. Given that $\lbrace y_k : k+1 \rbrace$ are well-defined (see \cref{eqn-common-y-def}), suppose $\lbrace y_{k} : k + 1 \in \mathbb{N} \rbrace$ are generated by a Markovian, $N,\pi$-exploratory RBAS. If the elements of $\lbrace \col(\Prj_k) : k+1 \in \mathbb{N} \rbrace$ take value in a finite set, then either
\begin{remunerate}
\item There exist a stopping time $\tau$ with finite expectation such that $y_{\tau} = 0$.
\item There exist stopping times $\lbrace \tau_j : j+1 \in \mathbb{N} \rbrace$ such that $\E{ \tau_j} \leq j [ (\rnk{A} - 1) (N/\pi) + 1 ]$ for all $j+1 \in \mathbb{N}$, and there exist $\gamma \in (0,1)$ and a sequence of random variables $\lbrace \gamma_j : j+1 \in \mathbb{N}\rbrace \subset (0,\gamma]$, such that
$
\inPrb{ \cap_{j=1}^\infty \lbrace \innorm{y_{\tau_j}}_2^2 \leq ( \prod_{\ell=0}^{j-1} \gamma_{\ell}  ) \innorm{ y_0}_2^2 \rbrace } = 1.
$
\end{remunerate}
\end{theorem}
\begin{proof}
By \cref{corollary-convergence}, we can focus on the second case and we need only show that there exists a $\gamma \in (0,1)$ such that $\gamma_{\ell} \leq \gamma$. To this end, let $\lbrace \mathcal{U}_i : i = 1,\ldots,r \rbrace$ denote the set of linear spaces in which $\lbrace \col( \Prj_k) \rbrace$ takes value. For each $\mathcal{U}_i$, we can define the set of all orthonormal bases of $\mathcal{U}_i$, denoted $\mathfrak{U}_i$. Let $\mathfrak{P}$ denote the power set of $\lbrace \mathfrak{U}_i : i = 1,\ldots, r \rbrace$. For a given element $\lbrace \mathfrak{U}_{i_1},\ldots,\mathfrak{U}_{i_s} \rbrace \in \mathfrak{P}$, we can choose a set $\lbrace \cup_{j=1}^s U_j : U_j \in \mathfrak{U}_{i_j} \rbrace$, and let $\mathcal{H}$ denote the set of all matrices whose columns are maximal linearly independent subsets of $\lbrace \cup_{j=1}^s U_j : U_j \in \mathfrak{U}_{i_j} \rbrace$. Finally, define
\begin{equation} \label{eqn-Gamma}
\Gamma = \left\lbrace 1 - \sup_{ \lbrace \cup_{j=1}^s U_j : U_j \in \mathfrak{U}_{i_j} \rbrace } \min_{ H \in \mathcal{H} } \det(H^\intercal H) :  \lbrace \mathfrak{U}_{i_1},\ldots,\mathfrak{U}_{i_s} \rbrace \in \mathfrak{P},~s=1,\ldots,r\right\rbrace.
\end{equation}

Now, since $\lbrace \mathfrak{Q}_k \rbrace$ takes value in $\lbrace \mathfrak{U}_i : i = 1,\ldots, r \rbrace$,
$\lbrace \mathfrak{Q}_{i} : i =\tau_{\ell},\ldots, \tau_{\ell} + \nu(\tau_{\ell}) \rbrace \in \mathfrak{P}$ for all $\ell + 1 \in \mathbb{N}$. Therefore, $\gamma_{\ell} \in \Gamma$ for all $\ell + 1 \in \mathbb{N}$. Thus, $\gamma_{\ell} \leq \max \lbrace \Gamma \rbrace =: \gamma$. By Hadamard's inequality, each element of $\Gamma$ is in $[0,1)$, which implies that $\gamma \in [0,1)$. As we are only proving the second case, $\gamma \neq 0$ (else we would have converged finitely and would be in the first case), which implies $\gamma \in (0,1)$.
\end{proof}

We make two remarks. First, by substituting in the appropriate definitions of $\lbrace y_k : k+1 \in \mathbb{N} \rbrace$ and $\lbrace \Prj_k : k + 1 \in \mathbb{N} \rbrace$ into \cref{theorem-convergence-common-finite}, then we have proven \cref{theorem-convergence-row-finite,theorem-convergence-col-finite}. Second, the value of $\gamma$ can be vary depending on how the set to which $\lbrace \col( \Prj_k ) \rbrace$ belongs is designed, which was a central point of discussion in \cite{needell2014}. For instance, consider the three unique partitions of the rows of the coefficient matrix 
\begin{equation} \label{eqn-gamma-partitioning-example}
\begin{bmatrix}
1 & -1 & 1\\
1 & -1 & 1+10^{-5}\\
3 & -1 & 3 \\
0 & 1 & 6
\end{bmatrix}
\end{equation}
such that each partition contains two rows. Now, consider a sampling scheme that selects a partition and cycles through the blocks in this partition. For such a method, we compute can $\gamma$. The results for each of the three partitions are in presented in \cref{table-gamma-partitioning-example}.

\begin{table}[htb]
\caption{Estimates of the values of $\gamma$ in \cref{theorem-convergence-common-finite} for the three unique equally-sized partitions of \cref{eqn-gamma-partitioning-example}.}
\label{table-gamma-partitioning-example}
\centering
\begin{tabular}{@{}ccc@{}} \toprule
Partition I & Partition II & Partition III \\
$0.880$ & $0.372$ & $0.372$ \\\bottomrule
\end{tabular}
\end{table}

From \cref{table-gamma-partitioning-example}, we see that to get the same guaranteed relative reduction in error from Partition I in comparison to Partition II or III requires over seven fold more iterations. Indeed, as shown in \cref{figure-partition-comparison}, we observe exactly this behavior when we implement cyclic block Kaczmarz on \cref{eqn-gamma-partitioning-example} for the three different partitions up to an absolute error of $10^{-4}$.

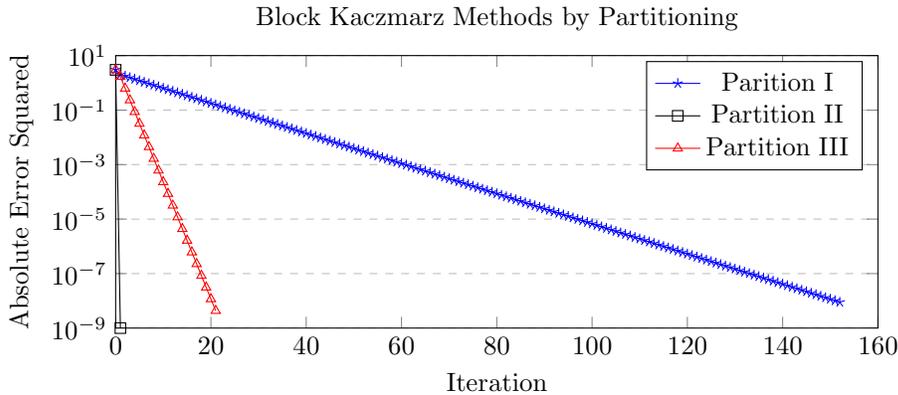
\begin{figure}[htb]
\centering
\begin{tikzpicture}
\begin{axis}[
	width=0.9\textwidth,
	height=0.4\textwidth,
    title={Block Kaczmarz Methods by Partitioning},
    xlabel={Iteration},
    ylabel={Absolute Error Squared},
    ymode=log,
    xmin=0, xmax=160,
    ymin=1e-9, ymax=10,
    xtick={0,20,40,60,80,100,120,140,160},
    ytick={1e-9,1e-7,1e-5,1e-3,1e-1,10.0},
    ymajorgrids=true,
    grid style=dashed,
]

\addplot[
    color=blue,
    mark=star,
    ]
    coordinates {
(0, 2.9999999999999996)
(1, 2.0000000000077796)
(2, 1.7608700108021251)
(3, 1.5503312051628155)
(4, 1.364965555966869)
(5, 1.2017631525361052)
(6, 1.0580741096985553)
(7, 0.9315652487563815)
(8, 0.8201824703388301)
(9, 0.7221171749722655)
(10, 0.6357770958391792)
(11, 0.5597602691627822)
(12, 0.49283242474989425)
(13, 0.4339068087472103)
(14, 0.3820266575751276)
(15, 0.33634955721290044)
(16, 0.29613385761028194)
(17, 0.26072654854819555)
(18, 0.2295527285453575)
(19, 0.2021062066540745)
(20, 0.1779413391226859)
(21, 0.15666574422758972)
(22, 0.13793397433926438)
(23, 0.12144186871174951)
(24, 0.10692164828589709)
(25, 0.09413753816476177)
(26, 0.08288196527961905)
(27, 0.07297216508314633)
(28, 0.06424723410749723)
(29, 0.056565499594647045)
(30, 0.04980223473927609)
(31, 0.04384761971610389)
(32, 0.03860497061720917)
(33, 0.033989158912979274)
(34, 0.029925238579838167)
(35, 0.026347220923554435)
(36, 0.02319701037830528)
(37, 0.02042345478959584)
(38, 0.01798152048355056)
(39, 0.015831556077909394)
(40, 0.013938653026237683)
(41, 0.01227207494701419)
(42, 0.01080476194189404)
(43, 0.009512888231450628)
(44, 0.008375477917474086)
(45, 0.00737406207951557)
(46, 0.006492380926356895)
(47, 0.005716117989510153)
(48, 0.005032669240944132)
(49, 0.004430937049092336)
(50, 0.0039011512070496173)
(51, 0.003434709214902548)
(52, 0.003024037546068682)
(53, 0.0026624678394736433)
(54, 0.002344129359232517)
(55, 0.002063853022802642)
(56, 0.0018170880385307456)
(57, 0.0015998275121847563)
(58, 0.001408543827667019)
(59, 0.0012401309787068934)
(60, 0.0010918544792910317)
(61, 0.0009613066611149737)
(62, 0.0008463678449750614)
(63, 0.0007451716895976415)
(64, 0.0006560750930215202)
(65, 0.0005776313318994075)
(66, 0.0005085667304287221)
(67, 0.000447759838747049)
(68, 0.00039422334733882586)
(69, 0.0003470879471135485)
(70, 0.00030558830986863653)
(71, 0.0002690505771669657)
(72, 0.00023688149309681776)
(73, 0.00020855870587873146)
(74, 0.0001836223440420081)
(75, 0.00016166749855873292)
(76, 0.00014233769296336502)
(77, 0.0001253190557612298)
(78, 0.00011033525867823994)
(79, 9.71429994884553e-5)
(80, 8.552807806738423e-5)
(81, 7.530189482020075e-5)
(82, 6.629840923220554e-5)
(83, 5.837142551966245e-5)
(84, 5.139223479621074e-5)
(85, 4.5247511070579095e-5)
(86, 3.983748368453293e-5)
(87, 3.507430628747079e-5)
(88, 3.0880640086445906e-5)
(89, 2.7188389641328735e-5)
(90, 2.393760457347447e-5)
(91, 2.1075499678820826e-5)
(92, 1.855560348106343e-5)
(93, 1.6336998717024097e-5)
(94, 1.4383662308406407e-5)
(95, 1.2663876597619337e-5)
(96, 1.114971775739137e-5)
(97, 9.81659932987761e-6)
(98, 8.64287573388194e-6)
(99, 7.609488417874733e-6)
(100, 6.699658464908077e-6)
(101, 5.898612344104124e-6)
(102, 5.193343626464993e-6)
(103, 4.572400366779152e-6)
(104, 4.025700438167085e-6)
(105, 3.544366690125553e-6)
(106, 3.1205838014543246e-6)
(107, 2.747470520845825e-6)
(108, 2.418968679590948e-6)
(109, 2.1297441635530087e-6)
(110, 1.875100890166139e-6)
(111, 1.6509040446028859e-6)
(112, 1.4535133859802474e-6)
(113, 1.2797237420044712e-6)
(114, 1.1267133252491635e-6)
(115, 9.919976015780499e-7)
(116, 8.733892177948333e-7)
(117, 7.689622461019122e-7)
(118, 6.770211264957941e-7)
(119, 5.960729483278953e-7)
(120, 5.248033722467808e-7)
(121, 4.6205514295651164e-7)
(122, 4.0680943164637775e-7)
(123, 3.5816917351472334e-7)
(124, 3.1534460764049286e-7)
(125, 2.776403610748286e-7)
(126, 2.444442371606273e-7)
(127, 2.1521720880451968e-7)
(128, 1.8948472166950808e-7)
(129, 1.6682893973075717e-7)
(130, 1.4688200589654064e-7)
(131, 1.293200269306496e-7)
(132, 1.1385785257191149e-7)
(133, 1.0024441367742569e-7)
(134, 8.825867124809074e-8)
(135, 7.770600403364399e-8)
(136, 6.841506965095561e-8)
(137, 6.023500697526978e-8)
(138, 5.303299696339795e-8)
(139, 4.669209515255224e-8)
(140, 4.110934608455725e-8)
(141, 3.6194098183148215e-8)
(142, 3.186654396593288e-8)
(143, 2.805641370916217e-8)
(144, 2.470184384671972e-8)
(145, 2.1748362517219027e-8)
(146, 1.914801555685685e-8)
(147, 1.685857891418374e-8)
(148, 1.4842879665714554e-8)
(149, 1.3068187531762439e-8)
(150, 1.1505687165772206e-8)
(151, 1.013000717856401e-8)
(152, 8.91881101842811e-9)
}; \addlegendentry{Parition I}

\addplot[
    color=black,
    mark=square,
    ]
    coordinates {
 (0, 2.9999999999999996)
 (1, 1e-9)
    };\addlegendentry{Partition II}
    
\addplot[
    color=red,
    mark=triangle,
    ]
    coordinates {
(0, 2.9999999999999996)
(1, 1.674418604651163)
(2, 0.6230231260065191)
(3, 0.2318164731690852)
(4, 0.08625502808701699)
(5, 0.032094051680552585)
(6, 0.011941659241416746)
(7, 0.004443291450313339)
(8, 0.0016532743493428794)
(9, 0.0006151557026497492)
(10, 0.00022888913667167278)
(11, 8.516581519221363e-5)
(12, 3.168877380038467e-5)
(13, 1.1790862128256764e-5)
(14, 4.3871823694817136e-6)
(15, 1.6323970998661435e-6)
(16, 6.073876276920899e-7)
(17, 2.2599876605023497e-7)
(18, 8.409035668086818e-8)
(19, 3.128861369596562e-8)
(20, 1.1641969253773942e-8)
(21, 4.3317818239264195e-9)
    };\addlegendentry{Partition III}
\end{axis}
\end{tikzpicture}
\caption{A comparison of the aboslute errors of cyclic block Kaczmarz methods for the coefficient matrix in \cref{eqn-gamma-partitioning-example}. Corresponding to the theory, Partition I produces the worst convergence rate.}
\label{figure-partition-comparison}
\end{figure}
\subsection{Convergence for an Infinite Set} \label{subsection-infiniteset}

In the second case, $\lbrace \col(\Prj_k) \rbrace$ can take value over an infinite set, as in \cref{example-gaussian-column} with $n_k < \rnk A$. Suppose we attempt to prove the convergence result as we did in \cref{subsection-finiteset}. Then, we would need to prove that $\sup \lbrace \Gamma \rbrace < 1$. However, when $\Gamma$ is infinite, we could potentially have $\sup \lbrace \Gamma \rbrace = 1$. For instance, consider a $3 \times 2$ coefficient matrix whose first two rows are the first standard basis element of $\mathbb{R}^2$ and the last row is the second standard basis element, 
and a procedure that alternates between choosing either the first row of the matrix, or taking a linear combination of the second row and a product of a standard Guassian random variable with the third row. If we let $\mathcal{N}(0,1)$ denote a standard Gaussian distribution, then $\Gamma$ is made up of all possible values of $(Z^2 + 1)^{-1}$ with $Z \sim \mathcal{N}(0,1)$. 
Since $Z$ has nonzero density about $0$, the supremum of $\Gamma$ would be $1$ in this case.

As this example suggests, it is possible to have arbitrarily poor values for $\lbrace \gamma_{\ell} : \ell +1 \in \mathbb{N} \rbrace$. However, this example also shows that shows that bulk of values of $\lbrace \gamma_{\ell} : \ell + 1 \in \mathbb{N} \rbrace$ are well-behaved (i.e., the mean and standard deviation of $(Z^2+1)^{-1}$ is approximately $0.66$ and $0.26$, respectively), which partially motivates \cref{def-unif-control}. 
Under \cref{def-unif-control}, we have the following result, from which \cref{theorem-convergence-row-infinite,theorem-convergence-col-infinite} follow immediately.
\begin{theorem} \label{theorem-convergence-infinite}
Let $A \in \mathbb{R}^{n \times d}$ and $b \in \mathbb{R}^n$. Given that $\lbrace y_k : k+1 \rbrace$ are well-defined (see \cref{eqn-common-y-def}), suppose $\lbrace y_{k} : k + 1 \in \mathbb{N} \rbrace$ are generated by a Markovian, $N,\pi$-exploratory, and uniformly nontrivial RBAS. One of the following is true.
\begin{remunerate}
\item There exist a stopping time $\tau$ with finite expectation such that $y_{\tau} = 0$.
\item There exists a sequence of non-negative stopping times $\lbrace \tau_{j} : j +1 \in \mathbb{N} \rbrace$ for which $\E{ \tau_{j} } \leq j [ (\rnk{A} - 1) (N /\pi) + 1]$, there exists $\bar \gamma \in (0,1)$, and there exists a sequence of random variables $\lbrace \gamma_j : j+1 \in \mathbb{N}\rbrace \subset (0,1)$, such that
$
\inPrb{ \cap_{j=0}^\infty \lbrace \innorm{ y_{\tau_j} }_2^2 \leq ( \prod_{\ell=0}^{j-1} \gamma_{\ell} ) \innorm{ y_0 }_2^2 \rbrace} = 1,
$
where for any $\gamma \in (\bar \gamma,1)$,  $\inPrb{\cup_{L=0}^\infty \cap_{j=L}^\infty \lbrace \prod_{\ell=0}^{j-1} \gamma_{\ell} \leq \gamma^j \rbrace} = 1$.
\end{remunerate}
\end{theorem} 
\begin{proof}
By \cref{corollary-convergence}, we can focus on the second case and we need only prove that there exists $\bar \gamma \in (0,1)$ such that for any $\gamma \in (\bar \gamma,1)$, $\inPrb{\cup_{L=0}^\infty \cap_{j=L}^\infty \lbrace \prod_{\ell=0}^{j-1} \gamma_{\ell} \leq \gamma^j \rbrace} = 1$. To show this, we need to prove $\incond{ \prod_{\ell=0}^{j-1} \gamma_{\ell} }{\mathcal{F}_1^0} \leq \bar \gamma^j$ for each $j$, which we will do by induction. For the base case, $j=0$,
\begin{align}
&\cond{ 1- \gamma_0 }{\mathcal{F}_1^0} \nonumber\\
&= \cond{ \sup_{Q_i \in \mathfrak{Q}_i, i \in \lbrace \tau_0,\ldots,\tau_0+ \nu(\tau_{0}) \rbrace}\min_{ G \in \mathcal{G}(Q_{\tau_{0}},\ldots,Q_{\tau_{0}+\nu(\tau_{0})})} \det(G^\intercal G) }{\mathcal{F}_1^0} \\
&= \cond{ \sum_{k=0}^\infty \1{ \nu(\tau_0) = k} \sup_{Q_i \in \mathfrak{Q}_i, i \in \lbrace \tau_0,\ldots,\tau_0+ k \rbrace}\min_{ G \in \mathcal{G}(Q_{\tau_{0}},\ldots,Q_{\tau_{0}+k})} \det(G^\intercal G) }{\mathcal{F}_1^0}.
\end{align}
Since $\min_{ G \in \mathcal{G}(Q_{\tau_0},\ldots, Q_{\tau_0+k})} \det(G^\intercal G) \geq \min_{ G \in \mathcal{G}(Q_{\tau_0},\ldots, Q_{\tau_0 + k},Q)} \det(G^\intercal G)$ for any $Q \in \mathfrak{Q}_{\tau_0 + k + 1}$ and any $k + 1 \in \mathbb{N}$, then, for every $k+1 \in \mathbb{N}$,
$
\incond{ 1- \gamma_0}{\mathcal{F}_1^0}$ is bounded below by 
\begin{equation}
\cond{  \1{ \nu(\tau_0) \leq k} \sup_{Q_i \in \mathfrak{Q}_i, i \in \lbrace \tau_0,\ldots,\tau_0+ k \rbrace}\min_{ G \in \mathcal{G}(Q_{\tau_{0}},\ldots,Q_{\tau_{0}+k})} \det(G^\intercal G) }{\mathcal{F}_1^0}.
\end{equation}
Now, by \cref{theorem-nu-finite} and Markov's inequality, for any $y_0 \neq 0$ and any $\zeta_{-1} \in \mathfrak{Z}$,
$
\incondPrb{ \nu(\tau_0) \leq k }{\mathcal{F}_{1}^0} \geq 1 - N(\rnk{A} - 1)/(k \pi).
$
Hence, we can apply \cref{def-unif-control} to conclude that there exists a $g \in (0,1]$ such that $\incond{ 1 - \gamma_0}{\mathcal{F}_1^0} \geq g$. If we let $\bar \gamma = 1 - g$, then $\incond{ \gamma_0}{\mathcal{F}_1^0} \leq \bar \gamma$. Now, for the induction hypothesis, suppose that $\incond{ \prod_{\ell=0}^{j-2} \gamma_{\ell}}{\mathcal{F}_1^0} \leq \bar \gamma^{j-1}$. To conclude, we note that by the Markovian property and the base case, $\incond{ \gamma_{j-1}}{\mathcal{F}_{\tau_j+1}^{\tau_j}} \leq \bar{\gamma}$. Therefore, 
$
\incond{ \prod_{\ell=0}^{j-1} \gamma_{\ell}}{\mathcal{F}_1^0} = \incond{ \incond{ \gamma_{j-1}}{\mathcal{F}_{\tau_j+1}^{\tau_j}} \prod_{\ell=0}^{j-1} \gamma_{\ell}}{\mathcal{F}_1^0} \leq \bar \gamma^{j}
$.

Now, for any $\gamma \in (\bar \gamma, 1)$, the preceding proof and Markov's inequality provide
\begin{equation}
\sum_{j=1}^\infty \condPrb{ \prod_{\ell=0}^{j-1} \gamma_{\ell} > \gamma^{j} }{\mathcal{F}_1^0} \leq \sum_{j=1}^\infty \left( \frac{\bar \gamma}{\gamma} \right)^j < \infty.
\end{equation}
By the Borel-Cantelli lemma, $\inPrb{\cup_{L=0}^\infty \cap_{j=L}^\infty \lbrace \prod_{\ell=0}^{j-1} \gamma_{\ell} \leq \gamma^j \rbrace} = 1$.
\end{proof}
\begin{remark}
Our proof readily allows us to bound the convergence rates of the moments of $\lbrace y_{k}: k+1 \rbrace$.
\end{remark}

\section{Examples} \label{section-examples}
We provide a series of examples to demonstrate how we can apply our theory to a variety of methods. 
Of particular practical value, we will show how to verify the relevant properties (e.g., Exploratory). We summarize these examples, references, and reference the convergence result for the given method based on our theory in \cref{table-examples}.

\begin{small}
\begin{longtable}{@{}>{\centering}m{1.5in}ccc@{}}
\caption{Summary of worked examples using our theory.} \label{table-examples} \\ 
\toprule
\textbf{Method} & \textbf{References} & \textbf{Details} & \textbf{Convergence Result} \\ \midrule
Cyclic Vector Kaczmarz  & \cite{karczmarz1937,kaczmarz1993,dai2015} & \Cref{subsection-cyclic-vector-kaczmarz} & \cref{theorem-cyclic-vector-kaczmarz}\\ 
\hdashline
Gaussian Vector Kaczmarz & \cite{gower2015,richtarik2020} & \Cref{subsection-gaussian-vector-kaczmarz} & \cref{theorem-gaussian-vector-kaczmarz}\\
\hdashline
Strohmer-Vershynin Vector Kaczmarz & \cite{strohmer2009} & \Cref{subsection-strohmer-vershynin-vector-kaczmarz} & \cref{theorem-strohmer-verhsynin-vector-kaczmarz} \\
\hdashline
Steinerberger Vector Kaczmarz & \cite{steinerberger2021} & \Cref{subsection-steinerberger-vector-kaczmarz} & \cref{theorem-steinerberger-vector-kaczmarz} \\
\hdashline
Motzkin's Method & \cite{motzkin1954,agmon1954}  & \Cref{subsection-motzkin-vector} & \cref{theorem-motzkin-vector}\\
\hdashline
Agmon's Method & \cite{agmon1954} & \Cref{subsection-agmon-vector} & \cref{theorem-agmon-vector} \\
\hdashline
Greedy Randomized Vector Kaczmarz & \cite{bai2018} & \Cref{subsection-greedy-randomized-vector-kaczmarz} & \cref{theorem-greedy-randomized-vector-kaczmarz}\\
\hdashline
Sampling Kaczmarz-Motzkin Method & \cite{haddock2019} & \Cref{subsection-sampling-kaczmarz-motzkin-vector} & \cref{theorem-sampling-kaczmarz-motzkin-vector}\\
\hdashline
Streaming Vector Kaczmarz & \cite{PAJAMA2021adaptive} & \Cref{subsection-streaming-vector-kaczmarz} & \cref{theorem-streaming-vector-kaczmarz} \\
\hdashline
Cyclic Vector Coordinate Descent & \cite{warga1963} & \Cref{subsection-cyclic-vector-cd} &  \cref{theorem-cyclic-vector-cd}\\
\hdashline
Gaussian Vector Column Space Descent & & \Cref{subsection-gaussian-vector-cs} & \cref{theorem-gaussian-vector-cs} \\
\hdashline
Zouzias-Freris Vector Coordinate Descent & \cite{zouzias2013} & \Cref{subsection-zouzias-freris-vector-cd} & \cref{theorem-zouzias-freris-vector-cd}\\
\hdashline
Max Residual Coordinate Descent & & \Cref{subsection-max-residual-vector-cd} & \cref{theorem-max-residual-vector-cd} \\
\hdashline
Max Distance Coordinate Descent & & \Cref{subsection-max-distance-vector-cd} & \cref{theorem-max-distance-vector-cd}\\
\hdashline
Random Permutation Block Kaczmarz & \cite{needell2014,necoara2019} &  \Cref{subsection-random-permutation-block-kaczmarz} & \cref{theorem-random-permutation-block-kaczmarz} \\
\hdashline
Steinerberger Block Kaczmarz &  \cite{richtarik2020,gower2021} & \Cref{subsection-steinerberger-block-kaczmarz} & \cref{theorem-steinerberger-block-kaczmarz} \\
\hdashline
Motzkin's Block Method &  & \Cref{subsection-motzkin-block} & \cref{theorem-motzkin-block}\\
\hdashline
Agmon's Block Method & & \Cref{subsection-agmon-block} & \cref{theorem-agmon-block} \\
\hdashline
Adaptive Sketch-and-Project & \cite{gower2021} & \Cref{subsection-adaptive-sketch-and-project} & \cref{theorem-adaptive-sketch-and-project} \\
\hdashline
Greedy Randomized Block Kaczmarz & & \Cref{subsection-greedy-randomized-block-kaczmarz} & \cref{theorem-greedy-randomized-block-kaczmarz} \\
\hdashline
Streaming Block Kaczmarz & \cite{gower2015,rebrova2021} & \Cref{subsection-streaming-block-kaczmarz} & \cref{theorem-streaming-block-kaczmarz} \\
\hdashline
Random Permutation Block Coordinate Descent & \cite{necoara2016,wright2020} & \Cref{subsection-random-permutation-block-cd} & \cref{theorem-random-permutation-block-cd} \\
\hdashline
Gaussian Block Column Space Descent & & \Cref{subsection-gaussian-block-cs} &\cref{theorem-gaussian-block-cs} \\
\hdashline
Zouzias-Freris Block Coordinate Descent & & \Cref{subsection-zouzias-freris-block-cd} & \cref{theorem-zouzias-freris-block-cd} \\
\hdashline
Max Residual Block Coordinate Descent & & \Cref{subsection-max-residual-block-cd} & \cref{theorem-max-residual-block-cd} \\
\hdashline
Max Distance Block Coordinate Descent & & \Cref{subsection-max-distance-block-cd} & \cref{theorem-max-distance-block-cd}\\
\bottomrule
\end{longtable}
\end{small}

\section{Conclusion} \label{section-conclusion}
In order to enable broader use of highly tailored randomized methods for solvign linear systems, we began with the challenge of providing a unifying theory for randomized block adaptive solvers (RBASs) for linear systems---regardless of whether the linear systems are underdetermined, overdetermined, or rank deficient. To this end, we studied two archetypes of RBAS solvers---row-action methods for consistent linear systems and column-action methods for arbitrary linear systems---, and showed that under very general conditions both archetypes will converge exponentially fast to a solution. Specifically, we had two results.
\begin{remunerate}
\item When a RBAS is Markovian, $N,\pi$-exploratory, and projects either the absolute error (for row-action methods) or residual (for column-action methods) onto only a finite number of spaces, then the RBAS will converge exponentially fast to a solution of the linear system.
\item When a RBAS is Markovian, $N,\pi$-exploratory, and uniformly nontrivial, then, after some finite number of iterations, the RBAS will converge exponentially fast to a solution of the linear system.
\end{remunerate} 

We further provided numerical evidence to elucidate key aspects of theory at key points. In particular, we demonstrated the value of the supremum in our generalization of Meany's inequality (see \cref{theorem-block-meany,figure-vector-vs-block}), and we discussed the importance of finding appropriate partitions when using block cyclic solvers (see \cref{figure-partition-comparison}), which was quite carefully studied in \cite{needell2014}. Finally, we provided a host of examples of how to apply our theory to existing methods and some novel methods, which we complemented with appropriate numerical experiments.

In completing the above tasks, we have provided practitioners with a powerful theory and demonstrations of how to use the theory to rigorously analyze a wide variety of RBASs. Thus, we hope that practitioners will be empowered to use this theory and create novel RBASs that are optimized to their specific applications and computing environments.

\pagebreak 
\appendix
\section{Worked Examples}
\numberwithin{equation}{subsection}
\numberwithin{theorem}{subsection}
\subsection{Cyclic Vector Kaczmarz} \label{subsection-cyclic-vector-kaczmarz}

For a description of this method, see \cref{example-cyclic-vector-kaczmarz}. We assume that $A$ and $b$ are arbitrary so long as they form a consistent system. Recall, $\zeta_{-1} = 0$,
\begin{equation}
\varphi_R( A, b, \lbrace x_j : j \leq k \rbrace, \lbrace W_j : j < k \rbrace, \lbrace \zeta_j : j < k \rbrace) = (e_{\mathrm{rem}(\zeta_{k-1}, n) + 1}, \zeta_{k-1}+1).
\end{equation}
and
\begin{equation}
x_{k+1} = x_k - A^\intercal e_{\mathrm{rem}(\zeta_{k-1},n) + 1} \frac{e_{\mathrm{rem}(\zeta_{k-1},n)+1}^\intercal(Ax_k - b)}{\norm{ A^\intercal e_{\mathrm{rem}(\zeta_{k-1},n) + 1} }_2^2},
\end{equation}
where $\lbrace e_{1},\ldots,e_{n} \rbrace$ are the standard basis elements of $\mathbb{R}^n$.

\begin{lemma}
Cyclic Vector Kaczmarz is Markovian.
\end{lemma}
\begin{proof}
Note, $W_k = e_{\mathrm{rem}(\zeta_{k-1},n),n) + 1}$ and $\zeta_{k} = \zeta_{k-1} + 1$. That is, $(W_k,\zeta_{k})$ are fully determined by $\zeta_{k-1}$. Therefore, $\incondPrb{ W_k \in \mathcal{W}, \zeta_{k} \in \mathcal{Z}}{\mathcal{F}_{k+1}^k} = \incondPrb{ W_k \in \mathcal{W}, \zeta_{k}}{\mathcal{F}_{1}^k}$.
\end{proof}

\begin{lemma}
Cyclic Vector Kaczmarz is $n,1$-Exploratory.
\end{lemma}
\begin{proof}
If $x_0$ is not a solution to $Ax=b$, then there exists a row of $A$, denoted by $A_{x_0} \in \mathbb{R}^d$, and corresponding constant, $b_{x_0}$, such that $A_{x_0}^\intercal x_0 \neq b_{x_0}$. Hence, 
\begin{equation}
\bigcap_{j=0}^{n-1} \left\lbrace e_{j+1}^\intercal A \perp x_0 - \Prj_{\mathcal{H}} x_0 \right\rbrace 
\subset \left\lbrace A_{x_0} \perp x_0 - \Prj_{\mathcal{H}} x_0 \right\rbrace 
= \left\lbrace A_{x_0}^\intercal x_0 - b_{x_0} = 0 \right\rbrace.
\end{equation}
Clearly, the last event must be empty. The conclusion follows.
\end{proof}

Since the set of $|\lbrace \col(A^\intercal e_{j}) : j=1,\ldots,n \rbrace| \leq n$, we can apply \cref{theorem-convergence-row-finite} to conclude as follows.

\begin{theorem} \label{theorem-cyclic-vector-kaczmarz}
Let $A \in \mathbb{R}^{n \times d}$ and $b \in \mathbb{R}^n$ such that the linear system's solution set, $\mathcal{H}$, is nonempty. Let $x_0 \in \mathbb{R}^d$. Let $\lbrace x_k : k \in \mathbb{N} \rbrace$ be a sequence generated by the Cyclic Vector Kaczmarz method. Then, there exists a stopping time $\tau$ with finite expectation such that $x_{\tau} \in \mathcal{H}$; or there exists a sequence of non-negative stopping times $\lbrace \tau_{j} : j +1 \in \mathbb{N} \rbrace$ for which $\E{ \tau_{j} } \leq j [ (\rnk{A} - 1) n + 1]$,  and there exist $\gamma \in (0,1)$ and a sequence of random variables $\lbrace \gamma_j : j+1 \in \mathbb{N}\rbrace \subset (0,\gamma]$, such that
\begin{equation}
\Prb{ \bigcap_{j=0}^\infty \left\lbrace \norm{ x_{\tau_j} - \Prj_{\mathcal{H}} x_0 }_2^2 \leq \left( \prod_{\ell=0}^{j-1} \gamma_{\ell} \right) \norm{ x_0 - \Prj_{\mathcal{H}}x_0 }_2^2 \right\rbrace} = 1.
\end{equation}
\end{theorem}
\subsection{Gaussian Vector Kaczmarz} \label{subsection-gaussian-vector-kaczmarz}

Suppose $(A,b)$ form a consistent linear system. In Gaussian Vector Kaczmarz, we generate independent standard normal vectors $\lbrace w_k: k+1 \in \mathbb{N} \rbrace \subset \mathbb{R}^n$ at each iteration and then apply
\begin{equation}
x_{k+1} = x_k - A^\intercal w_k \frac{w_k^\intercal (Ax_k - b)}{\norm{A^\intercal w_k}_2^2}.
\end{equation}
Therefore, we can define
\begin{equation}
\varphi_R( A, b, \lbrace x_j : j \leq k \rbrace, \lbrace W_j : j < k \rbrace, \lbrace \zeta_j : j < k \rbrace) = (w_k, \emptyset).
\end{equation}

\begin{lemma}
Gaussian Vector Kaczmarz is Markovian.
\end{lemma}
\begin{proof}
Since $w_k$ is independently generated at each iteration, $\incondPrb{ W_k \in \mathcal{W}, \zeta_{-1} \in \mathcal{Z}}{ \mathcal{F}_{k+1}^k} = \inPrb{ w_k \in \mathcal{W} } = \incondPrb{ W_k \in \mathcal{W}, \zeta_{-1} \in \mathcal{Z}}{\mathcal{F}_1^k}$.
\end{proof}

\begin{lemma}
Gaussian Vector Kaczmarz is $1,1$-Exploratory.
\end{lemma}
\begin{proof}
For any vector $v \neq 0$, $\inPrb{ w_0^\intercal v = 0 } = 0$ since $w_0$ is a continuous random variable. In particular, when $x_0 \neq \Prj_{\mathcal{H}}x_0$,
$\incondPrb{ w_0^\intercal A(x_0 - \Prj_{\mathcal{H}} x_0) = 0 }{\mathcal{F}_1^0} = 0.$
\end{proof}

\begin{lemma}
Gaussian Vector Kaczmarz is Uniformly Nontrivial.
\end{lemma}
\begin{proof}
Suppose $x_k \neq \Prj_{\mathcal{H}} x_0$. Then, $w_k^\intercal (Ax_k - b) \neq 0$ and $A^\intercal w_k \neq 0$ with probability one since $w_k$ is a continuous random variable. Hence, $\chi_k = 1$ with probability one so long as $x_k \neq \Prj_{\mathcal{H}}x_0$. We now have two cases: either $\rnk A = 1$ or $\rnk A > 1$. 

If $\rnk A = 1$, then $x_{0} - \Prj_{\mathcal{H}}x_0 \in \row( A)$. Hence, $\linspan{x_0 - \Prj_{\mathcal{H}} x_0}$ equals $\linspan{ A^\intercal w_0}$ with probability one. Therefore,
\begin{align}
x_1 - \Prj_\mathcal{H} x_0 = x_0 - \Prj_{\mathcal{H}}x_0 - \frac{(x_0 - \Prj_\mathcal{H} x_0)(x_0 - \Prj_\mathcal{H} x_0)^\intercal}{\norm{(x_0 - \Prj_\mathcal{H} x_0)}_2^2} (x_0 - \Prj_{\mathcal{H}} x_0) = 0.
\end{align}
Hence, $x_1$ solves the system, $\chi_0 = 1$ and $\chi_k = 0$ for all $k\in\mathbb{N}$. Thus, $\mathcal{G}(Q_0,\ldots,Q_s)$ equals $\mathcal{G}(Q_0)$, which is simply the $\mathbb{R}^{d \times 1}$ matrix whose column is a unit vector in $\row(A)$. To conclude this case, $\det(G^\intercal G) = 1$ and
\begin{equation}
\begin{aligned}
&\inf_{ \substack{x_0: Ax_0 \neq b \\ \zeta_{-1} \in \mathfrak{Z}} } \sup_{k \in \mathbb{N} \cup \lbrace 0 \rbrace } \cond{ \sup_{ \substack{ Q_s \in \mathfrak{Q}_s \\ s \in \lbrace 0,\ldots,k \rbrace}} \min_{ G \in \mathcal{G}(Q_0,\ldots,Q_k)} \det(G^\intercal G) \1{ \mathcal{A}_k(x_0, \zeta_{-1})}}{\mathcal{F}^0_1}\\
& = \inf_{ \substack{x_0: Ax_0 \neq b \\ \zeta_{-1} \in \mathfrak{Z}} } \sup_{k \in \mathbb{N} \cup \lbrace 0 \rbrace } \condPrb{ \mathcal{A}_k(x_0,\zeta_{-1})}{\mathcal{F}_1^0} = 1.
\end{aligned}
\end{equation}
In other words, when $\rnk A = 1$, Gaussian Vector Kaczmarz is uniformly nontrivial.

If $\rnk A > 1$, we will proceed by induction. Since $A^\intercal w_0$ is a continuous random variable with co-domain of dimension at least two, $x_0 - \Prj_{\mathcal{H}} x_0 \in \linspan{ A^\intercal w_0}$ with probability zero. Thus, $x_1 \neq \Prj_{\mathcal{H}} x_0$. Suppose that $x_{k-1} \neq \Prj_{\mathcal{H}} x_0$ with probability one. Then, by the same reasoning as the base case, $x_{k} \neq \Prj_{\mathcal{H}} x_0$. Hence, $\chi_k = 1$ for all $k+1 \in \mathbb{N}$ with probability one. It follows that
\begin{equation} \label{eqn-xl0z3d}
\sup_{ \substack{ Q_s \in \mathfrak{Q}_s \\ s \in \lbrace 0,\ldots,k \rbrace}} \min_{ G \in \mathcal{G}(Q_0,\ldots,Q_k)} \det(G^\intercal G)
\end{equation}
is independent of $x_0$. Moreover, \cref{eqn-xl0z3d} is positive with probability one. Hence, $\exists \lbrace \epsilon_k : k+1 \in \mathbb{N} \rbrace \in \mathbb{R}_{\geq 0}$ such that
\begin{equation}
\Prb{ \sup_{ \substack{ Q_s \in \mathfrak{Q}_s \\ s \in \lbrace 0,\ldots,k \rbrace}} \min_{ G \in \mathcal{G}(Q_0,\ldots,Q_k)} \det(G^\intercal G) > \epsilon_k} \geq 1/2.
\end{equation}
Therefore, by Markov's inequality,
\begin{equation}
\begin{aligned}
&\inf_{ \substack{x_0: Ax_0 \neq b \\ \zeta_{-1} \in \mathfrak{Z}} } \sup_{k \in \mathbb{N} \cup \lbrace 0 \rbrace } \cond{ \sup_{ \substack{ Q_s \in \mathfrak{Q}_s \\ s \in \lbrace 0,\ldots,k \rbrace}} \min_{ G \in \mathcal{G}(Q_0,\ldots,Q_k)} \det(G^\intercal G) \1{ \mathcal{A}_k(x_0, \zeta_{-1})}}{\mathcal{F}^0_1}\\
&\geq\inf_{ \substack{x_0: Ax_0 \neq b \\ \zeta_{-1} \in \mathfrak{Z}} } \sup_{k \in \mathbb{N} \cup \lbrace 0 \rbrace } \epsilon_k \condPrb{\sup_{ \substack{ Q_s \in \mathfrak{Q}_s \\ s \in \lbrace 0,\ldots,k \rbrace}} \min_{ G \in \mathcal{G}(Q_0,\ldots,Q_k)} \det(G^\intercal G) > \epsilon_k, \mathcal{A}_k(x_0, \zeta_{-1})}{\mathcal{F}^0_1}
\end{aligned}
\end{equation}
Now, $\exists K \in \mathbb{N}$ such that, for all $k \geq K$, $\incondPrb{ \mathcal{A}_k(x_0,\zeta_{-1})}{\mathcal{F}^0_1} \geq 3/4$. By the inclusion-exclusion principle, for $k \geq K$, 
\begin{equation}
\condPrb{\sup_{ \substack{ Q_s \in \mathfrak{Q}_s \\ s \in \lbrace 0,\ldots,k \rbrace}} \min_{ G \in \mathcal{G}(Q_0,\ldots,Q_k)} \det(G^\intercal G) > \epsilon_k, \mathcal{A}_k(x_0, \zeta_{-1})}{\mathcal{F}^0_1} \geq 1/4.
\end{equation}
Therefore, for all $k \geq K$, 
\begin{equation}
\inf_{ \substack{x_0: Ax_0 \neq b \\ \zeta_{-1} \in \mathfrak{Z}} } \sup_{k \in \mathbb{N} \cup \lbrace 0 \rbrace } \cond{ \sup_{ \substack{ Q_s \in \mathfrak{Q}_s \\ s \in \lbrace 0,\ldots,k \rbrace}} \min_{ G \in \mathcal{G}(Q_0,\ldots,Q_k)} \det(G^\intercal G) \1{ \mathcal{A}_k(x_0, \zeta_{-1})}}{\mathcal{F}^0_1} \geq \epsilon_k / 4.
\end{equation}
The result follows.
\end{proof}

We can now apply \cref{theorem-convergence-row-infinite} to conclude as follows

\begin{theorem} \label{theorem-gaussian-vector-kaczmarz}
Let $A \in \mathbb{R}^{n \times d}$ and $b \in \mathbb{R}^n$ such that the linear system's solution set, $\mathcal{H}$, is nonempty. Let $x_0 \in \mathbb{R}^d$. Let $\lbrace x_k : k \in \mathbb{N} \rbrace$ be a sequence generated by Gaussian Vector Kaczmarz. Then, there exists a stopping time $\tau$ with finite expectation such that $x_{\tau} \in \mathcal{H}$; or there exists a sequence of non-negative stopping times $\lbrace \tau_{j} : j +1 \in \mathbb{N} \rbrace$ for which $\E{ \tau_{j} } \leq j \rnk{A}$, and there exists $\bar \gamma \in (0,1)$ such that for any $\gamma \in (\bar \gamma, 1)$,
\begin{equation}
\Prb{ \bigcup_{L=0}^\infty \bigcap_{j=L}^\infty \left\lbrace \norm{ x_{\tau_j} - \Prj_{\mathcal{H}}x_0 }_2^2 \leq \gamma^j \norm{ x_0 - \Prj_{\mathcal{H}} x_0 }_2^2 \right\rbrace} = 1.
\end{equation}
\end{theorem}
\subsection{Strohmer-Vershynin Vector Kaczmarz} \label{subsection-strohmer-vershynin-vector-kaczmarz}
Suppose we are given $A \in \mathbb{R}^{n \times d}$ and $b \in \mathbb{R}$ such that $Ax=b$ has a solution and every row of $A$ has at least one nonzero entry.
In Strohmer-Vershynin Vector Kaczmarz method, we select an equation by sampling it with replacement from the set of all equations with a probability proportional to the sum of squares of the coefficients of the equation. Once this equation is selected, the regular Kaczmarz update is used. In our notation,
$
\varphi_R(A, b) = (e_{i_k}, \emptyset),
$
where 
\begin{equation}
\Prb{i_k = j} \propto \begin{cases}
\innorm{ e_j^\intercal A}_2^2 & j=1,\ldots,n \\
0 & \text{otherwise}.
\end{cases}
\end{equation}

\begin{lemma}
Strohmer-Vershynin Vector Kaczmarz is Markovian.
\end{lemma}
\begin{proof}
By the independence of $\lbrace i_k : k+1 \in \mathbb{N} \rbrace$, $\condPrb{ e_{i_k} \in \mathcal{W}, \zeta_{-1} \in \mathcal{Z}}{\mathcal{F}_{k+1}^k} = \Prb{ e_{i_k} \in \mathcal{W} } = \condPrb{e_{i_k} \in \mathcal{W}, \zeta_{-1} \in \mathcal{Z}}{\mathcal{F}_1^k}.$
\end{proof}

\begin{lemma}
Let $\pi_{\min} = \min_{j=1,\ldots,n} \innorm{e_j^\intercal A} / \innorm{A}_F^2$. Strohmer-Vershynin Vector Kaczmarz is $1,\pi_{\min}$-Exploratory.
\end{lemma}
\begin{proof}
Let $x_0 \neq \Prj_{\mathcal{H}}x_0$.  Then, there is an equation of $A x = b$ which is not satisfied by $x_0$. The probability of selecting this row on the first iteration is at least $\pi_{\min}$. Hence,
\begin{equation}
\sup_{\substack{ x_0 \in \mathbb{R}^d:\, x_0 \neq \Prj_{\mathcal{H}}x_0  \\ \zeta_{-1} \in \mathfrak{Z} }} \condPrb{ e_{i_0}^\intercal A(x_0 - \Prj_\mathcal{H}x_0) = 0}{\mathcal{F}_1^0} \leq 1 - \pi_{\min}.
\end{equation}
\end{proof}

\begin{remark}
Depending on the rows of $A$, we can always play with constants in the $N,\pi$-Exploratory definition to choose one that will have the smallest ratio. However, for demonstration, our selection is sufficient.
\end{remark}

Now, we apply \cref{theorem-convergence-row-finite} to conclude as follows.

\begin{theorem} \label{theorem-strohmer-verhsynin-vector-kaczmarz}
Let $A \in \mathbb{R}^{n \times d}$ and $b \in \mathbb{R}^n$ such that the linear system's solution set, $\mathcal{H}$, is nonempty. Let $\pi_{\min} = \min_{j=1,\ldots,n} \innorm{e_j^\intercal A} / \innorm{A}_F^2$. Let $x_0 \in \mathbb{R}^d$. Let $\lbrace x_k : k \in \mathbb{N} \rbrace$ be a sequence generated by the Strohmer-Vershynin Vector Kaczmarz method. Then, there exists a stopping time $\tau$ with finite expectation such that $x_{\tau} \in \mathcal{H}$; or there exists a sequence of non-negative stopping times $\lbrace \tau_{j} : j +1 \in \mathbb{N} \rbrace$ for which $\E{ \tau_{j} } \leq j [ (\rnk{A} - 1)/\pi_{\min} + 1]$,  and there exist $\gamma \in (0,1)$ and a sequence of random variables $\lbrace \gamma_j : j+1 \in \mathbb{N}\rbrace \subset (0,\gamma]$, such that
\begin{equation}
\Prb{ \bigcap_{j=0}^\infty \left\lbrace \norm{ x_{\tau_j} - \Prj_{\mathcal{H}} x_0 }_2^2 \leq \left( \prod_{\ell=0}^{j-1} \gamma_{\ell} \right) \norm{ x_0 - \Prj_{\mathcal{H}}x_0 }_2^2 \right\rbrace} = 1.
\end{equation}
\end{theorem}

\subsection{Steinerberger's Vector Kaczmarz} \label{subsection-steinerberger-vector-kaczmarz}

Suppose $Ax=b$ is a consistent system. Steinerberger's method selects an equation from the system by using an $l^p$ weighted residual, and then performing the Kaczmarz update with this equation. In our notation, $\varphi_R( A, b, x_k) = (e_{i_k}, \emptyset)$ where
\begin{equation}
\Prb{e_{i_k} = j} \propto \begin{cases}
| e_j^\intercal (A x_k - b) |^p & j=1,\ldots,n \\
0 & \text{otherwise},
\end{cases}
\end{equation} 
and
\begin{equation}
x_{k+1} = x_k - A^\intercal e_{i_k} \frac{e_{i_k}^\intercal (Ax_k - b)}{\norm{ A^\intercal e_{i_k}}_2^2}.
\end{equation}

\begin{lemma}
Steinerberger's Vector Kaczmarz method is Markovian.
\end{lemma}
\begin{proof}
As $e_{i_k}$'s distribution only depends on $x_k$, $\incondPrb{ e_{i_k} \in \mathcal{W}}{\mathcal{F}_{k+1}^k} = \incondPrb{ e_{i_k} \in \mathcal{W}}{\mathcal{F}_{1}^k}$. 
\end{proof}

\begin{lemma}
Steinberger's Vector Kaczmarz method is $1,1$-Exploratory.
\end{lemma}
\begin{proof}
For any $x_0 \neq \Prj_\mathcal{H} x_0$, $e_{i_0}^\intercal (Ax_0 - b) \neq 0$ with probability one. Hence,
\begin{equation}
\sup_{ x_0: x_0 \neq \Prj_{\mathcal{H}} x_0 } \condPrb{ e_{i_0}^\intercal (Ax_0 - b) = 0}{\mathcal{F}_{1}^0} = 0.
\end{equation}
The claim follows.
\end{proof}

We can conclude using \cref{theorem-convergence-row-finite}.
\begin{theorem} \label{theorem-steinerberger-vector-kaczmarz}
Let $A \in \mathbb{R}^{n \times d}$ and $b \in \mathbb{R}^n$ such that the linear system's solution set, $\mathcal{H}$, is nonempty. Let $x_0 \in \mathbb{R}^d$. Let $\lbrace x_k : k \in \mathbb{N} \rbrace$ be a sequence generated by the Steinerberger's Vector Kaczmarz method. Then, there exists a stopping time $\tau$ with finite expectation such that $x_{\tau} \in \mathcal{H}$; or there exists a sequence of non-negative stopping times $\lbrace \tau_{j} : j +1 \in \mathbb{N} \rbrace$ for which $\E{ \tau_{j} } \leq j \rnk{A}$,  and there exist $\gamma \in (0,1)$ and a sequence of random variables $\lbrace \gamma_j : j+1 \in \mathbb{N}\rbrace \subset (0,\gamma]$, such that
\begin{equation}
\Prb{ \bigcap_{j=0}^\infty \left\lbrace \norm{ x_{\tau_j} - \Prj_{\mathcal{H}} x_0 }_2^2 \leq \left( \prod_{\ell=0}^{j-1} \gamma_{\ell} \right) \norm{ x_0 - \Prj_{\mathcal{H}}x_0 }_2^2 \right\rbrace} = 1.
\end{equation}
\end{theorem}
\subsection{Motkzin's Method} \label{subsection-motzkin-vector}

Given a consistent linear systems, $Ax=b$, Motzkin's method starts by choosing an equation from the system whose solution hyperplane is the furthest from the current iterate, and then performs the Kaczmarz update. In our notation, $\varphi_R(A,b,x_k) = (e_{i_k}, \emptyset)$ where
\begin{equation}
i_k \in \argmax_{j \in \lbrace 1,\ldots,n \rbrace} \frac{|e_j^\intercal (Ax_k - b)|}{\norm{ A^\intercal e_j}_2^2},
\end{equation}
and
\begin{equation}
x_{k+1} = x_k - A^\intercal e_{i_k} \frac{e_{i_k}^\intercal (A x_k - b)}{\norm{A^\intercal e_j}_2^2}.
\end{equation}

\begin{lemma}
Motzkin's method is Markovian.
\end{lemma}
\begin{proof}
The method is deterministic and only depends on the information in $\mathcal{F}_1^k$. 
\end{proof}

\begin{lemma}
Motzkin's method is $1,1$-Exploratory
\end{lemma}
\begin{proof}
For any $x_0 \neq \Prj_\mathcal{H} x_0$, $Ax_0 - b \neq 0$. Therefore, there is an equation that is not satisfied by $x_0$. Hence, the set of equations that maximize the distance between $x_0$ and the solution set of the equation is nonempty. The result follows.
\end{proof}

We can now apply \cref{theorem-convergence-row-finite} to conclude.
\begin{theorem} \label{theorem-motzkin-vector}
Let $A \in \mathbb{R}^{n \times d}$ and $b \in \mathbb{R}^n$ such that the linear system's solution set, $\mathcal{H}$, is nonempty. Let $x_0 \in \mathbb{R}^d$. Let $\lbrace x_k : k \in \mathbb{N} \rbrace$ be a sequence generated by Motzkin's method. Then, there exists a stopping time $\tau$ with finite expectation such that $x_{\tau} \in \mathcal{H}$; or there exists a sequence of non-negative stopping times $\lbrace \tau_{j} : j +1 \in \mathbb{N} \rbrace$ for which $\E{ \tau_{j} } \leq j \rnk{A}$,  and there exist $\gamma \in (0,1)$ and a sequence of random variables $\lbrace \gamma_j : j+1 \in \mathbb{N}\rbrace \subset (0,\gamma]$, such that
\begin{equation}
\Prb{ \bigcap_{j=0}^\infty \left\lbrace \norm{ x_{\tau_j} - \Prj_{\mathcal{H}} x_0 }_2^2 \leq \left( \prod_{\ell=0}^{j-1} \gamma_{\ell} \right) \norm{ x_0 - \Prj_{\mathcal{H}}x_0 }_2^2 \right\rbrace} = 1.
\end{equation}
\end{theorem}
\subsection{Agmon's Method} \label{subsection-agmon-vector}

Given a consistent linear systems, $Ax=b$, Agmon's method starts by choosing an equation from the system with the largest absolute residual at the current iterate, and then performs the Kaczmarz update. In our notation, $\varphi_R(A,b,x_k) = (e_{i_k}, \emptyset)$ where
\begin{equation}
i_k \in \argmax_{j \in \lbrace 1,\ldots,n \rbrace} |e_j^\intercal (Ax_k - b)|,
\end{equation}
and
\begin{equation}
x_{k+1} = x_k - A^\intercal e_{i_k} \frac{e_{i_k}^\intercal (A x_k - b)}{\norm{A^\intercal e_j}_2^2}.
\end{equation}

\begin{lemma}
Agmon's method is Markovian.
\end{lemma}
\begin{proof}
The method is deterministic and only depends on the information in $\mathcal{F}_1^k$. 
\end{proof}

\begin{lemma}
Agmon's method is $1,1$-Exploratory
\end{lemma}
\begin{proof}
For any $x_0 \neq \Prj_\mathcal{H} x_0$, $Ax_0 - b \neq 0$. Therefore, there is an equation that is not satisfied by $x_0$. Hence, the set of equations that maximize the absolute residual at $x_0$ is nonempty. The result follows.
\end{proof}

We can now apply \cref{theorem-convergence-row-finite} to conclude.
\begin{theorem} \label{theorem-agmon-vector}
Let $A \in \mathbb{R}^{n \times d}$ and $b \in \mathbb{R}^n$ such that the linear system's solution set, $\mathcal{H}$, is nonempty. Let $x_0 \in \mathbb{R}^d$. Let $\lbrace x_k : k \in \mathbb{N} \rbrace$ be a sequence generated by Agmon's method. Then, there exists a stopping time $\tau$ with finite expectation such that $x_{\tau} \in \mathcal{H}$; or there exists a sequence of non-negative stopping times $\lbrace \tau_{j} : j +1 \in \mathbb{N} \rbrace$ for which $\E{ \tau_{j} } \leq j \rnk{A}$,  and there exist $\gamma \in (0,1)$ and a sequence of random variables $\lbrace \gamma_j : j+1 \in \mathbb{N}\rbrace \subset (0,\gamma]$, such that
\begin{equation}
\Prb{ \bigcap_{j=0}^\infty \left\lbrace \norm{ x_{\tau_j} - \Prj_{\mathcal{H}} x_0 }_2^2 \leq \left( \prod_{\ell=0}^{j-1} \gamma_{\ell} \right) \norm{ x_0 - \Prj_{\mathcal{H}}x_0 }_2^2 \right\rbrace} = 1.
\end{equation}
\end{theorem}
\subsection{Greedy Randomized Vector Kaczmarz} \label{subsection-greedy-randomized-vector-kaczmarz}

In this method, at each iteration, a threshold value is calculated based on the current iterate value
\begin{equation}
\epsilon_k = \frac{1}{2} \max_{j \in \lbrace 1,\ldots,n \rbrace} \frac{|e_j^\intercal (Ax_k - b)|^2}{\norm{Ax_k - b}_2^2\norm{ A^\intercal e_j }_2^2} + \frac{1}{2\norm{A}_F^2} ;
\end{equation}
then a subset of equations whose residual surpassing this threshold is generated,
\begin{equation}
\mathcal{U}_k = \left\lbrace j : \frac{ | e_j^\intercal (Ax_k - b) |^2}{\norm{A^\intercal e_j}_2^2\norm{Ax_k - b}_2^2} \geq \epsilon_k \right\rbrace;
\end{equation}
an equation, $i_k$, is then selected from this subset with a probability distribution
\begin{equation}
\Prb{ i_k = j } \propto \begin{cases}
| e_j^\intercal (Ax_k-b) |^2 & j \in \mathcal{U}_k \\
0 & \text{otherwise};
\end{cases}
\end{equation}
and a Kaczmarz update is then performed. 

\begin{lemma}
Greedy Randomized Vector Kaczmarz is Markovian.
\end{lemma}
\begin{proof}
The probability distribution of $e_{i_k}$ only depends on the current iterate, $x_k$. Hence, $\condPrb{ e_{i_k} \in \mathcal{W} }{\mathcal{F}_{k+1}^k} = \condPrb{ e_{i_k} \in \mathcal{W}}{\mathcal{F}_1^k}$.
\end{proof}

\begin{lemma}
Greedy Randomized Vector Kaczmarz is $1,1$-Exploratory.
\end{lemma}
\begin{proof}
Let $x_0 \neq \Prj_{\mathcal{H}} x_0$. Then, $\epsilon_0$ is a positive real number, and, if $\mathcal{U}_k$ is nonempty, for any $j \in \mathcal{U}_k$, $e_j^\intercal (A x_0 - b) \neq 0$. Therefore, if $\mathcal{U}_0$ is nonempty for all $x_0 \neq \Prj_\mathcal{H} x_0$, then
\begin{equation}
\sup_{x_0: x_0 \neq \Prj_\mathcal{H} x_0 } \condPrb{ e_{i_0}^\intercal (Ax_0 - b) = 0 } = 0.
\end{equation}
Therefore, we need only verify that $\mathcal{U}_0$ is nonempty for all $x_0 \neq \Prj_\mathcal{H} x_0$. Note,
\begin{align}
\max_{j \in \lbrace 1,\ldots, n \rbrace} \frac{| e_j^\intercal (Ax_0 - b) |^2}{\norm{A^\intercal e_j}_2^2\norm{Ax_0-b}_2^2} \geq  \sum_{j=1}^n \frac{\norm{A^\intercal e_j}_2^2}{\norm{A}_F^2} \frac{|e_j^\intercal (Ax_0 - b)|^2}{\norm{A^\intercal e_j}_2^2\norm{Ax_0-b}_2^2} = \frac{1}{\norm{A}_F^2}.
\end{align}
Hence, 
\begin{equation}
\max_{j \in \lbrace 1,\ldots, n \rbrace} \frac{| e_j^\intercal (Ax_0 - b) |^2}{\norm{A^\intercal e_j}_2^2\norm{Ax_0-b}_2^2} \geq \frac{1}{2}\max_{j \in \lbrace 1,\ldots, n \rbrace} \frac{| e_j^\intercal (Ax_0 - b) |^2}{\norm{A^\intercal e_j}_2^2\norm{Ax_0-b}_2^2} + \frac{1}{2 \norm{A}_F^2}.
\end{equation}
Therefore, for every $x_0 \neq \Prj_\mathcal{H} x_0$, $\mathcal{U}_0 \neq \emptyset$.
\end{proof}

We can now apply \cref{theorem-convergence-row-finite} to conclude as follows.
\begin{theorem} \label{theorem-greedy-randomized-vector-kaczmarz}
Let $A \in \mathbb{R}^{n \times d}$ and $b \in \mathbb{R}^n$ such that the linear system's solution set, $\mathcal{H}$, is nonempty. Let $x_0 \in \mathbb{R}^d$. Let $\lbrace x_k : k \in \mathbb{N} \rbrace$ be a sequence generated by the Greedy Randomized Vector Kaczmarz method. Then, there exists a stopping time $\tau$ with finite expectation such that $x_{\tau} \in \mathcal{H}$; or there exists a sequence of non-negative stopping times $\lbrace \tau_{j} : j +1 \in \mathbb{N} \rbrace$ for which $\E{ \tau_{j} } \leq j [ (\rnk{A} - 1)/\pi_{\min} + 1]$,  and there exist $\gamma \in (0,1)$ and a sequence of random variables $\lbrace \gamma_j : j+1 \in \mathbb{N}\rbrace \subset (0,\gamma]$, such that
\begin{equation}
\Prb{ \bigcap_{j=0}^\infty \left\lbrace \norm{ x_{\tau_j} - \Prj_{\mathcal{H}} x_0 }_2^2 \leq \left( \prod_{\ell=0}^{j-1} \gamma_{\ell} \right) \norm{ x_0 - \Prj_{\mathcal{H}}x_0 }_2^2 \right\rbrace} = 1.
\end{equation}
\end{theorem}
\subsection{Sampling Kaczmarz-Motzkin Method} \label{subsection-sampling-kaczmarz-motzkin-vector}

Let $Ax=b$ be a consistent linear system. In this method, a subset of equations of a fixed sized is randomly selected from the linear system at each iteration (independently and with uniform probability); then the equation with the largest absolute residual is selected from this subset; and the Kaczmarz update is then applied with this equation. 

\begin{lemma}
The Sampling Kaczmarz-Motzkin Method is Markovian.
\end{lemma}
\begin{proof}
The subset is selected independently and then the equation with the largest absolute residual is selected from this subset. Hence, the selection process only depends on knowledge of the current iterate.
\end{proof}

\begin{lemma}
The Sampling Kaczmarz-Motzkin Method is $1,n^{-1}$-Exploratory.
\end{lemma}
\begin{proof}
If $x_0 \neq \Prj_\mathcal{H} x_0$, then there is at least one equation in the system that is not satisfied by $x_0$. The probability that this equation is included in a sample is $1/n$. By the second step of the selection process, this equation is selected or one with a larger absolute residual is selected. Therefore, letting $e_{i_0}$ denote the standard basis element corresponding to the equation selected,
\begin{equation}
\sup_{ x_0 : x_0 \neq \Prj_\mathcal{H} x_0} \condPrb{ e_{i_0}^\intercal (Ax_0 - b) = 0}{\mathcal{F}_1^0} \leq 1 - \frac{1}{n}.
\end{equation}
\end{proof}
\begin{remark}
The ratio of $N/\pi = n$ can be improved if we have more information about the system, but our choice applies quite generally.
\end{remark}

We can now apply \cref{theorem-convergence-row-finite} to conclude as follows.

\begin{theorem} \label{theorem-sampling-kaczmarz-motzkin-vector}
Let $A \in \mathbb{R}^{n \times d}$ and $b \in \mathbb{R}^n$ such that the linear system's solution set, $\mathcal{H}$, is nonempty. Let $x_0 \in \mathbb{R}^d$. Let $\lbrace x_k : k \in \mathbb{N} \rbrace$ be a sequence generated by the Sampling Kaczmarz-Motzkin method. Then, there exists a stopping time $\tau$ with finite expectation such that $x_{\tau} \in \mathcal{H}$; or there exists a sequence of non-negative stopping times $\lbrace \tau_{j} : j +1 \in \mathbb{N} \rbrace$ for which $\E{ \tau_{j} } \leq j [(\rnk{A}-1)n + 1]$,  and there exist $\gamma \in (0,1)$ and a sequence of random variables $\lbrace \gamma_j : j+1 \in \mathbb{N}\rbrace \subset (0,\gamma]$, such that
\begin{equation}
\Prb{ \bigcap_{j=0}^\infty \left\lbrace \norm{ x_{\tau_j} - \Prj_{\mathcal{H}} x_0 }_2^2 \leq \left( \prod_{\ell=0}^{j-1} \gamma_{\ell} \right) \norm{ x_0 - \Prj_{\mathcal{H}}x_0 }_2^2 \right\rbrace} = 1.
\end{equation}

\end{theorem}
\subsection{Streaming Vector Kaczmarz} \label{subsection-streaming-vector-kaczmarz}

In this case, we apply Kaczmarz method to a stream of equations that are assumed to be independent and identically distributed. We will refer to this method as the Streaming Vector Kaczmarz method. We assume that the set $\mathcal{H} \subset \mathbb{R}^d$ of vectors that satisfy an equation with probability one is nonempty. To encapsulate this method in our framework, we assume, at iteration $k$, that $\varphi_R() = (W_k,\emptyset)$ and that we only observe $\alpha_k =  A^\intercal W_k$ and $\beta_k = b^\intercal W_k$. Then, we perform the update
\begin{equation} \label{eqn-streaming-vector-kaczmarz}
x_{k+1} = x_k - \alpha_k \frac{\alpha_k^\intercal x_k - \beta_k}{\norm{ \alpha_k}_2^2}.
\end{equation}

\begin{lemma}
The Streaming Vector Kaczmarz method is Markovian.
\end{lemma}
\begin{proof}
Since $\lbrace (\alpha_k,\beta_k) : k+1 \in \mathbb{N} \rbrace$ are independent and identically distributed, then $\condPrb{ \alpha_k \in \mathcal{W}}{\mathcal{F}_{k+1}^k} = \Prb{ \alpha_k \in \mathcal{W}} = \condPrb{\alpha_k \in \mathcal{W}}{\mathcal{F}_{1}^k}.$
\end{proof}

\begin{lemma}
There exists a $\pi \in (0,1]$ such that the Streaming Vector Kaczmarz method is $1,\pi$-Exploratory.
\end{lemma}
\begin{proof}
For any $x$, let $\mathcal{N} = \lbrace z - \Prj_\mathcal{H} x : z \in \mathcal{H} \rbrace$, and define $\mathcal{R} = \mathcal{N}^\perp$. Of course, $\mathcal{N}$ is independent of the choice of $x$, and, for any $x$, $x - \Prj_{\mathcal{H}} x \in \mathcal{R}$. Let $\mathcal{S}$ denote the unit sphere in $\mathbb{R}^{d}$. For a contradiction, suppose $\sup_{ v \in \mathcal{R} \setminus \lbrace 0 \rbrace } \inPrb{ \alpha_0^\intercal v = 0 } = 1$. Then, there is a sequence $\lbrace v_k : k \in \mathbb{N} \rbrace \subset \mathcal{R} \cap \mathcal{S}$, such that $\lim_{k \to \infty} \inPrb{ \alpha_0^\intercal v_k = 0 } = 1$. By the compactness of $\mathcal{R} \cap \mathcal{S}$, there exists a subsequence, $\lbrace v_{k_j} \rbrace$, and a vector, $w \in \mathcal{R} \cap \mathcal{S}$, such that $\lim_{j \to \infty} v_{k_j} = w$. Note, $w \in \inlinspan{ \cup_{j = \ell}^\infty \lbrace v_{k_j} \rbrace}$ for all $\ell \in \mathbb{N}$, and $\inlinspan{w} = \cap_{\ell=1}^\infty \inlinspan{ \cup_{j=\ell}^\infty \lbrace v_{k_j} \rbrace}$. Therefore,
\begin{align}
1 
&=\lim_{j \to \infty} \Prb{ \alpha_{0}^\intercal v_{k_j} = 0} =\lim_{j \to \infty} \Prb{ \alpha_{0} \perp \linspan{ \bigcup_{\ell=j}^\infty \lbrace v_{k_\ell} \rbrace} } \\
&= \Prb{ \alpha_0 \perp \bigcap_{j=1}^\infty \linspan{ \bigcup_{\ell=j}^\infty \lbrace v_{k_\ell} \rbrace}}
= \Prb{ \alpha_0 \perp w } 
= \Prb{ w \in \mathcal{R} \cap \mathcal{N} \cap \mathcal{S}}.
\end{align}
Hence, we have a contradiction as the last probability must be $0$ since $\mathcal{R} \perp \mathcal{N}$. It follows that there exists a $\pi \in (0,1]$ such that $\sup_{ v \in \mathcal{R} \setminus \lbrace 0 \rbrace } \inPrb{ \alpha_0 \perp v } \leq 1- \pi$. The result follows.
\end{proof}

\begin{lemma}
The Streaming Vector Kaczmarz is Uniformly Nontrivial.
\end{lemma}
\begin{proof}
Let $\mathfrak{Q}_k'$ be the set of all orthonormal bases of $\col(\alpha_k)$ (c.f., $\mathfrak{Q}_k$ is the set of all orthonormal bases of $\col(\alpha_k\chi_k)$). Then, for all $k+1 \in \mathbb{N}$, 
\begin{equation}
\begin{aligned}
\sup_{ Q_s \in \mathfrak{Q}_s, s \in \lbrace 1,\ldots,k \rbrace } \min_{ G \in \mathcal{G}(Q_0,\ldots,Q_k)} \det(G^\intercal G)
\geq \sup_{ Q_s \in \mathfrak{Q}_s', s \in \lbrace 1,\ldots,k \rbrace } \min_{ G \in \mathcal{G}(Q_0,\ldots,Q_k)} \det(G^\intercal G),
\end{aligned}
\end{equation}
where the latter quantity is independent of $(x_0,\zeta_{-1})$ and is positive with probability one. Therefore, for every $k+1 \in \mathbb{N}$, there exists $\epsilon_k > 0$ such that
\begin{equation}
\Prb{ \sup_{ Q_s \in \mathfrak{Q}_s', s \in \lbrace 1,\ldots,k \rbrace } \min_{ G \in \mathcal{G}(Q_0,\ldots,Q_k)} \det(G^\intercal G) > \epsilon_k } \geq \frac{1}{2}.
\end{equation}
Moreover, there exists a $K \in \mathbb{N}$ such that for $k \geq K$, $\condPrb{ \mathcal{A}_k(x_0,\zeta_{-1})}{\mathcal{F}_1^0} \geq 3/4$ for all $x_0 \neq \Prj_\mathcal{H} x_0$ and $\zeta_{-1} \in \mathfrak{Z}$, which implies that, for $k \geq K$,
\begin{equation}
\condPrb{ \sup_{ Q_s \in \mathfrak{Q}_s', s \in \lbrace 1,\ldots,k \rbrace } \min_{ G \in \mathcal{G}(Q_0,\ldots,Q_k)} \det(G^\intercal G) > \epsilon_k, \mathcal{A}_k(x_0,\zeta_{-1}) }{\mathcal{F}_1^0} \geq \frac{1}{4}.
\end{equation}
Hence, by Markov's Inequality, we can find $g_{\mathcal{A}} \geq \epsilon_k/4 > 0$.
\end{proof}

We can conclude by \cref{theorem-convergence-row-infinite}.

\begin{theorem} \label{theorem-streaming-vector-kaczmarz}
Let $\lbrace (\alpha_k,\beta_k) : k+1 \in \mathbb{N} \rbrace \subset \mathbb{R}^d \times \mathbb{R}$ be a sequence of independent, identically distributed random variables such that $\mathcal{H} = \lbrace x \in \mathbb{R}^d : \Prb{ \alpha_0^\intercal x = b} = 1 \rbrace \neq \emptyset$. Let $x_0 \in \mathbb{R}^d$ and let $\lbrace x_k : k \in \mathbb{N} \rbrace$ be generated by \cref{eqn-streaming-vector-kaczmarz}. Then, there exists a stopping time $\tau$ with finite expectation such that $x_{\tau} \in \mathcal{H}$; or there exists $\pi \in (0,1]$, there exists a sequence of non-negative stopping times $\lbrace \tau_{j} : j +1 \in \mathbb{N} \rbrace$ for which $\E{ \tau_{j} } \leq j [ (\dim{\mathcal{H}}-1)/\pi + 1]$, and there exists $\bar \gamma \in (0,1)$ such that for any $\gamma \in (\bar \gamma, 1)$,
\begin{equation}
\Prb{ \bigcup_{L=0}^\infty \bigcap_{j=L}^\infty \left\lbrace \norm{ x_{\tau_j} - \Prj_{\mathcal{H}}x_0 }_2^2 \leq \gamma^j \norm{ x_0 - \Prj_{\mathcal{H}} x_0 }_2^2 \right\rbrace} = 1.
\end{equation}
\end{theorem}
\subsection{Cyclic Vector Coordinate Descent} \label{subsection-cyclic-vector-cd}

For a description of this method, see \cref{example-cyclic-cd}. We assume that $A \in \mathbb{R}^{n\times d}$ and $b \in \mathbb{R}^n$ are arbitrary---that is, we \textit{do not} require that they form a consistent system---, and we let $r^* = - \Prj_{\ker(A^\intercal)} b$. 

\begin{lemma}
Cyclic Vector Coordinate Descent is Markovian.
\end{lemma}
\begin{proof}
Note, the search coordinate at iteration $k$ is fully determined by $\zeta_{k-1}$. Hence, the result follows.
\end{proof}

\begin{lemma}
Cyclic Vector Coordinate Descent is $d,1$-Exploratory. 
\end{lemma}
\begin{proof}
Suppose $x_0 \in \mathbb{R}^d$ such that $Ax_0 \neq b$. Then, $A^\intercal (Ax_0 - b) \neq 0$. Hence, there is some $e_{i}$ for $i \in \lbrace 1,\ldots,d \rbrace$ such that $e_i^\intercal A^\intercal (Ax_0 - b) \neq 0$. Therefore,
\begin{equation}
\sup_{ x_0 : Ax_0 - b \neq r^* } \condPrb{ \bigcap_{j=0}^{d-1} \lbrace e_i^\intercal A^\intercal (Ax_0 - b) = 0 \rbrace }{\mathcal{F}_1^0} = 0.
\end{equation}
The conclusion follows. 
\end{proof}

Since $\lbrace \col(A e_i) : i =1,\ldots,d \rbrace$ is a finite set, we can apply \cref{theorem-convergence-col-finite} to conclude as follows.
\begin{theorem} \label{theorem-cyclic-vector-cd}
Let $A \in \mathbb{R}^{n \times d}$ and $b \in \mathbb{R}^n$, and define $r^* = - \Prj_{\ker(A^\intercal)} b$. Let $x_0 \in \mathbb{R}^d$ and $\lbrace x_k : k \in \mathbb{N} \rbrace$ be a sequence generated by cyclic vector coordinate descent. Then, either there exists a stopping time $\tau$ with finite expectation such that $Ax_{\tau} - b = r^*$; or
there exists a sequence of non-negative stopping times $\lbrace \tau_{j} : j +1 \in \mathbb{N} \rbrace$ for which $\E{ \tau_{j} } \leq j [ (\rnk{A} - 1) d + 1]$,  and there exist $\gamma \in (0,1)$ and a sequence of random variables $\lbrace \gamma_j : j+1 \in \mathbb{N}\rbrace \subset (0,\gamma]$, such that
\begin{equation}
\Prb{ \bigcap_{j=0}^\infty \left\lbrace \norm{ Ax_{\tau_j} - b  - r^*}_2^2 \leq \left( \prod_{\ell=0}^{j-1} \gamma_{\ell} \right) \norm{ Ax_0 - b - r^* }_2^2 \right\rbrace} = 1.
\end{equation}
\end{theorem}
\subsection{Gaussian Vector Column Space Descent} \label{subsection-gaussian-vector-cs}
Let $A \in \mathbb{R}^{n \times d}$ and $b \in \mathbb{R}^n$ be arbitrary, and let $r^* = - \Prj_{\ker(A^\intercal)} b$. The Gaussian Vector Column Space method is specified as follows. 
Let $\lbrace w_k : k+1 \rbrace \subset \mathbb{R}^d$ be a sequence of Gaussian random variables, and consider an update scheme $x_{k+1} = x_k + \alpha_k w_k$, where
\begin{equation}
\alpha_k \in \argmin_{\alpha} \norm{ Ax_k - b + \alpha Aw_k }_2^2.
\end{equation}
By solving for $\alpha_k$, the update is
\begin{equation}
x_{k+1} = x_k + w_k \frac{w_k^\intercal A^\intercal( b - Ax_k)}{\norm{ A w_k}_2^2},
\end{equation}
which we see is of the form \cref{eqn-base-col-update} with $\varphi_C() = (w_k,\emptyset)$. 

\begin{lemma}
The Gaussian Vector Column Space method is Markovian.
\end{lemma}
\begin{proof}
This follows from the independence of $\lbrace w_k \rbrace$.
\end{proof}

\begin{lemma}
The Gaussian Vector Column Space method is $1,1$-Exploratory.
\end{lemma}
\begin{proof}
Since $w_0$ is a continuous random variable, $\inPrb{ w_0^\intercal A^\intercal v = 0 } = 0$ for any $v \neq 0$. In particular, for any $x_0$ such that $Ax_0 - b \neq r^*$, $\incondPrb{ w_0^\intercal A^\intercal (Ax_0 - b ) = 0 }{\mathcal{F}_1^0} = 0$.
\end{proof}

\begin{lemma}
The Gaussian Vector Column Space method is Uniformly Nontrivial.
\end{lemma}
\begin{proof}
Because $\lbrace w_k \rbrace$ are continuous random variables, $\incondPrb{ w_k^\intercal A^\intercal (Ax_k - b) =0 }{\mathcal{F}_{1}^k}$ is $0$ whenever $Ax_k - b \neq r^*$. Therefore, $\chi_k = 1$ whenever $Ax_k - b \neq r^*$. We now have two cases.

For the first case, $\rnk A = 1$. Then, for any $Ax_0 - b - r^*,A w_0 \in \col(A) \setminus \lbrace 0 \rbrace$. Therefore, we can substitute $A w_0 / \innorm{A w_0}_2$ with $(Ax_0 - b - r^*)/ \innorm{Ax_0 - b - r^*}_2$ in the update for $x_1$, which implies $Ax_1 - b = r^*$. Therefore, $\chi_k = 0$ for all $k \in \mathbb{N}$. So $\mathcal{G}(Q_0,\ldots,Q_k) = \mathcal{G}(Q_0)$ for all $k \in \mathbb{N}$, and for all $G \in \mathcal{G}(Q_0)$, $\det(G^\intercal G) = 1$. Therefore, 
\begin{equation}
\inf_{ x_0: Ax_0 - b \neq r^*} \sup_{k \in \mathbb{N}\cup\lbrace 0 \rbrace} \cond{ \sup_{\substack{ Q_s \in \mathfrak{Q}_s \\ s \in \lbrace 1,\ldots,k \rbrace }} \min_{G \in \mathcal{G}(Q_0,\ldots,Q_k )} \det(G^\intercal G) \1{ \mathcal{A}_k(x_{0}, \zeta_{-1})} }{\mathcal{F}_1^0} = 1.
\end{equation}
In other words, when $\rnk A = 1$, the Gaussian Vector Column Space method is uniformly nontrivial.

For the second case, $\rnk A > 1$. We now proceed by induction. Note, by the continuity of the random variables, $Ax_{k+1} - b \neq r^*$ if $Ax_{k} - b \neq r^*$, which implies $\chi_k = 1$. By induction, we can conclude that $\chi_k = 1$ for all $k+1 \in \mathbb{N}$. Therefore, owing to their independence of $x_0$ and positivity, for every $k+1 \in \mathbb{N}$ there exists $\epsilon_k > 0$ such that
\begin{equation}
\Prb{ \sup_{\substack{ Q_s \in \mathfrak{Q}_s \\ s \in \lbrace 1,\ldots,k \rbrace }} \min_{G \in \mathcal{G}(Q_0,\ldots,Q_k )} \det(G^\intercal G) > \epsilon_k } \geq \frac{1}{2}.
\end{equation}
Moreover, $\exists K \in \mathbb{N}$ such that for all $k \geq K$, $\incondPrb{ \mathcal{A}_k(x_0,\zeta_{-1} }{\mathcal{F}_1^0} \geq 3/4$. By the inclusion-exclusion principles and Markov's Inequality, for all $k \geq K$,
\begin{equation}
\cond{ \sup_{\substack{ Q_s \in \mathfrak{Q}_s \\ s \in \lbrace 1,\ldots,k \rbrace }} \min_{G \in \mathcal{G}(Q_0,\ldots,Q_k )} \det(G^\intercal G) \1{ \mathcal{A}_k(x_{0}, \zeta_{-1})} }{\mathcal{F}_1^0} \geq \epsilon_k/4.
\end{equation}
The conclusion follows.
\end{proof}

From \cref{theorem-convergence-col-infinite}, we can conclude as follows. 

\begin{theorem} \label{theorem-gaussian-vector-cs}
Let $A \in \mathbb{R}^{n \times d}$ and $b \in \mathbb{R}^n$, and define $r^* = - \Prj_{\ker(A^\intercal)} b$. Let $x_0 \in \mathbb{R}^d$ and $\lbrace x_k : k \in \mathbb{N} \rbrace$ be a sequence generated by the Gaussian Vector Column Space method. Then, either there exists a stopping time $\tau$ with finite expectation such that $Ax_{\tau} - b= r^*$; or
there exists a sequence of non-negative stopping times $\lbrace \tau_{j} : j +1 \in \mathbb{N} \rbrace$ for which $\E{ \tau_{j} } \leq j \rnk{A} $,  and there exists $\bar \gamma \in (0,1)$ such that for any $\gamma \in (\bar \gamma, 1)$,
\begin{equation}
\Prb{ \bigcup_{L=0}^\infty \bigcap_{j=L}^\infty \left\lbrace \norm{ Ax_{\tau_j} - b - r^* }_2^2 \leq \gamma^j \norm{ Ax_0 - b - r^* }_2^2 \right\rbrace} = 1.
\end{equation}
\end{theorem}
\subsection{Zouzias-Freris Vector Coordinate Descent} \label{subsection-zouzias-freris-vector-cd}

Let $A \in \mathbb{R}^{n \times d}$ and $b \in \mathbb{R}^n$ be arbitrary. Note, if $A$ has a column that is entirely zero, we can eliminate it, so we will assume this case does not happen.
In Zouzias-Freris Vector Coordinate Descent,\footnote{In \cite{zouzias2013}, the authors propose a randomized extended Kaczmarz method, not a coordinate descent method. However, the authors did propose the distribution for the coordinate descent method that we discuss here, which is why we have named this method as we have.} at iteration $k$, we select independently sample an element, $e_{i_k}$, from the standard basis of $\mathbb{R}^d$ from distribution
\begin{equation}
\Prb{ i_k = j} \propto \begin{cases}
\norm{ Ae_j}_2^2 & j=1,\ldots,n \\
0 & \text{otherwise};
\end{cases}
\end{equation}
and then apply coordinate descent with this row.

\begin{lemma} 
Zouzias-Freris Vector Coordinate Descent is Markovian.
\end{lemma}
\begin{proof}
This follows immediately from the fact that $\lbrace i_k \rbrace$ are selected independently at each iteration.
\end{proof}

\begin{lemma}
Let $\pi_{\min} = \min_j \lbrace \innorm{A e_j}_2^2/\innorm{A}_F^2 \rbrace$. Zouzias-Freris Vector Coordinate Descent is $1,\pi_{\min}$-Exploratory.
\end{lemma}
\begin{proof}
Suppose $x_0 \in \mathbb{R}^d$ such that $Ax_0 - b \neq r^*$, where $r^* = - \Prj_{\ker(A^\intercal)} b$. Then, for some $e_j$, $e_j^\intercal A^\intercal (Ax_0 - b) \neq 0$. The probability that we select this basis element on the first iteration (i.e., $k=0$) is at least $\pi_{\min}$. The conclusion follows.
\end{proof}

We now apply \cref{theorem-convergence-col-finite} to conclude.

\begin{theorem} \label{theorem-zouzias-freris-vector-cd}
Let $A \in \mathbb{R}^{n \times d}$ and $b \in \mathbb{R}^n$, and define $r^* = - \Prj_{\ker(A^\intercal)} b$. Let $x_0 \in \mathbb{R}^d$ and $\lbrace x_k : k \in \mathbb{N} \rbrace$ be a sequence generated by Zouzias-Freris vector coordinate descent. Then, either there exists a stopping time $\tau$ with finite expectation such that $Ax_{\tau} - b= r^*$; or
there exists a sequence of non-negative stopping times $\lbrace \tau_{j} : j +1 \in \mathbb{N} \rbrace$ for which $\E{ \tau_{j} } \leq j [ (\rnk{A} - 1)/\pi_{\min} + 1]$,  and there exist $\gamma \in (0,1)$ and a sequence of random variables $\lbrace \gamma_j : j+1 \in \mathbb{N}\rbrace \subset (0,\gamma]$, such that
\begin{equation}
\Prb{ \bigcap_{j=0}^\infty \left\lbrace \norm{ Ax_{\tau_j} - b  - r^*}_2^2 \leq \left( \prod_{\ell=0}^{j-1} \gamma_{\ell} \right) \norm{ Ax_0 - b - r^* }_2^2 \right\rbrace} = 1.
\end{equation}
\end{theorem}

The methods described in in this section, \cref{subsection-cyclic-vector-cd,subsection-gaussian-vector-cs} are compared on a simple statistical regression problem in \cref{figure-cd-comparison}.
\begin{figure}
\centering
\begin{tikzpicture}
\begin{axis}[
	width=0.9\textwidth,
	height=0.4\textwidth,
    title={A Comparison of Some Vector Column-action Methods},
    xlabel={Iteration},
    ylabel={Normal Residual Norm},
    ymode=log,
    xmin=0, xmax=1200,
    ymin=1e-15, ymax=1000,
    xtick={0,200,400,600,800,1000,1200},
    ytick={1e-15,1e-12,1e-9,1e-6,1e-3,1e0, 1e3},
    ymajorgrids=true,
    grid style=dashed,
]

\addplot[
    color=blue,
    mark=star,
    ]
    coordinates {
(0, 753.9367470420124)
(20, 135.49978254914294)
(40, 3.3521382066984574e-14)
}; \addlegendentry{Cyclic}

\addplot[
    color=red,
    mark=triangle,
    ]
    coordinates {
(0, 753.9367470420124)
(20, 228.69088549449305)
(40, 83.92044506159563)
(60, 67.99096225363814)
(80, 35.33495367486853)
(100, 28.24281828746998)
(120, 15.186685838463083)
(140, 13.996799033117519)
(160, 9.203506415626265)
(180, 3.3424722821509283)
(200, 1.8404946478403017)
(220, 1.3223720576244589)
(240, 0.5308163167813693)
(260, 0.3493251623855677)
(280, 0.20952000172915017)
(300, 0.12872555941080838)
(320, 0.07278538598563225)
(340, 0.038546562960857164)
(360, 0.030012161911000857)
(380, 0.018013085498288516)
(400, 0.01468277561854185)
(420, 0.011820584199559149)
(440, 0.0076003024841426025)
(460, 0.004614284475707255)
(480, 0.002347206465172056)
(500, 0.002534107836049238)
(520, 0.0007123343940431133)
(540, 0.0003510981890026634)
(560, 0.000242640099389283)
(580, 0.00012725058533352522)
(600, 0.00010193418436325157)
(620, 5.532256398265451e-5)
(640, 3.838666265212035e-5)
(660, 2.0743979014779352e-5)
(680, 2.2638324290328666e-5)
(700, 1.2095494978257245e-5)
(720, 9.172694145674525e-6)
(740, 8.037922810160893e-6)
(760, 2.5181380176444496e-6)
(780, 2.067577441502e-6)
(800, 1.1083759262133593e-6)
(820, 8.911409397940236e-7)
(840, 6.259986601091271e-7)
(860, 3.2456829426065137e-7)
(880, 1.7699440160016657e-7)
(900, 1.0942172878835439e-7)
(920, 9.231870011870951e-8)
(940, 9.561179077391397e-8)
(960, 5.2162256565834774e-8)
(980, 4.220091481103811e-8)
(1000, 2.231117554998759e-8)
(1020, 1.853361167465002e-8)
(1040, 1.0503599522286897e-8)
(1060, 9.746207907866915e-9)
};\addlegendentry{Gauss}

\addplot[
    color=black,
    mark=square,
    ]
    coordinates {
(0, 753.9367470420124)
(20, 186.10039807953115)
(40, 157.38502419683633)
(60, 156.8689614737774)
(80, 123.42322989709707)
(100, 81.4915984126917)
(120, 72.1081380758748)
(140, 70.08014492158962)
(160, 67.56292490859143)
(180, 12.643064138954221)
(200, 12.085816308000663)
(220, 3.7134180444314473)
(240, 2.5012809934816485)
(260, 2.485188324103951)
(280, 2.334914081187615)
(300, 2.0705583868250645)
(320, 1.986197906982011)
(340, 0.5945945207186315)
(360, 0.5807871346637252)
(380, 0.5659565493214924)
(400, 0.550879423692666)
(420, 0.1164283717007869)
(440, 0.060535424418227866)
(460, 0.060357516545823796)
(480, 0.0274848436190407)
(500, 0.015567851605910456)
(520, 0.01412587189608026)
(540, 0.01313174398507232)
(560, 0.010171476490129349)
(580, 0.006405461867538396)
(600, 0.0015072391502457924)
(620, 0.0009306532139607922)
(640, 0.0008495426564304165)
(660, 0.000312452588182086)
(680, 0.00023309626559241253)
(700, 7.002390443831613e-5)
(720, 2.648505513261435e-5)
(740, 2.5353553075659004e-5)
(760, 2.356946172026237e-5)
(780, 2.27619018718793e-5)
(800, 2.1092505209878564e-5)
(820, 2.1068507944316303e-5)
(840, 9.839684451804445e-6)
(860, 9.000890838850795e-6)
(880, 8.58990469043043e-6)
(900, 1.8371791828307136e-6)
(920, 1.344240324933869e-6)
(940, 1.0921306043243282e-6)
(960, 9.814290525632572e-7)
(980, 4.21932118700471e-7)
(1000, 3.2766260448719263e-7)
(1020, 2.931751953593992e-7)
(1040, 2.8710919293562086e-7)
(1060, 2.3734531896222046e-7)
(1080, 5.0565141247831436e-8)
(1100, 1.966283357656033e-8)
(1120, 1.848866667873235e-8)
(1140, 1.7993609439168736e-8)
(1160, 1.7584930349743718e-8)
(1180, 2.541330113673414e-9)
    };\addlegendentry{Zouzias-Freris}

\end{axis}
\end{tikzpicture}
\caption{A comparison of three vector column-action methods on a linear regression problem where the design matrix is derived from a balanced design of 50 treatments with twenty replicates each. All methods are stopped when the residual-norm of the normal equation is less than $10^{-8}$.} \label{figure-cd-comparison}
\end{figure}
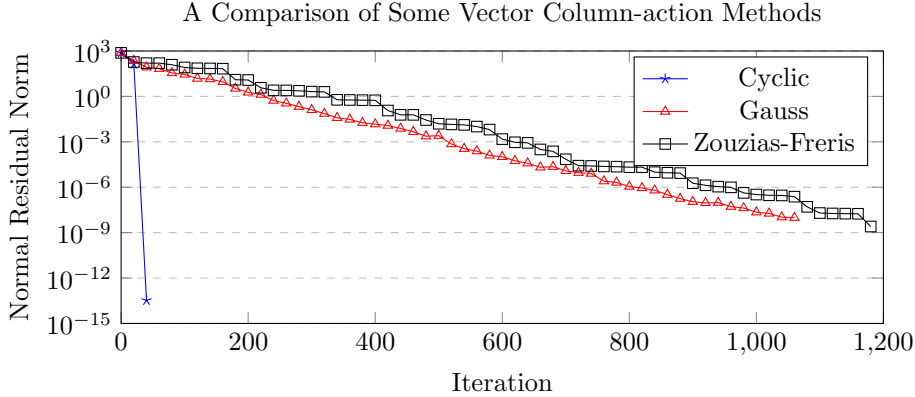
\subsection{Max Residual Vector Coordinate Descent} \label{subsection-max-residual-vector-cd}

In this approach, we have $e_{i_k}$ at an iteration $k$ such that
\begin{equation}
i_k \in \argmax_{j \in \lbrace 1,\ldots,d\rbrace} | e_j^\intercal A^\intercal (A x_k - b) |,
\end{equation}
where ties are broken by choosing the smallest index. We then do the coordinate descent update with this choice of $i_k$.\footnote{This selection procedure is an analogue of Agmon's method, but it is not Agmon's method applied to the normal equations.}

\begin{lemma}
Max Residual Vector Coordinate Descent is Markovian.
\end{lemma}
\begin{proof}
The choice of $e_{i_k}$ only depends on the most recent iterate, $x_k$. Thus, the procedure is Markovian.
\end{proof}

\begin{lemma}
Max Residual Vector Coordinate Descent is $1,1$-Exploratory.
\end{lemma}
\begin{proof}
Let $x_0$ be such that $Ax_0 - b \neq r^*$ where $r^* = - \Prj_{\ker(A^\intercal)} b$. Then, there exists a $j \in \lbrace 1,\ldots,d\rbrace$ such that $A e_j \not\perp Ax_0 - b$. As a result, $e_{i_0}^\intercal A^\intercal (Ax_0 - b) \neq 0$. The conclusion follows.
\end{proof}

We can now apply \cref{theorem-convergence-col-finite} to conclude.

\begin{theorem} \label{theorem-max-residual-vector-cd}
Let $A \in \mathbb{R}^{n \times d}$ and $b \in \mathbb{R}^n$, and define $r^* = - \Prj_{\ker(A^\intercal)} b$. Let $x_0 \in \mathbb{R}^d$ and $\lbrace x_k : k \in \mathbb{N} \rbrace$ be a sequence generated by Max Residual vector coordinate descent. Then, either there exists a stopping time $\tau$ with finite expectation such that $Ax_{\tau} - b = r^*$; or
there exists a sequence of non-negative stopping times $\lbrace \tau_{j} : j +1 \in \mathbb{N} \rbrace$ for which $\E{ \tau_{j} } \leq j \rnk{A}$,  and there exist $\gamma \in (0,1)$ and a sequence of random variables $\lbrace \gamma_j : j+1 \in \mathbb{N}\rbrace \subset (0,\gamma]$, such that
\begin{equation}
\Prb{ \bigcap_{j=0}^\infty \left\lbrace \norm{ Ax_{\tau_j} - b  - r^*}_2^2 \leq \left( \prod_{\ell=0}^{j-1} \gamma_{\ell} \right) \norm{ Ax_0 - b - r^* }_2^2 \right\rbrace} = 1.
\end{equation}
\end{theorem}
\subsection{Max Distance Vector Coordinate Descent} \label{subsection-max-distance-vector-cd}

In this approach, we have $e_{i_k}$ such that
\begin{equation}
i_k \in \argmax_{j \in \lbrace 1,\ldots,d \rbrace} \frac{|e_j^\intercal A^\intercal (A x_k - b)|}{\norm{ A^\intercal A e_j }_2^2},
\end{equation}
where ties are broken by choosing the smallest index. Then, we update $x_{k}$ using coordinate descent at coordinate $i_k$.\footnote{The selection process is analogous to Motkzin's method, but the update is not equivalent to applying Motzkin's method to the normal equations.}

\begin{lemma}
Max Distance Vector Coordinate Descent is Markovian.
\end{lemma}
\begin{proof}
The choice of $e_{i_k}$ only depends on $x_k$, which makes this procedure Markovian.
\end{proof}

\begin{lemma}
Max Distance Vector Coordinate Descent is $1,1$-Exploratory.
\end{lemma}
\begin{proof}
For any $x_0$ such that $Ax_0 - b \neq r^*$ where $r^* = - \Prj_{\ker(A^\intercal)} b$, there exists an $e_j$ such that $e_j^\intercal A^\intercal (Ax_0 - b) \neq 0$. Hence, $e_{i_0}^\intercal A^\intercal (Ax_0 - b) \neq 0$.
\end{proof}

We can now conclude by \cref{theorem-convergence-col-finite}.

\begin{theorem} \label{theorem-max-distance-vector-cd}
Let $A \in \mathbb{R}^{n \times d}$ and $b \in \mathbb{R}^n$, and define $r^* = - \Prj_{\ker(A^\intercal)} b$. Let $x_0 \in \mathbb{R}^d$ and $\lbrace x_k : k \in \mathbb{N} \rbrace$ be a sequence generated by Max Distance vector coordinate descent. Then, either there exists a stopping time $\tau$ with finite expectation such that $Ax_{\tau} - b = r^*$; or
there exists a sequence of non-negative stopping times $\lbrace \tau_{j} : j +1 \in \mathbb{N} \rbrace$ for which $\E{ \tau_{j} } \leq j \rnk{A}$,  and there exist $\gamma \in (0,1)$ and a sequence of random variables $\lbrace \gamma_j : j+1 \in \mathbb{N}\rbrace \subset (0,\gamma]$, such that
\begin{equation}
\Prb{ \bigcap_{j=0}^\infty \left\lbrace \norm{ Ax_{\tau_j} - b  - r^*}_2^2 \leq \left( \prod_{\ell=0}^{j-1} \gamma_{\ell} \right) \norm{ Ax_0 - b - r^* }_2^2 \right\rbrace} = 1.
\end{equation}
\end{theorem}
\subsection{Random Permutation Block Kaczmarz} \label{subsection-random-permutation-block-kaczmarz}

For details on this method, see \cref{example-random-permute-kaczmarz}. Suppose $A \in \mathbb{R}^{n \times d}$ and $b \in \mathbb{R}^n$ form a consistent system with solution set $\mathcal{H}$. Suppose we split the system up into $\epsilon$ blocks. 

\begin{lemma}
Random Permutation Block Kaczmarz is Markovian.
\end{lemma}
\begin{proof}
The value of $E_{i_k}$ and $\zeta_{k}$ are dependent only on $\zeta_{k-1}$ and, periodically, on an independently generated permutation. Hence, the procedure is Markovian.
\end{proof}

\begin{lemma}
Random Permutation Block Kaczmarz is $1,\epsilon^{-1}$-Exploratory.
\end{lemma}
\begin{proof}
If $x_0 \neq \Prj_{\mathcal{H}} x_0$, then for some $j \in \lbrace 1,\ldots,n\rbrace$, $E_j^\intercal (Ax_0 - b) \neq 0$. The probability of selecting this particular $E_j$ on the first sample is $1/\epsilon$. The conclusion follows.
\end{proof}

We can conclude now using \cref{theorem-convergence-row-finite}.

\begin{theorem} \label{theorem-random-permutation-block-kaczmarz}
Let $A \in \mathbb{R}^{n\times d}$ and $b \in \mathbb{R}^n$ such that the linear system's solution set, $\mathcal{H}$, is nonempty. Let $x_0 \in \mathbb{R}^d$ and let $\lbrace x_k : k\in \mathbb{N} \rbrace$ be generated by Random Permutation Block Kaczmarz. Then either there exists a stopping time $\tau$ with finite expectation such that $x_{\tau} \in \mathcal{H}$; or there exist a sequence of non-negative stopping times, $\lbrace \tau_j : j+1 \in \mathbb{N} \rbrace$, such that $\E{ \tau_j } \leq j[ (\rnk A - 1)/\epsilon + 1]$, and $\exists \gamma \in (0,1)$ and a sequence of random variables $\lbrace \gamma_j : j+1 \in \mathbb{N} \rbrace \subset (0,\gamma]$ such that
\begin{equation}
\Prb{ \bigcap_{j=0}^\infty \left\lbrace \norm{ x_{\tau_j} - \Prj_\mathcal{H} x_0 }_2^2 \leq \left( \prod_{\ell=0}^{j-1} \gamma_{\ell} \right) \norm{ x_0 - \Prj_\mathcal{H} x_0}_2^2 \right\rbrace} = 1.
\end{equation}
\end{theorem}
\subsection{Steinerberger's Block Kaczmarz} \label{subsection-steinerberger-block-kaczmarz}

Though not explicitly discussed, this method is an example of sketch-and-project \cite{gower2015,richtarik2020,gower2021}. Let $\lbrace E_j : j=1,\ldots,\epsilon \rbrace$ where $\begin{bmatrix} E_1 & \cdots & E_{\epsilon} \end{bmatrix}$ is a column permutation of the $\mathbb{R}^{n \times n}$ identity matrix. In Steinerberger's block method, at iteration $x_k$, we select $E_{i_k}$ such that
\begin{equation}
\Prb{ i_k = j } \propto \begin{cases}
\norm{ E_j^\intercal(Ax_k - b)}_p^p & j=1,\ldots,\epsilon \\
0 & \text{otherwise}.
\end{cases}
\end{equation}
With this choice of $E_{i_k}$, we compute $x_{k+1}$ using the block Kaczmarz update.

\begin{lemma}
Steinerberger's Block Kaczmarz is Markovian.
\end{lemma}
\begin{proof}
The choice of $E_{i_k}$ only depends on $x_k$, which implies the conclusion.
\end{proof}

\begin{lemma}
Steinerberger's block Kaczmarz is $1,1$-Exploratory.
\end{lemma}
\begin{proof}
Let $\mathcal{H}$ denote the solution set of the linear system. For any $x_0 \neq \Prj_\mathcal{H} x_0$, $\exists j \in \lbrace 1,\ldots,\epsilon \rbrace$ such that $E_j^\intercal (Ax_0 - b) \neq 0$. Therefore, $\incondPrb{ E_{i_0}^\intercal (Ax_0 - b) = 0 }{\mathcal{F}_1^0} = 0$. The result follows.
\end{proof}

We can conclude using \cref{theorem-convergence-row-finite} as follows.

\begin{theorem} \label{theorem-steinerberger-block-kaczmarz}
Let $A \in \mathbb{R}^{n \times d}$ and $b \in \mathbb{R}^n$ such that the linear system's solution set, $\mathcal{H}$, is nonempty. Let $x_0 \in \mathbb{R}^d$. Let $\lbrace x_k : k \in \mathbb{N} \rbrace$ be a sequence generated by the Steinerberger's Block Kaczmarz method. Then, there exists a stopping time $\tau$ with finite expectation such that $x_{\tau} \in \mathcal{H}$; or there exists a sequence of non-negative stopping times $\lbrace \tau_{j} : j +1 \in \mathbb{N} \rbrace$ for which $\E{ \tau_{j} } \leq j \rnk{A}$,  and there exist $\gamma \in (0,1)$ and a sequence of random variables $\lbrace \gamma_j : j+1 \in \mathbb{N}\rbrace \subset (0,\gamma]$, such that
\begin{equation}
\Prb{ \bigcap_{j=0}^\infty \left\lbrace \norm{ x_{\tau_j} - \Prj_{\mathcal{H}} x_0 }_2^2 \leq \left( \prod_{\ell=0}^{j-1} \gamma_{\ell} \right) \norm{ x_0 - \Prj_{\mathcal{H}}x_0 }_2^2 \right\rbrace} = 1.
\end{equation}
\end{theorem}
\subsection{Motzkin's Block Method} \label{subsection-motzkin-block}

Let $A \in \mathbb{R}^{n \times d}$ and $b \in \mathbb{R}^n$ form a consistent linear system.
In this method, we begin with $\lbrace E_{j} : j =1,\ldots,\epsilon \rbrace$ where $\begin{bmatrix} E_1 &\cdots & E_\epsilon \end{bmatrix}$ is a column permutation of the $\mathbb{R}^{n \times n}$ identity matrix. At iteration $k$, we choose $i_k$ such that
\begin{equation}
i_k \in \argmax_{j \in \lbrace 1,\ldots, \epsilon \rbrace} \norm{ (E_j^\intercal A A^\intercal E_j)^\dagger E_j^\intercal (b - Ax_k) }_2.
\end{equation}
Then, we perform the block Kaczmarz update using $E_{i_k}$. Note, this update rule is not the same as the max distance adaptive sketch-and-project \cite{gower2021}.

\begin{lemma}
Motzkin's Block Method is Markovian.
\end{lemma}
\begin{proof}
The selection of $E_{i_k}$ only depends on $x_k$. Hence, the procedure is Markovian.
\end{proof}

\begin{lemma}
Motzkin's Block Method is $1,1$-Exploratory.
\end{lemma}
\begin{proof}
Let $\mathcal{H} = \lbrace x: Ax=b \rbrace$. For any $x_0 \neq \Prj_\mathcal{H} x_0$, there exists $j \in \lbrace 1,\ldots, \epsilon \rbrace$ such that $E_{j}^\intercal (b - Ax_0)  \neq 0$. Hence, $(E_j^\intercal A A^\intercal E_j)^\dagger E_j^\intercal (b - Ax_0) \neq 0$. Therefore, $E_{i_0}^\intercal (Ax_0 - b) \neq 0$.
\end{proof}

We can now conclude using \cref{theorem-convergence-row-finite} as follows.

\begin{theorem} \label{theorem-motzkin-block}
Let $A \in \mathbb{R}^{n \times d}$ and $b \in \mathbb{R}^n$ such that the linear system's solution set, $\mathcal{H}$, is nonempty. Let $x_0 \in \mathbb{R}^d$. Let $\lbrace x_k : k \in \mathbb{N} \rbrace$ be a sequence generated by the Motzkin's block method. Then, there exists a stopping time $\tau$ with finite expectation such that $x_{\tau} \in \mathcal{H}$; or there exists a sequence of non-negative stopping times $\lbrace \tau_{j} : j +1 \in \mathbb{N} \rbrace$ for which $\E{ \tau_{j} } \leq j \rnk{A}$,  and there exist $\gamma \in (0,1)$ and a sequence of random variables $\lbrace \gamma_j : j+1 \in \mathbb{N}\rbrace \subset (0,\gamma]$, such that
\begin{equation}
\Prb{ \bigcap_{j=0}^\infty \left\lbrace \norm{ x_{\tau_j} - \Prj_{\mathcal{H}} x_0 }_2^2 \leq \left( \prod_{\ell=0}^{j-1} \gamma_{\ell} \right) \norm{ x_0 - \Prj_{\mathcal{H}}x_0 }_2^2 \right\rbrace} = 1.
\end{equation}
\end{theorem}
\subsection{Agmon's Block Method} \label{subsection-agmon-block}

Let $A \in \mathbb{R}^{n \times d}$ and $b \in \mathbb{R}^n$ form a consistent linear system.
In this method, we begin with $\lbrace E_{j} : j =1,\ldots,\epsilon \rbrace$ where $\begin{bmatrix} E_1 &\cdots & E_\epsilon \end{bmatrix}$ is a column permutation of the $\mathbb{R}^{n \times n}$ identity matrix. At iteration $k$, we choose $i_k$ such that
\begin{equation}
i_k \in \argmax_{j \in \lbrace 1,\ldots, \epsilon \rbrace} \norm{ E_j^\intercal (b - Ax_k) }_2.
\end{equation}
Then, we perform the block Kaczmarz update using $E_{i_k}$. Note, this update rule is not the same as the max distance adaptive sketch-and-project \cite{gower2021}.

\begin{lemma}
Agmon's Block Method is Markovian.
\end{lemma}
\begin{proof}
The selection of $E_{i_k}$ only depends on $x_k$. Hence, the procedure is Markovian.
\end{proof}

\begin{lemma}
Agmon's Block Method is $1,1$-Exploratory.
\end{lemma}
\begin{proof}
Let $\mathcal{H} = \lbrace x: Ax=b \rbrace$. For any $x_0 \neq \Prj_\mathcal{H} x_0$, there exists $j \in \lbrace 1,\ldots, \epsilon \rbrace$ such that $E_{j}^\intercal (b - Ax_0)  \neq 0$. Therefore, $E_{i_0}^\intercal (Ax_0 - b) \neq 0$.
\end{proof}

We can now conclude using \cref{theorem-convergence-row-finite} as follows.

\begin{theorem} \label{theorem-agmon-block}
Let $A \in \mathbb{R}^{n \times d}$ and $b \in \mathbb{R}^n$ such that the linear system's solution set, $\mathcal{H}$, is nonempty. Let $x_0 \in \mathbb{R}^d$. Let $\lbrace x_k : k \in \mathbb{N} \rbrace$ be a sequence generated by the Agmon's block method. Then, there exists a stopping time $\tau$ with finite expectation such that $x_{\tau} \in \mathcal{H}$; or there exists a sequence of non-negative stopping times $\lbrace \tau_{j} : j +1 \in \mathbb{N} \rbrace$ for which $\E{ \tau_{j} } \leq j \rnk{A}$,  and there exist $\gamma \in (0,1)$ and a sequence of random variables $\lbrace \gamma_j : j+1 \in \mathbb{N}\rbrace \subset (0,\gamma]$, such that
\begin{equation}
\Prb{ \bigcap_{j=0}^\infty \left\lbrace \norm{ x_{\tau_j} - \Prj_{\mathcal{H}} x_0 }_2^2 \leq \left( \prod_{\ell=0}^{j-1} \gamma_{\ell} \right) \norm{ x_0 - \Prj_{\mathcal{H}}x_0 }_2^2 \right\rbrace} = 1.
\end{equation}
\end{theorem}
\subsection{Adaptive Sketch-and-Project} \label{subsection-adaptive-sketch-and-project}

Let $\bar A \in \mathbb{R}^{n \times d}$ and $b \in \mathbb{R}^n$ form a consistent system. Let $B$ be a positive definite symmetric matrix. 
In adaptive sketch-and-project, we first generate a set of sketching matrices $\lbrace S_j : j=1,\ldots,\epsilon \rbrace \subset \mathbb{R}^{n \times p}$. Then, we initialize with a vector $z_0 \in \mathbb{R}^d$ and perform the update
\begin{equation}
z_{k+1} = z_k - B^{-1} \bar{A}^\intercal S_{i_k} (S_{i_k}^\intercal \bar A B^{-1} \bar{A}^\intercal S_{i_k} )^\dagger S_{i_k}^\intercal (\bar A z_k - b),
\end{equation}
where $i_k$ is selected either deterministically or by sampling from a distribution that depends on the functions
\begin{equation}
f_j(z_k) = (\bar A z_k - b)^\intercal S_j (S_j^\intercal \bar A B^{-1} \bar{A}^\intercal S_{i_k} )^\dagger S_{i_k}^\intercal (\bar A z_k - b),~j=1,\ldots,\epsilon.
\end{equation}
To fix a procedure, we consider selecting $i_k \in \argmax_{j} f_j(z_k)$, which we refer to as the maximum approach. Note, random procedures will follow a similar pattern to what we have seen previously.

To rewrite this in our notation, we let $x_k = B^{1/2} z_k$ and $A = \bar A B^{-1/2}$. Then, the update of $x_k$ is
\begin{equation}
x_{k+1} = x_k - A^\intercal  S_{i_k} (S_{i_k}^\intercal  A B^{-1} A^\intercal S_{i_k} )^\dagger S_{i_k}^\intercal (A x_k - b),
\end{equation}
and
\begin{equation}
f_j(x_k) = ( A x_k - b)^\intercal S_j (S_j^\intercal AA^\intercal S_{i_k} )^\dagger S_{i_k}^\intercal (A x_k - b),~j=1,\ldots.\epsilon.
\end{equation}
Thus, we see that the inner product is readily accounted for in our framework. 

While the sketch-and-project framework is quite general, we have already covered many of the interesting special cases already. Therefore, we will focus on a special set of sketching matrices that allow us to effectively do away with the exactness assumption for the sketch-and-project framework (see \cite[Assumption 1]{gower2021} and \cite[Assumption 2]{richtarik2020}). That is, we will focus on the case where $\lbrace S_j : j =1,\ldots,\epsilon \rbrace \subset \mathbb{R}^{n \times p}$ are independently drawn from a distribution that satisfies the Johnson-Lindenstrauss Property: there exist $C, w > 0$ such that for all $\delta \geq 0$ and for any $r \in \mathbb{R}^n$
\begin{equation}
\Prb{ \left\vert \norm{ S_1^\intercal r}_2^2 - \norm{ r }_2^2 \right\vert > \delta \norm{ r}_2^2 } 
 < 2\exp\left(-Cp\delta\min\left\lbrace \delta, w^{-1} \right\rbrace  \right).
\end{equation}

\begin{remark} 
This condition is satisfied by many interesting distributions such as Gaussian Matrices \cite{dasgupta2003}, Achlioptas Sketches \cite{achlioptas2001}, and Fast Johnson Lindenstrauss Sketches \cite{ailon2009}. Moreover, values of $C$ and $w$ are summarized in \cite[Table 2]{pritchard2022}.
\end{remark}

Moreover, for such matrices, the exactness condition \cite[Assumption 2]{richtarik2020} is effectively moot as we now explain. 
\begin{lemma} \label{lemma-jl-bound}
Suppose $\lbrace S_{j} : j=1,\ldots,\epsilon \rbrace \subset \mathbb{R}^{n \times p}$ are drawn independnetly from a distribution that satisfies the Johnson-Lindenstrauss Property. If for some $\rho > 0$
\begin{equation}
p > \frac{(\rho + 1)\log(2)}{0.999C} \max \left\lbrace \frac{1}{0.999}, w \right\rbrace,
\end{equation}
then
\begin{equation}
\sup_{ v \in \mathbb{R}^{n} \setminus \lbrace 0 \rbrace } \Prb{ \bigcap_{j=1}^\epsilon \left\lbrace \norm{ S_j^\intercal v }_2^2 = 0 \right\rbrace } \leq 2^{-\epsilon \rho }.
\end{equation}
\end{lemma}
\begin{proof}
For any $v \in \mathbb{R}^n$ such that $v \neq 0$, we apply independence and the Johnson-Lindenstrauss Property with $\delta = 0.999$ to proceed as follows.
\begin{align}
\Prb{ \bigcap_{j=1}^\epsilon \left\lbrace \norm{ S_j^\intercal v }_2^2 = 0 \right\rbrace } 
&=\prod_{j=1}^\epsilon \Prb{ \norm{ S_j^\intercal v }_2^2 = 0 } \\
&\leq \prod_{j=1}^\epsilon \Prb{ \left\vert \norm{ S_1^\intercal v}_2^2 - \norm{ v }_2^2 \right\vert > 0.999 \norm{v}_2^2 } \\
&\leq \prod_{j=1}^\epsilon \exp( - 0.999 C p \min\lbrace 0.999, w^{-1} \rbrace + \log(2) ) \\
&\leq \exp( -\epsilon \rho \log(2) ),
\end{align}
where the last line follows for the constraint on $p$.
\end{proof}

To demonstrate how we use this, consider Achlioptas Sketches for which $C=0.23467$ and $w =0.1127$. If we choose $\rho = 4$, then we can choose $p=15$. Moreover, if we sample $\epsilon = 20$ such matrices, then the probability bound in the preceding lemma is bounded by $10^{-30}$. Thus, if we used $20$ independently sampled Achlioptas Sketches of embedding dimension $15$, then the probability that each of them would find $S_j^\intercal (Ax_0 - b) = 0$ when $Ax_0 \neq b$ is less than $10^{-30}$. To explain the scale of this, in expectation, we would need to repeat this process (independently) a trillion times a second for the remaining life of the sun (8 billion years, conservatively) before we generated a sample (in exact arithmetic) for which we could find an $x_0$ such that $Ax_0 \neq b$ and $S_j^\intercal (Ax_0 - b) = 0$ for all $j=1,\ldots,20$.

We can now verify the relevant properties to analyze this case of maximum adaptive sketch-and-project.

\begin{lemma}
For a fixed $\lbrace S_1,\ldots,S_{\epsilon} \rbrace$, Maximum Adaptive Sketch-and-Project is Markovian.
\end{lemma}
\begin{proof}
The selection method for $S_{i_k}$ only depends on $x_k$, hence it is Markovian.
\end{proof}

\begin{lemma}
Suppose $\lbrace S_1,\ldots,S_{\epsilon} \rbrace$ are sampled independently from a distribution satisfying the Johnson-Lindenstrauss Property with
\begin{equation}
p > \frac{(\rho + 1)\log(2)}{0.999C} \max \left\lbrace \frac{1}{0.999}, w \right\rbrace,
\end{equation}
for some choice of $\rho > 0$. Maximum Adaptive Sketch-and-Project is $1,1$-Exploratory with probability at least $1-2^{-\epsilon \rho}$.
\end{lemma}
\begin{proof}
Suppose $x_0$ is such that $Ax_0 \neq b$. Then, off of an event of probability at most $2^{-\epsilon \rho}$, we will choose $\lbrace S_1,\ldots,S_\epsilon \rbrace$ such that $\exists j \in \lbrace 1,\ldots,\epsilon \rbrace$ for which $S_j^\intercal (Ax_0 - b) \neq 0$ on this event (up to measure zero). Therefore, $S_{i_0}^\intercal (Ax_0 - b) \neq 0$ with probability one on this event (up to measure zero). The claim follows.
\end{proof}

We can now apply \cref{theorem-convergence-row-finite} to conclude as follows.
\begin{theorem} \label{theorem-adaptive-sketch-and-project}
Suppose $\lbrace S_1,\ldots,S_{\epsilon} \rbrace$ are sampled independently from a distribution satisfying the Johnson-Lindenstrauss Property with
\begin{equation}
p > \frac{(\rho + 1)\log(2)}{0.999C} \max \left\lbrace \frac{1}{0.999}, w \right\rbrace,
\end{equation}
for some choice of $\rho > 0$. Let $A \in \mathbb{R}^{n \times d}$ and $b \in \mathbb{R}^n$ such that the linear system's solution set, $\mathcal{H}$, is nonempty. Let $x_0 \in \mathbb{R}^d$. Let $\lbrace x_k : k \in \mathbb{N} \rbrace$ be a sequence generated by Maximum Adaptive Sketch-and-Project. Then, on an event of probability at least $1 - 2^{-\epsilon \rho}$, either there exists a stopping time $\tau$ such that $\incond{ \tau }{\sigma(S_1,\ldots,S_\epsilon)} < \infty$ and such that $x_{\tau} \in \mathcal{H}$; or there exists a sequence of stopping times $\lbrace \tau_{j} : j+1 \in \mathbb{N} \rbrace$ such that $\incond{ \tau_j }{\sigma(S_1,\ldots,S_\epsilon)} \leq j \rnk A$ and $\exists \gamma \in (0,1)$ (depending on $\lbrace S_1,\ldots,S_\epsilon \rbrace$) such that 
\begin{equation}
\norm{ x_{\tau_j} - \Prj_\mathcal{H} x_0 }_2^2 \leq \gamma^j \norm{ x_0 - \Prj_\mathcal{H} x_0 }_2^2
\end{equation}
up to a set of measure zero.
\end{theorem}

\subsection{Greedy Randomized Block Kaczmarz} \label{subsection-greedy-randomized-block-kaczmarz}

Suppose $A \in \mathbb{R}^{n\times d}$ and $b \in \mathbb{R}^n$ form a consistent linear system with a solution set $\mathcal{H}$. Moreover, let $\lbrace E_j : j=1,\ldots,\epsilon \rbrace$ be such that $\begin{bmatrix} E_1 & \cdots & E_\epsilon \end{bmatrix}$ is a column permutation of the $\mathbb{R}^{n\times n}$ identity matrix. For the Greedy Randomized Block Kaczmarz method at iteration $k$, we first compute a threshold,
\begin{equation}
\epsilon_k = \max_{j \in\lbrace 1,\ldots,\epsilon \rbrace} \frac{\norm{E_j^\intercal (Ax_k - b)}_2^2}{2\norm{Ax_k - b}_2^2 \norm{A^\intercal E_j}_F^2} + \frac{1}{2 \norm{A}_F^2},
\end{equation}
and identify the set
\begin{equation}
\mathcal{U}_k = \left\lbrace j : \frac{\norm{E_j^\intercal (Ax_k - b)}_2^2}{\norm{Ax_k - b}_2^2 \norm{A^\intercal E_j}_F^2} \geq \epsilon_k \right\rbrace.
\end{equation}
Then, we randomly choose an index $i_k$ according to the distribution
\begin{equation}
\Prb{ i_k = j } \propto \begin{cases}
\norm{ E_j^\intercal (Ax_k - b)}_2^2 & j \in \mathcal{U}_k \\
0 & \text{otherwise}.
\end{cases}
\end{equation}
We then perform the Kaczmarz update with $E_{i_k}$ to compute $x_{k+1}$ from $x_k$.

\begin{lemma}
Greedy Randomized Block Kaczmarz is Markovian.
\end{lemma}
\begin{proof}
The selection of $E_{i_k}$ only depends on knowing $x_k$. Hence, the procedure is Markovian.
\end{proof}

\begin{lemma}
Greedy Randomized Block Kaczmarz is $1,1$-Exploratory.
\end{lemma}
\begin{proof}
Let $x_0 \neq \Prj_{\mathcal{H}}x_0$. If we show that $\mathcal{U}_0$ is nonempty, then it follows that for any $j \in \mathcal{U}_0$, $E_j^\intercal (Ax_0 - b) \neq 0$, and, in particular, $E_{i_0}^\intercal (Ax_0 - b) \neq 0$. To verify that $\mathcal{U}_0$ is nonempty, note
\begin{equation}
\max_{j \in\lbrace 1,\ldots,\epsilon \rbrace} \frac{\norm{E_j^\intercal (Ax_0 - b)}_2^2}{\norm{Ax_0 - b}_2^2 \norm{A^\intercal E_j}_F^2} \geq \sum_{j=1}^\epsilon \frac{\norm{A^\intercal E_j}_F^2}{\norm{A}_F^2} \frac{\norm{E_j^\intercal (Ax_k - b)}_2^2}{\norm{Ax_0 - b}_2^2 \norm{A^\intercal E_j}_F^2} = \frac{1}{\norm{A}_F^2}.
\end{equation}
Therefore, 
\begin{equation}
\max_{j \in\lbrace 1,\ldots,\epsilon \rbrace} \frac{\norm{E_j^\intercal (Ax_0 - b)}_2^2}{\norm{Ax_0 - b}_2^2 \norm{A^\intercal E_j}_F^2} \geq \epsilon_0.
\end{equation}
\end{proof}

We can now conclude using \cref{theorem-convergence-row-finite} as follows.

\begin{theorem} \label{theorem-greedy-randomized-block-kaczmarz}
Let $A \in \mathbb{R}^{n \times d}$ and $b \in \mathbb{R}^n$ such that the linear system's solution set, $\mathcal{H}$, is nonempty. Let $x_0 \in \mathbb{R}^d$. Let $\lbrace x_k : k \in \mathbb{N} \rbrace$ be a sequence generated by Greedy Block Randomized Kaczmarz. Then, there exists a stopping time $\tau$ with finite expectation such that $x_{\tau} \in \mathcal{H}$; or there exists a sequence of non-negative stopping times $\lbrace \tau_{j} : j +1 \in \mathbb{N} \rbrace$ for which $\E{ \tau_{j} } \leq j \rnk{A}$,  and there exist $\gamma \in (0,1)$ and a sequence of random variables $\lbrace \gamma_j : j+1 \in \mathbb{N}\rbrace \subset (0,\gamma]$, such that
\begin{equation}
\Prb{ \bigcap_{j=0}^\infty \left\lbrace \norm{ x_{\tau_j} - \Prj_{\mathcal{H}} x_0 }_2^2 \leq \left( \prod_{\ell=0}^{j-1} \gamma_{\ell} \right) \norm{ x_0 - \Prj_{\mathcal{H}}x_0 }_2^2 \right\rbrace} = 1.
\end{equation}
\end{theorem}
\subsection{Streaming Block Kaczmarz} \label{subsection-streaming-block-kaczmarz}

We assume that we have a sequence of independent, identically distributed random variables, $\lbrace (\alpha_k, \beta_k) : k+1 \in \mathbb{N} \rbrace \subset \mathbb{R}^{d \times p} \times \mathbb{R}^p$ (for $p > 1$), such that $\mathcal{H} = \lbrace x \in \mathbb{R}^d : \Prb{ \alpha_0^\intercal x = \beta_0} = 1 \rbrace \neq \emptyset$. Given $x_0 \in \mathbb{R}^d$, we generate $\lbrace x_k : k \in \mathbb{N} \rbrace$ according to $x_{k+1} = x_k - \alpha_k(\alpha_k^\intercal \alpha_k )^\dagger (\alpha_k^\intercal x_k - \beta_k)$.

\begin{lemma}
Streaming Block Kaczmarz is Markovian.
\end{lemma}
\begin{proof}
By independence, $\condPrb{ \alpha_k \in \mathcal{W} }{\mathcal{F}_{k+1}^k } = \Prb{\alpha_k \in \mathcal{W}} = \condPrb{\alpha_k \in \mathcal{W}}{\mathcal{F}_1^k}$. 
\end{proof}

\begin{lemma}
There exists a $\pi \in (0,1]$ such that Streaming Block Kaczmarz is $1,\pi$-Exploratory.
\end{lemma}
\begin{proof}
For any $x$, let $\mathcal{N} = \lbrace z - \Prj_\mathcal{H} x : z \in \mathcal{H} \rbrace$ and $\mathcal{R} = \mathcal{N}^\perp$. Note, $\mathcal{N}$ does not depend on the choice of $x$ and, for any $x$, $x - \Prj_\mathcal{H} x \in \mathcal{R}$. Let $\mathcal{S}$ denote the unit sphere in $\mathbb{R}^d$.

For a contradiction, suppose there exists a sequence $\lbrace v_k : k \in \mathbb{N} \rbrace \subset \mathcal{R} \cap \mathcal{S}$ such that $\lim_{k \to \infty} \Prb{ \col(\alpha_0) \perp v_k } = 1$. Then, there exists a $w \in \mathcal{R} \cap \mathcal{S}$ and a subsequence $\lbrace v_{k_j} \rbrace$ such that $\lim_{j \to \infty} v_{k_j} = w$. Note, $w \in \inlinspan{ \cup_{\ell=j}^\infty \lbrace v_{k_\ell} \rbrace}$ for all $j \in \mathbb{N}$, and, consequently, $\inlinspan(w) = \cap_{j=1}^\infty \inlinspan{\cup_{\ell=j}^\infty \lbrace v_{k_\ell} \rbrace}$. Therefore,

\begin{equation}
\begin{aligned}
1 &= \lim_{k \to \infty} \Prb{ \col(\alpha_0) \perp v_{k}} = \lim_{j \to \infty} \Prb{ \col(\alpha_0) \perp \linspan{ \bigcup_{\ell=j}^\infty \lbrace v_{k_\ell} \rbrace } } \\
&= \Prb{\col(\alpha_0) \perp \bigcap_{j=1}^\infty \linspan{ \bigcup_{\ell=j}^\infty \lbrace v_{k_\ell} \rbrace}} = \Prb{ \col(\alpha_0) \perp w } = \Prb{ w \in \mathcal{R} \cap \mathcal{N} \cap \mathcal{S}}.
\end{aligned}
\end{equation}

The ultimate probability is zero, which supplies the contradiction. Thus, no such sequence $\lbrace v_k \rbrace$, which implies $\exists \pi \in (0,1]$ such that $\sup_{ v \in \mathcal{R} \setminus \lbrace 0 \rbrace} \inPrb{ \col(\alpha_0) \perp v } \leq 1 - \pi$.
\end{proof}

\begin{lemma}
The Streaming Vector Kaczmarz is Uniformly Nontrivial.
\end{lemma}
\begin{proof}
Let $\mathfrak{Q}_k'$ be the set of all orthonormal bases of $\col(\alpha_k)$ (c.f., $\mathfrak{Q}_k$ is the set of all orthonormal bases of $\col(\alpha_k\chi_k)$). Then, for all $k+1 \in \mathbb{N}$, 
\begin{equation}
\begin{aligned}
\sup_{ Q_s \in \mathfrak{Q}_s, s \in \lbrace 1,\ldots,k \rbrace } \min_{ G \in \mathcal{G}(Q_0,\ldots,Q_k)} \det(G^\intercal G)
\geq \sup_{ Q_s \in \mathfrak{Q}_s', s \in \lbrace 1,\ldots,k \rbrace } \min_{ G \in \mathcal{G}(Q_0,\ldots,Q_k)} \det(G^\intercal G),
\end{aligned}
\end{equation}
where the latter quantity is independent of $(x_0,\zeta_{-1})$ and is positive with probability one. Therefore, for every $k+1 \in \mathbb{N}$, there exists $\epsilon_k > 0$ such that
\begin{equation}
\Prb{ \sup_{ Q_s \in \mathfrak{Q}_s', s \in \lbrace 1,\ldots,k \rbrace } \min_{ G \in \mathcal{G}(Q_0,\ldots,Q_k)} \det(G^\intercal G) > \epsilon_k } \geq \frac{1}{2}.
\end{equation}
Moreover, there exists a $K \in \mathbb{N}$ such that for $k \geq K$, $\condPrb{ \mathcal{A}_k(x_0,\zeta_{-1})}{\mathcal{F}_1^0} \geq 3/4$ for all $x_0 \neq \Prj_\mathcal{H} x_0$ and $\zeta_{-1} \in \mathfrak{Z}$, which implies that, for $k \geq K$,
\begin{equation}
\condPrb{ \sup_{ Q_s \in \mathfrak{Q}_s', s \in \lbrace 1,\ldots,k \rbrace } \min_{ G \in \mathcal{G}(Q_0,\ldots,Q_k)} \det(G^\intercal G) > \epsilon_k, \mathcal{A}_k(x_0,\zeta_{-1}) }{\mathcal{F}_1^0} \geq \frac{1}{4}.
\end{equation}
Hence, by Markov's Inequality, we can find $g_{\mathcal{A}} \geq \epsilon_k/4 > 0$.
\end{proof}

We can conclude by \cref{theorem-convergence-row-infinite}.

\begin{theorem} \label{theorem-streaming-block-kaczmarz}
Let $\lbrace (\alpha_k,\beta_k) : k+1 \in \mathbb{N} \rbrace \subset \mathbb{R}^{d \times p} \times \mathbb{R}^p$ be a sequence of independent, identically distributed random variables such that $\mathcal{H} = \lbrace x \in \mathbb{R}^d : \Prb{ \alpha_0^\intercal x = b} = 1 \rbrace \neq \emptyset$. Let $x_0 \in \mathbb{R}^d$ and let $\lbrace x_k : k \in \mathbb{N} \rbrace$ be generated by Steaming Block Kaczmarz. Then, there exists a stopping time $\tau$ with finite expectation such that $x_{\tau} \in \mathcal{H}$; or there exists $\pi \in (0,1]$, there exists a sequence of non-negative stopping times $\lbrace \tau_{j} : j +1 \in \mathbb{N} \rbrace$ for which $\E{ \tau_{j} } \leq j [ (\dim{\mathcal{H}}-1)/\pi + 1]$, and there exists $\bar \gamma \in (0,1)$ such that for any $\gamma \in (\bar \gamma, 1)$,
\begin{equation}
\Prb{ \bigcup_{L=0}^\infty \bigcap_{j=L}^\infty \left\lbrace \norm{ x_{\tau_j} - \Prj_{\mathcal{H}}x_0 }_2^2 \leq \gamma^j \norm{ x_0 - \Prj_{\mathcal{H}} x_0 }_2^2 \right\rbrace} = 1.
\end{equation}
\end{theorem}
\subsection{Random Permutation Block Coordinate Descent} \label{subsection-random-permutation-block-cd}

For details of this method, see \cref{example-cyclic-block-cd}. Let $\mathbb{A}^{n \times d}$ and $b \in \mathbb{R}^n$, and let $r^* = - \Prj_{\ker(A^\intercal)} b$. 

\begin{lemma}
Random Permutation Block Coordinate Descent is Markovian.
\end{lemma}
\begin{proof}
The selection of $(W_k,\zeta_{k-1})$ only depends on $\zeta_{k-1}$. Hence, the method is Markovian.
\end{proof}

\begin{lemma}
Random Permutation Block Coordinate Descent is \\ $1,\epsilon^{-1}$-Exploratory.
\end{lemma}
\begin{proof}
If $Ax_0 - b \neq r^*$, then $\exists j \in \lbrace 1,\ldots,\epsilon \rbrace$ such that $E_j^\intercal A^\intercal (Ax_0 - b) \neq 0$. The probability that we observe this $E_j$ as the first element of a random permutation is $\epsilon^{-1}$. The result follows.
\end{proof}

We now conclude using \cref{theorem-convergence-col-finite} as follows.

\begin{theorem} \label{theorem-random-permutation-block-cd}
Let $A \in \mathbb{R}^{n \times d}$ and $b \in \mathbb{R}^n$, and define $r^* = - \Prj_{\ker(A^\intercal)} b$. Let $x_0 \in \mathbb{R}^d$ and $\lbrace x_k : k \in \mathbb{N} \rbrace$ be a sequence generated by Random Permutation Block Coordinate Descent. Then, either there exists a stopping time $\tau$ with finite expectation such that $Ax_{\tau} - b = r^*$; or
there exists a sequence of non-negative stopping times $\lbrace \tau_{j} : j +1 \in \mathbb{N} \rbrace$ for which $\E{ \tau_{j} } \leq j [ (\rnk{A} - 1) \epsilon + 1]$,  and there exist $\gamma \in (0,1)$ and a sequence of random variables $\lbrace \gamma_j : j+1 \in \mathbb{N}\rbrace \subset (0,\gamma]$, such that
\begin{equation}
\Prb{ \bigcap_{j=0}^\infty \left\lbrace \norm{ Ax_{\tau_j} - b  - r^*}_2^2 \leq \left( \prod_{\ell=0}^{j-1} \gamma_{\ell} \right) \norm{ Ax_0 - b - r^* }_2^2 \right\rbrace} = 1.
\end{equation}
\end{theorem}
\subsection{Gaussian Block Column Space Descent} \label{subsection-gaussian-block-cs}

Let $A \in \mathbb{R}^{n \times d}$, $b \in \mathbb{R}^{n}$, and $r^* = - \Prj_{\ker(A^\intercal)} b$. This method proceeds at each iteration by independently sampling a Gaussian matrix, $W_k$, that maps into $\mathbb{R}^d$, and computes $x_{k+1} = x_k + W_k \alpha_k$ where $\alpha_k \in \argmin_{\alpha} \innorm{ Ax_k - b + AW_k \alpha}_2$. 

\begin{lemma}
Gaussian Block Column Space Descent is Markovian.
\end{lemma}
\begin{proof}
This follows from the independence of $\mathcal{W}_k$. 
\end{proof}

\begin{lemma}
Gaussian Block Column Space Descent is $1,1$-Exploratory.
\end{lemma}
\begin{proof}
Since $W_0$ is a continuous random variable, $\inPrb{ W_0^\intercal A^\intercal (A x_0 - b) = 0 } = 0$ for any $Ax_0 - b \neq r^*$. The conclusion follows.
\end{proof}

\begin{lemma}
Gaussian Block Column Space Descent is Uniformly Nontrivial.
\end{lemma}
\begin{proof}
Let $\mathfrak{Q}_k'$ be the set of all orthonormal bases of $\col(\alpha_k)$ (c.f., $\mathfrak{Q}_k$ is the set of all orthonormal bases of $\col(\alpha_k\chi_k)$). Then, for all $k+1 \in \mathbb{N}$, 
\begin{equation}
\begin{aligned}
\sup_{ Q_s \in \mathfrak{Q}_s, s \in \lbrace 1,\ldots,k \rbrace } \min_{ G \in \mathcal{G}(Q_0,\ldots,Q_k)} \det(G^\intercal G)
\geq \sup_{ Q_s \in \mathfrak{Q}_s', s \in \lbrace 1,\ldots,k \rbrace } \min_{ G \in \mathcal{G}(Q_0,\ldots,Q_k)} \det(G^\intercal G),
\end{aligned}
\end{equation}
where the latter quantity is independent of $(x_0,\zeta_{-1})$ and is positive with probability one. Therefore, for every $k+1 \in \mathbb{N}$, there exists $\epsilon_k > 0$ such that
\begin{equation}
\Prb{ \sup_{ Q_s \in \mathfrak{Q}_s', s \in \lbrace 1,\ldots,k \rbrace } \min_{ G \in \mathcal{G}(Q_0,\ldots,Q_k)} \det(G^\intercal G) > \epsilon_k } \geq \frac{1}{2}.
\end{equation}
Moreover, there exists a $K \in \mathbb{N}$ such that for $k \geq K$, $\condPrb{ \mathcal{A}_k(x_0,\zeta_{-1})}{\mathcal{F}_1^0} \geq 3/4$ for all $x_0 \neq \Prj_\mathcal{H} x_0$ and $\zeta_{-1} \in \mathfrak{Z}$, which implies that, for $k \geq K$,
\begin{equation}
\condPrb{ \sup_{ Q_s \in \mathfrak{Q}_s', s \in \lbrace 1,\ldots,k \rbrace } \min_{ G \in \mathcal{G}(Q_0,\ldots,Q_k)} \det(G^\intercal G) > \epsilon_k, \mathcal{A}_k(x_0,\zeta_{-1}) }{\mathcal{F}_1^0} \geq \frac{1}{4}.
\end{equation}
Hence, by Markov's Inequality, $g_{\mathcal{A}} \geq \epsilon_k/4 > 0$ for all $k \geq K$. 
\end{proof}

We conclude by \cref{theorem-convergence-col-infinite} as follows.
\begin{theorem} \label{theorem-gaussian-block-cs}
Let $A \in \mathbb{R}^{n \times d}$ and $b \in \mathbb{R}^n$, and define $r^* = - \Prj_{\ker(A^\intercal)} b$. Let $x_0 \in \mathbb{R}^d$ and $\lbrace x_k : k \in \mathbb{N} \rbrace$ be a sequence generated by Gaussian Block Column Space Descent. Then, either there exists a stopping time $\tau$ with finite expectation such that $Ax_{\tau} - b = r^*$; or
there exists a sequence of non-negative stopping times $\lbrace \tau_{j} : j +1 \in \mathbb{N} \rbrace$ for which $\E{ \tau_{j} } \leq j \rnk{A} $,  and there exists $\bar \gamma \in (0,1)$ such that for any $\gamma \in (\bar \gamma, 1)$,
\begin{equation}
\Prb{ \bigcup_{L=0}^\infty \bigcap_{j=L}^\infty \left\lbrace \norm{ Ax_{\tau_j} - b - r^* }_2^2 \leq \gamma^j \norm{ Ax_0 - b - r^* }_2^2 \right\rbrace} = 1.
\end{equation}
\end{theorem}
\subsection{Zouzias-Freris Block Coordinate Descent} \label{subsection-zouzias-freris-block-cd}

Let $A \in \mathbb{R}^{n \times d}$, $b \in \mathbb{R}^{n}$ and $r^* = - \Prj_{\ker(A^\intercal)} b$. Let $\lbrace E_j : j=1,\ldots,\epsilon \rbrace$ be such that $\begin{bmatrix}
E_1 & \cdots & E_{\epsilon}
\end{bmatrix}$ is a column permutation of the $R^{d \times d }$ identity matrix. In this method, at iteration $k$, we select $E_{i_k}$ independently according to 
\begin{equation}
\Prb{ i_k = j} \propto \begin{cases}
\norm{A E_j}_F^2 & j=1,\ldots,\epsilon \\
0 & \text{Otherwise},
\end{cases}
\end{equation}
and we compute $x_{k+1} = x_k + E_{i_k} \alpha_k$ where $\alpha_k \in \argmin_\alpha \innorm{ Ax_k - b - A E_{i_k} \alpha }_2$. 

\begin{lemma}
Zouzias-Freris Block Coordinate Descent is Markovian.
\end{lemma}
\begin{proof}
This follows from the independence of $\lbrace i_k \rbrace$.
\end{proof}

\begin{lemma}
Let $\pi_{\min} = \min_j \lbrace \innorm{AE_j}_F^2/\innorm{A}_F^2 : AE_j \neq 0 \rbrace$. Zouzias-Freris Block Coordinate Descent is $1,\pi_{\min}$-Exploratory.
\end{lemma}
\begin{proof}
If $Ax_0 - b \neq r^*$, then $\exists j \in \lbrace 1,\ldots,\epsilon \rbrace$ such that $E_j^\intercal A^\intercal (Ax_0 - b)\neq 0$. Therefore, $\Prb{ i_0 = j } \geq \pi_{\min}$. The conclusion follows.
\end{proof}

By \cref{theorem-convergence-col-finite}, we conclude as follows.

\begin{theorem} \label{theorem-zouzias-freris-block-cd}
Let $A \in \mathbb{R}^{n \times d}$ and $b \in \mathbb{R}^n$, and define $r^* = - \Prj_{\ker(A^\intercal)} b$. Let $x_0 \in \mathbb{R}^d$ and $\lbrace x_k : k \in \mathbb{N} \rbrace$ be a sequence generated by Zouzias-Freris Block Coordinate Descent. Then, either there exists a stopping time $\tau$ with finite expectation such that $Ax_{\tau} - b = r^*$; or
there exists a sequence of non-negative stopping times $\lbrace \tau_{j} : j +1 \in \mathbb{N} \rbrace$ for which $\E{ \tau_{j} } \leq j [ (\rnk{A} - 1) /\pi_{\min} + 1]$,  and there exist $\gamma \in (0,1)$ and a sequence of random variables $\lbrace \gamma_j : j+1 \in \mathbb{N}\rbrace \subset (0,\gamma]$, such that
\begin{equation}
\Prb{ \bigcap_{j=0}^\infty \left\lbrace \norm{ Ax_{\tau_j} - b  - r^*}_2^2 \leq \left( \prod_{\ell=0}^{j-1} \gamma_{\ell} \right) \norm{ Ax_0 - b - r^* }_2^2 \right\rbrace} = 1.
\end{equation}
\end{theorem}
\subsection{Max Residual Block Coordinate Descent} \label{subsection-max-residual-block-cd}

Let $A \in \mathbb{R}^{n \times d}$, $b \in \mathbb{R}^{n}$ and $r^* = - \Prj_{\ker(A^\intercal)} b$. Let $\lbrace E_j : j=1,\ldots,\epsilon \rbrace$ be such that $\begin{bmatrix}
E_1 & \cdots & E_{\epsilon}
\end{bmatrix}$ is a column permutation of the $R^{d \times d }$ identity matrix. In this method, at iteration $k$, we select $E_{i_k}$ according to 
\begin{equation}
i_k \in \argmax_{j \in \lbrace 1,\ldots,\epsilon \rbrace } \norm{ E_j^\intercal A^\intercal (Ax_k - b) }_2,
\end{equation}
and we compute $x_{k+1} = x_k + E_{i_k} \alpha_k$ where $\alpha_k \in \argmin_\alpha \innorm{ Ax_k - b - A E_{i_k} \alpha }_2$. 

\begin{lemma}
Max Residual Block Coordinate Descent is Markovian.
\end{lemma}
\begin{proof}
The selection of $\lbrace i_k \rbrace$ depends only on $x_k$, which implies the conclusion.
\end{proof}

\begin{lemma}
Max Residual Block Coordinate Descent is $1,1$-Exploratory.
\end{lemma}
\begin{proof}
If $Ax_0 - b \neq r^*$, then $\exists j \in \lbrace 1,\ldots,\epsilon \rbrace$ such that $E_j^\intercal A^\intercal (Ax_0 - b)\neq 0$. The conclusion follows.
\end{proof}

By \cref{theorem-convergence-col-finite}, we conclude as follows.

\begin{theorem} \label{theorem-max-residual-block-cd}
Let $A \in \mathbb{R}^{n \times d}$ and $b \in \mathbb{R}^n$, and define $r^* = - \Prj_{\ker(A^\intercal)} b$. Let $x_0 \in \mathbb{R}^d$ and $\lbrace x_k : k \in \mathbb{N} \rbrace$ be a sequence generated by Max Residual Block Coordinate Descent. Then, either there exists a stopping time $\tau$ with finite expectation such that $Ax_{\tau} - b = r^*$; or
there exists a sequence of non-negative stopping times $\lbrace \tau_{j} : j +1 \in \mathbb{N} \rbrace$ for which $\E{ \tau_{j} } \leq j \rnk{A}$,  and there exist $\gamma \in (0,1)$ and a sequence of random variables $\lbrace \gamma_j : j+1 \in \mathbb{N}\rbrace \subset (0,\gamma]$, such that
\begin{equation}
\Prb{ \bigcap_{j=0}^\infty \left\lbrace \norm{ Ax_{\tau_j} - b  - r^*}_2^2 \leq \left( \prod_{\ell=0}^{j-1} \gamma_{\ell} \right) \norm{ Ax_0 - b - r^* }_2^2 \right\rbrace} = 1.
\end{equation}
\end{theorem}
\subsection{Max Distance Block Coordinate Descent} \label{subsection-max-distance-block-cd}

Let $A \in \mathbb{R}^{n \times d}$, $b \in \mathbb{R}^{n}$ and $r^* = - \Prj_{\ker(A^\intercal)} b$. Let $\lbrace E_j : j=1,\ldots,\epsilon \rbrace$ be such that $\begin{bmatrix}
E_1 & \cdots & E_{\epsilon}
\end{bmatrix}$ is a column permutation of the $R^{d \times d }$ identity matrix. In this method, at iteration $k$, we select $E_{i_k}$ according to 
\begin{equation}
i_k \in \argmax_{j \in \lbrace 1,\ldots,\epsilon \rbrace } \norm{ (E_j^\intercal A^\intercal A E_j)^\dagger E_j^\intercal A^\intercal (Ax_k - b) }_2,
\end{equation}
and we compute $x_{k+1} = x_k + E_{i_k} \alpha_k$ where $\alpha_k \in \argmin_\alpha \innorm{ Ax_k - b - A E_{i_k} \alpha }_2$. 

\begin{lemma}
Max Distance Block Coordinate Descent is Markovian.
\end{lemma}
\begin{proof}
The selection of $\lbrace i_k \rbrace$ depends only on $x_k$, which implies the conclusion.
\end{proof}

\begin{lemma}
Max Distance Block Coordinate Descent is $1,1$-Exploratory.
\end{lemma}
\begin{proof}
If $Ax_0 - b \neq r^*$, then $\exists j \in \lbrace 1,\ldots,\epsilon \rbrace$ such that $E_j^\intercal A^\intercal (Ax_0 - b)\neq 0$. Hence, $(E_j^\intercal A^\intercal A E_j)^\dagger E_j^\intercal A^\intercal (Ax_0 - b)\neq 0$. The conclusion follows.
\end{proof}

By \cref{theorem-convergence-col-finite}, we conclude as follows.

\begin{theorem} \label{theorem-max-distance-block-cd}
Let $A \in \mathbb{R}^{n \times d}$ and $b \in \mathbb{R}^n$, and define $r^* = - \Prj_{\ker(A^\intercal)} b$. Let $x_0 \in \mathbb{R}^d$ and $\lbrace x_k : k \in \mathbb{N} \rbrace$ be a sequence generated by Max Distance Block Coordinate Descent. Then, either there exists a stopping time $\tau$ with finite expectation such that $Ax_{\tau} - b = r^*$; or
there exists a sequence of non-negative stopping times $\lbrace \tau_{j} : j +1 \in \mathbb{N} \rbrace$ for which $\E{ \tau_{j} } \leq j \rnk{A}$,  and there exist $\gamma \in (0,1)$ and a sequence of random variables $\lbrace \gamma_j : j+1 \in \mathbb{N}\rbrace \subset (0,\gamma]$, such that
\begin{equation}
\Prb{ \bigcap_{j=0}^\infty \left\lbrace \norm{ Ax_{\tau_j} - b  - r^*}_2^2 \leq \left( \prod_{\ell=0}^{j-1} \gamma_{\ell} \right) \norm{ Ax_0 - b - r^* }_2^2 \right\rbrace} = 1.
\end{equation}
\end{theorem}

The methods described in in this section, \cref{subsection-gaussian-block-cs,subsection-max-distance-block-cd} are compared on a simple statistical regression problem in \cref{figure-block-cd-comparison}.
\begin{figure}
\centering
\begin{tikzpicture}
\begin{axis}[
	width=0.9\textwidth,
	height=0.4\textwidth,
    title={A Comparison of Some Block Column-action Methods},
    xlabel={Iteration},
    ylabel={Normal Residual Norm},
    ymode=log,
    xmin=0, xmax=800,
    ymin=1e-15, ymax=1000,
    xtick={0,200,400,600,800},
    ytick={1e-12,1e-9,1e-6,1e-3,1e0, 1e3},
    ymajorgrids=true,
    grid style=dashed,
]

\addplot[
    color=blue,
    mark=star,
    ]
    coordinates {
(0, 212.2653604638135)
(20, 29.89522811042826)
(40, 19.025380664854847)
(60, 9.06596445401196)
(80, 5.602110115963464)
(100, 2.992200044237467)
(120, 1.5598488961196622)
(140, 0.9926383674901108)
(160, 0.5772428679024261)
(180, 0.27912148755209526)
(200, 0.18146532117137615)
(220, 0.09618880740470227)
(240, 0.04647189864962331)
(260, 0.027838490819820493)
(280, 0.014194097300404962)
(300, 0.00838467153697422)
(320, 0.004570706374586055)
(340, 0.0024770484252132786)
(360, 0.0014283268779302068)
(380, 0.0008627448116871664)
(400, 0.0005390354053926182)
(420, 0.00030970390563368034)
(440, 0.00018384360451614149)
(460, 0.00010973031630284661)
(480, 6.230980728716747e-5)
(500, 3.247605803324459e-5)
(520, 1.6137621804760552e-5)
(540, 1.1035279386197281e-5)
(560, 5.024937682439701e-6)
(580, 3.0804471619487905e-6)
(600, 1.8293176502663125e-6)
(620, 8.62654040313314e-7)
(640, 4.909480100596813e-7)
(660, 2.907251899888864e-7)
(680, 1.3601313783311835e-7)
(700, 8.998494859580601e-8)
(720, 4.8411235316828447e-8)
(740, 2.3970746755105625e-8)
(760, 1.8234985832199276e-8)
(780, 9.768780755787199e-9)
}; \addlegendentry{Gaussian}

\addplot[
    color=red,
    mark=triangle,
    ]
    coordinates {
(0, 212.2653604638135)
(20, 2.318121496854798)
(40, 0.1969892898655076)
(60, 0.005581542688101672)
(80, 0.00013499689566826275)
(100, 6.674529434885207e-6)
(120, 1.8706158743806844e-7)
(140, 9.457011802577963e-9)
};\addlegendentry{Max Residual}

\addplot[
    color=black,
    mark=square,
    ]
    coordinates {
(0, 212.2653604638135)
(20, 3.1737014508960493)
(40, 0.053786204352292574)
(60, 0.001960512795957955)
(80, 5.2910584189242615e-5)
(100, 1.4600869544408539e-6)
(120, 4.685507421030384e-8)
(140, 6.823719221173745e-9)
    };\addlegendentry{Max Distance}

\end{axis}
\end{tikzpicture}
\caption{A comparison of three block column-action methods on a linear regression problem where the design matrix is derived from a balanced design of 50 treatments with twenty replicates each with $49$ columns of a random Gaussian matrix appended. All methods are stopped when the residual-norm of the normal equation is less than $10^{-8}$.} \label{figure-block-cd-comparison}
\end{figure}
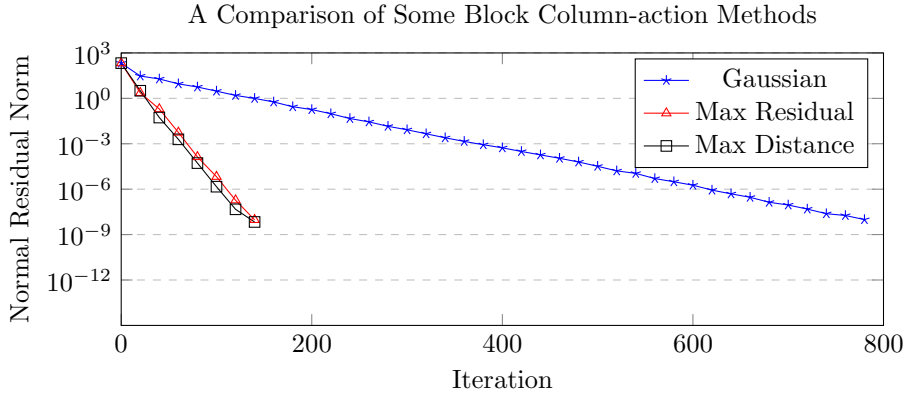


\bibliographystyle{siamplain}
\bibliography{bibliography}
\end{document}